\documentclass[11pt,psamsfonts]{article}

\usepackage[english]{babel}
\usepackage{amsmath,amssymb,graphicx,enumerate,bbm,accents,caption} 
\usepackage{amsthm}
\usepackage{stmaryrd}
\usepackage{float}
\usepackage[toc]{appendix}
\usepackage{array}
\usepackage{dsfont}
\usepackage{hyperref}

\newcommand{\nobracket}{}

\newcommand{\tmem}[1]{{\em #1\/}}
\newcommand{\tmmathbf}[1]{\ensuremath{\boldsymbol{#1}}}
\newcommand{\tmop}[1]{\ensuremath{\operatorname{#1}}}

\newenvironment{enumeratealpha}{\begin{enumerate}[a{\textup{)}}] }{\end{enumerate}}
\newenvironment{enumeratenumeric}{\begin{enumerate}[1.] }{\end{enumerate}}

\newtheorem{Theorem}{Theorem}[section]
\newtheorem*{SmoothingTheorem}{Smoothing Theorem}
\newtheorem{Proposition}[Theorem]{Proposition}
\newtheorem{Lemma}[Theorem]{Lemma}
\newtheorem{Corollary}[Theorem]{Corollary}

\theoremstyle{definition}
\newtheorem{Definition}[Theorem]{Definition}
\newtheorem{Example}[Theorem]{Example}
\newtheorem*{Hypothesis}{Hypothesis $(H_k)$}

\theoremstyle{remark}
\newtheorem{Remark}[Theorem]{Remark}

\newcommand{\cR}{\ensuremath{\mathbbm{R}}}

\newcommand{\Fol}{\mathfrak{F}}

\usepackage{amsthm}
\usepackage{amssymb}
\usepackage{amscd}
\usepackage{amsfonts}
\usepackage{amsbsy}
\usepackage{epsfig,afterpage}
\usepackage{overpic}
\usepackage{latexsym}
\usepackage{pstricks}
\usepackage{graphicx}
\usepackage{mathtools}
\usepackage[all]{xy}

\newcommand{\eps}{\varepsilon }

\newcommand{\Mol}{\mathcal{M}}

\newcommand{\vf}{f}
\newcommand{\comp}{\operatorname{comp}}
\newcommand{\rg}{\mathrm{reg}}

\newcommand{\R}{\ensuremath{\cR}}
\newcommand{\C}{\ensuremath{\mathbb{C}}}

\newcommand{\cS}{\mathbbm{S}}

\newcommand{\e}{\eps }

\newcommand{\vF}{\mathcal{F}}

\newcommand{\reg}{\mathrm{reg}}

\newcommand{\mon}{\mathbf{m}}

\newcommand{\us}{\underline{s}}

\title{Piecewise Smooth Dynamical Systems Regularized by Convolution}

\author{%
  Claudio A.~Buzzi\thanks{\scriptsize Departamento de Matem\'atica -- IBILCE--UNESP, 
  Rua C. Colombo, 2265, CEP 15054--000 S. J. Rio Preto, S\~ao Paulo, Brazil. 
  E-mail: \texttt{claudio.buzzi@unesp.br}}%
  \and
  Daniel Panazzolo\thanks{\scriptsize D\' epartement de Math\' ematiques -- IRIMAS--UHA,
  18 Rue des Fr\`eres Lumi\`ere, 68093 Mulhouse, France. 
  E-mail: \texttt{daniel.panazzolo@uha.fr}}%
  \and
  Paulo R.~da Silva\thanks{\scriptsize Departamento de Matem\'atica -- IBILCE--UNESP, 
  Rua C. Colombo, 2265, CEP 15054--000 S. J. Rio Preto, S\~ao Paulo, Brazil. 
  E-mail: \texttt{paulo.r.silva@unesp.br}}%
}

\date{\scriptsize \textbf{MSC 2020:} 34A36, 34D15, 34C45, 34C07, 34C23, 34C25}

\begin{document}

\maketitle
\begin{abstract}
We present a general regularization procedure for piecewise smooth vector fields whose discontinuity locus is  a variety of normal crossings type. We show that such regularization can be smoothed through a finite sequence of blowings-up, thereby reducing the problem to the study of the dynamics of a smooth vector field in a manifold with corners.  The procedure will be illustrated in the cases of piecewise smooth vector fields on $\cR^2$ with discontinuity locus $x=0$ or $xy=0$, and on $\cR^3$ with discontinuity locus $xyz=0$.  We will see that some unexpected dynamical phenomena such as Bogdanov-Takens bifurcations and Bykov cycles may arise even in the case of piecewise constant vector fields.
\end{abstract}

%\tableofcontents

\section{Introduction}\label{s0}
Consider a pair $(M,\Sigma)$ formed by a smooth manifold $M$ (the phase space) and closed subset $\Sigma \subset M$ (the discontinuity locus). A {\em piecewise smooth vector field} on $(M,\Sigma)$ is given by smooth vector field $X$ defined on $M\setminus \Sigma$ and satisfying the following {\em extension property}:  the restriction of $X$ to each connected component of $M\setminus \Sigma$ extends smoothly to the whole manifold. 

A basic example is given by $M = \R\times\R^{n-1}$ and $\Sigma=\{0\}\times \cR^{n-1}$. In Cartesian coordinates $(x,y) = (x,y_1,..,y_{n-1})$, a piecewise smooth vector field on $(M,\Sigma)$ takes the form 
\begin{equation}
X=f \frac{\partial}{\partial x} + \sum_{k=1}^{n-1} g_k \frac{\partial}{\partial y_k}
\label{sisdesc}
\end{equation} 
where the components $f,g_k$ can be written as 
$$f = \mathds{1}_{\{x>0\}}f_+ + \mathds{1}_{\{x<0\}}f_-\quad\text{and}\quad
g_k = \mathds{1}_{\{x>0\}} g_{k,+} + \mathds{1}_{\{x<0\}} g_{k,-} ,
$$
Here, for each sign $\pm$,  $f_{\pm}, g_{k,\pm}$ are smooth functions on $M$, and $\mathds{1}_S$ denotes the characteristic function of a set $S$. Alternatively, we can write 
$$
X = \mathds{1}_{\{x>0\}} X_+ + \mathds{1}_{\{x<0\}} X_-
$$
where $X_\pm = f_\pm \frac{\partial}{\partial x} + \sum g_{k,\pm}\frac{\partial}{\partial y_k}$ are globally smooth vector fields.

A fundamental question is how to associate a {\em dynamical system} to a piecewise smooth vector field. More precisely, one is interested in extending the local flow of $X$, which is well-defined in $M\setminus \Sigma$, to a flow (or a semi-flow) in the vicinity of the discontinuity locus. 

There is a vast literature dealing with the subject and some of the foundational 
ideas were introduced by Filippov in \cite{AF}. In that work, he described what are now known as the {\em Filippov conventions}, providing a natural extension of the flow to the discontinuity locus under suitable genericity assumptions.  

Later, in \cite{ST}, the authors adopt a different approach by studying families of smooth vector fields obtained from $X$ through a procedure known as {\em Sotomayor-Teixeira regularization}, or simply  {\em ST-regularization}.
For instance, keeping the notation of the above example, a ST-regularization of $X$ is
a one-parameter family $X_{\eps}$, with $\eps \in (\cR_{>0},0)$ (\footnote{The germ notation $(U,p)$ denotes an arbitrarily small neighborhood of a point $p$ in a topological space $U$.}), given by
\begin{equation}
X_\eps = f_\eps\frac{\partial}{\partial x} + \sum_{k=1}^{n-1} g_{k,\eps} \frac{\partial}{\partial y_k}
\label {regst}
\end{equation}
where, for $h = f$ or $h = g_k$, we define 
$$h_\eps(x,y)=\frac{1}{2}\left(1+\varphi\left(\frac{x}{\eps}\right) \right) h_+ + \frac{1}{2}\left(1-\varphi\left(\frac{x}{\eps}\right) \right)h_-$$
Here, $\phi$ is a so-called \emph{smoothed sign function}, that is, a smooth
and increasing function on $\mathbb{R}$ such that there exists $T>0$ with
$$
\phi(t) = -1 \quad \text{for } t \le -T,
\qquad
\phi(t) = 1 \quad \text{for } t \ge T.
$$
Notice that $X_{\eps}$ is a smooth vector field satisfying $X_{\eps} =  X_+$ on $\{x\geq T\eps\}$ and $X_{\eps}=  X_-$ on $\{x\leq-T\eps\}$. 

As later proved in \cite{LST5}, under genericity assumptions, the flow idealized by Filippov can be interpreted as suitable limit of the flow of the ST-regularization, therefore establishing a connection between the two approaches.

In \cite{BPT}, Buzzi, Silva and Teixeira - strongly inspired by an idea of Dumortier -   
proved that, under a suitable rescaling of coordinates, the ST-regularization $X_\eps$ gives rise to a singular perturbation problem as $\eps \rightarrow 0$.
This result provided a very fruitful connection between geometric singular perturbation theory (GSP) and piecewise smooth dynamics. 

In a series of papers, \cite{LST, LST3, LST4, LST5}, Llibre, Silva and Teixeira studied several regularization problems in $\R^n$, including the double 
regularization in the case where the discontinuity set is given by the union of two transversal hyperplanes and 
the regularization in more degenerate surfaces.  

In \cite {BRT, BLT}  the authors used classical singular perturbation techniques to study the regularization of the fold-fold singularity.  In \cite{Kr2}, the same problem is investigated using the blow-up techniques developed by Dumortier and Roussarie in \cite {DR}.

Based on the above example, the definition of the Sotomayor-Teixeira regularization can be extended to the setting where $M$ is an arbitrary manifold and the discontinuity locus $\Sigma$ is a {\em smooth} codimension one-submanifold (see for instance \cite{PS}, section 3.1). In \cite{PS}, the second and third authors developed a general geometric framework for studying piecewise smooth vector fields in cases where the discontinuity locus is not assumed to be smooth.

More precisely, assuming that $\Sigma$ is a (coherent) real analytic subvariety of $M$, the authors employed the classical results from resolution of singularities to prove the following statement: There exists a finite sequence of blowing ups,
$$
(M,\Sigma) = (M_0,\Sigma_0) \leftarrow (M_1,\Sigma_1) \leftarrow \cdots \leftarrow (M_r,\Sigma_r) = (\widetilde{M},\widetilde{\Sigma})
$$
such that $\widetilde{\Sigma}$, the total transform of $\Sigma$ under the blowing-up sequence, is a codimension one normal crossings variety of $\widetilde{M}$ (i.e.~locally defined by a union of coordinate hyperplanes). Moreover, each piecewise smooth vector field $X$ on $(M,\Sigma)$ pulls back to a {\em piecewise smooth oriented 1-dimensional foliation} $\cal F$ on $(\widetilde{M},\widetilde{\Sigma})$ (see subsection \ref{subsect-piecewisesmoothfoliations} for a precise definition). 

Additionally, building on the ideas of \cite{BPT},  it is shown in Theorem 2.1 of \cite{PS} that, under the assumption that $\Sigma$ is a smooth submanifold, the ST-regularization is {\em blow-up smoothable}. 

Let us briefly explain this result. Firstly, we observe that the regularized family $X_\eps$ can be seen as a vector field in the product space $N = M \times (\cR_{\ge 0},0)$, with discontinuity locus $\Sigma \times \{ 0\}$. Then, performing a single blowing-up on $N$,
$$
N = N_0
\xleftarrow{\;\varphi\;}
N_1 = \widetilde{N}
$$
with center on $\Sigma \times \{ 0\}$, it is proved that the pull-back of  $X_\eps$ extends to a {\em globally smooth} oriented 1-foliation on the blown-up space $\widetilde{N}$  (i.e.~a foliation locally generated by smooth vector fields).

As mentioned above, the ST-regularization was originally defined under the assumption that the discontinuity locus $\Sigma$ is smooth.  One can generalize this definition to the case where $\Sigma$ is a normal crossings subvariety through the use of a {\em multi-regularization} (see e.g.~\cite{PS}, section 3.1 for an example of a double regularization of the cross). Intuitively, the multi-regularization regularizes separately each irreducible component of $\Sigma$, at the cost of introducing several new parameters.
 
The main goal of the present paper is to systematically study the regularization of piecewise smooth vector fields through a different regularization method based on the {\em convolution} integral. 

The regularization by convolution is a very classical tool from analysis, which can be defined in a much broader context.  We now briefly describe the construction and refer to subsection \ref{subsect-piecewisesmoothfcs} for more details:  Let $X$ be a locally integrable vector field in $\cR^n$. In other words, we assume that $X$ has the form 
$$
X = \sum_{i=1}^n f_i \frac{\partial}{\partial x_i}
$$
where each component $f_i$ belongs to $L^1_{\mathrm{loc}}(\cR^n)$. 
Given a smooth function $m\ge 0$ with compact support such that $\displaystyle \int m = 1$ (called a {\em mollifier}) we define 
\begin{equation}
\label{regularization-expression}
X_\eps = m_\eps \ast X : =  \sum_{i=1}^n (m_\eps\ast f_i)  \frac{\partial}{\partial x_i} 
\end{equation}
where, for each $\eps > 0$,
$$
m_\eps\ast f(x) = \int_{\cR^n} f(x-y)\, m_\eps(y) dy
$$
is the convolution of $f$ with the rescaled mollifier $m_\eps(x) = \frac{1}{\eps^n}m(\frac{x}{\eps})$. 

It is easy to prove (see e.g.~\cite{HormanderI}, section 1.3) that $X_\eps$ is a smooth vector field for each $\eps > 0$, and that $X_\eps$ converges uniformly to $X$ on each 
compact set $K$ where $X$ is continuous. 

Note that $X_\eps$ defines a vector field on the product space $\cR^n \times \cR_{>0}$ tangent to the fibers $\{\eps = \mathrm{cte}\}$. We {\em complete} this family to the fiber $\eps = 0$ by defining $X_0 = X$. The resulting vector field in $\cR^n \times \cR_{\ge 0}$, noted $X^\reg$, is called {\em regularization by convolution of $X$ (with mollifier $m$)}.  We remark that such completion results into a vector field which still has a discontinuous behavior when restricted to the fiber $\{\eps = 0\}$.  

The above definition makes no assumption on the geometry of the set $\Sigma$ where $X$ fails to be smooth (referred to in this context as the singular support of $X$). At this level of generality, it becomes difficult to describe the qualitative behavior of $X_\eps$ in the limit as $\eps \to 0$. For instance, the associated Cauchy initial value problem may admit multiple distinct limit solutions (see, e.g., \cite{HormanderNHDE}, Section 1.4).

Essentially, the goal of this paper is to prove that such description is possible in the case where the following conditions hold:
\begin{itemize}
\item $M$ is an open subset of $\mathbb{R}^n$
\item $\Sigma$ is a normal crossings subvariety of $M$, and
\item $X$ is a piecewise smooth vector field on $(M,\Sigma)$.
\end{itemize}
Let us now state our main result and refer the reader to subsection \ref{subsect-smoothing-vf} for the precise definitions of the concepts involved. Consider a 4-uple $(N,M,\Sigma,\pi)$, where $(M,\Sigma)$ are given as above, $N = M \times (\cR_{\ge 0},0)$ is the product space where the regularized vector field 
$X^\reg$ is defined and $\pi:N \rightarrow (\cR_{\ge 0})$ is the linear projection into the $\eps$-coordinate, so that the regularized vector field is tangent to the fibers $\pi=\text{const}$.  In the statement, we denote by $\mathfrak{X}^{\infty}(M, \Sigma)$ the space of all piecewise smooth vector fields on $M$ with discontinuity locus $\Sigma$ and denote by $\reg_m:X \mapsto X^\reg$ the linear operator defined by (\ref{regularization-expression}). Finally, we denote by $\mathfrak{F}^{\infty}(N, M, \Sigma, \pi)$ the space of foliations locally defined by those reguralized vector fields. 
\begin{SmoothingTheorem}
There exists a finite sequence of blowings-up
\[
(N, M, \Sigma, \pi) = (N_0, M_0, \Sigma_0, \pi_0)
\xleftarrow{\;\varphi_1\;}
\cdots
\xleftarrow{\;\varphi_r\;}
(N_r, M_r, \Sigma_r, \pi_r)
\]
such that the following properties hold:
\begin{itemize}
  \item Each blow-up center \(C_k\) is contained in the corresponding discontinuity locus \(\Sigma_k\).
  
  \item Given an arbitrary vector field \(X \in \mathfrak{X}^{\infty}(M, \Sigma)\) and a mollifier \(m\), let \(\vF = X^{\mathrm{reg}} = \mathrm{reg}_m(X)\) denote the $1$-foliation in \(\mathfrak{F}^{\infty}(N, M, \Sigma, \pi)\) defined by the regularization of \(X\). Then there exists a sequence of $1$-foliations
  \begin{equation}
  \vF_k \in \mathfrak{F}^{\infty}(N_k, M_k, \Sigma_k, \pi_k), \quad k = 0, \ldots, r,
  \end{equation}
  such that \(\vF_0 = \vF\) and \(\vF_{k}\) is a blowing-up of \(\vF_{k-1}\) under \(\varphi_k\).
  
  \item The final discontinuity locus \(\Sigma_r\) is empty.
\end{itemize}
\end{SmoothingTheorem}
The final statement says that the final foliation obtained by the sequence of blowing-up is {\em globally smooth}. For this reason, we refer to this sequence of blowing-ups as a {\em smoothing procedure}. Note that the smoothing procedure depends only on $(M,\Sigma)$, and therefore is {\em universal} with respect to the choice of vector field 
$X \in \mathfrak{X}^{\infty}(M, \Sigma)$ and mollifier $m$.

We further observe that, in the particular case where $\Sigma$ is a smooth submanifold, there is an explicit relation between the ST-regularization and the regularization by convolution in the vicinity of $\Sigma$ (see subsection \ref{subsection-st-reg-link} for the details). Consequently, our result can be viewed as a generalization of the smoothing result proved in \cite{PS}. 

To summarize, the present paper is a natural continuation of our previous work~\cite{PS}. 
In both papers we establish desingularization results related to piecewise smooth vector fields, though in different contexts. 
In~\cite{PS}, we desingularized piecewise-smooth vector fields with a singular 
discontinuity set and obtained a situation in which the discontinuity locus has a 
normal crossing structure. In the present work, we consider piecewise-smooth systems 
whose discontinuity set already has normal crossings and regularize them by convolution. 
We prove that the resulting regularization \( X^\reg \) becomes a smooth vector 
field on a manifold (with corners) after a finite sequence of blow-ups.

\subsection{Related work}
Regularizations of systems whose discontinuity locus is the transverse intersection of two
hypersurfaces have mostly been built from products of smoothed sign functions, one for each
hypersurface, with a common parameter~$\eps$. Dieci and Guglielmi~\cite{dieci2013} regularize the system with
products of smoothed sign functions and relate the limit to this sliding vector field.
Kaklamanos and Kristiansen~\cite{kaklamanos2019} use two regularization functions and a single
small parameter. They show that at most one sliding equilibrium on the intersection is stable,
use blow-up to describe how sliding along the intersection appears and disappears, and prove the
existence of canard explosions in~$\R^3$. Guglielmi and Hairer~\cite{guglielmi2015} classify the
hidden dynamics of such regularizations, and And\`o, Edwards and Guglielmi~\cite{ando2024} find
Hopf bifurcations and limit cycles in the hidden dynamics.

For a product mollifier $m(t)=\prod_j\mathbf{m}(t_j)$, the regularization by convolution presented here
reproduces, to leading order on the exceptional divisors, a product regularization with smoothed
sign functions; see~\eqref{inner-field}. The approach of the present paper
differs in three respects. It allows arbitrary kernels, whose weights are the orthant masses
of~$m$ and are not multilinear in general (see also~\cite{novaes2015}). It applies, with the same construction, to every normal-crossings
discontinuity locus, whatever the number of its components. And it yields a {\em global smoothing}
result, by a sequence of blowings-up that depends only on $(M,\Sigma)$.

Finally, to our knowledge, Bogdanov--Takens points and Bykov cycles have not been identified in
regularizations of systems on the cross. We study them in
Section~\ref{s7} .

\noindent{\em Relation with the analysis on manifolds with corners.}
The blown-up space $N_r$ of the Smoothing Theorem is the iterated blow-up of
$N=M\times(\cR_{\ge 0},0)$ along the strata $\Sigma_I\times\{0\}$, taken in order of
increasing dimension; this is the standard construction of geometric microlocal analysis
(blow-up of submanifolds of a manifold with corners, see \cite{Melrose, KottkeMelrose2015}).
In Theorem \ref{theorem-smoothing-rpss}, we establish a version of the Smoothing theorem for functions of the form $\vf^\reg = m_\eps \ast f$ with $f$ piecewise smooth. In that framework one usually expects a function such as $\vf^\reg\circ\Phi_r$ to be
{\em polyhomogeneous}: at each boundary component it has an asymptotic expansion in powers of
a defining function of that component, possibly multiplied by powers of its logarithm.
The content of the Smoothing Theorem is sharper: $\vf^\reg\circ\Phi_r$ is {\em smooth}, with
no logarithmic terms, on a blown-up space which depends only on $(M,\Sigma)$. This is a
consequence of the compact support of the mollifiers used in the regularization. We expect
the same result to hold for rapidly decreasing kernels, such as the Gaussian kernel. In
particular, the solution of the heat equation with piecewise smooth initial
datum $\vf$ would lift to a smooth function on the space obtained from
$M\times(\cR_{\ge 0},0)$ by the same sequence of blowings-up. This would be the analogue, for initial data with jumps along a normal crossings
variety, of Melrose's theorem \cite[Chapter~7]{Melrose1993}.
\subsection{Overview of the paper}\label{s2}
Section~\ref{section-examples-Paulo} is devoted to a series of examples illustrating the application of convolution regularization to classical discontinuous vector fields in dimension n = 2.

In Section \ref{s4},  we develop the theoretical foundation of the convolution-based regularization procedure 
for discontinuous functions and vector fields defined on manifolds with corners. We begin by introducing the 
notions of \emph{piecewise-smooth spaces} and \emph{piecewise-smooth functions}, that is, smooth functions 
defined on each component of \( M \setminus \Sigma \) which admit smooth local extensions across 
the components of the discontinuity locus \(\Sigma\). We formalize the space \( C^{\infty}(M,\Sigma) \) of 
such functions and define the \emph{convolution regularization} of \( f \in C^{\infty}(M,\Sigma) \) 
by means of a mollifier \( m \), obtaining the linear operator
\[
\mathrm{reg}_m(f) = m_\varepsilon * f,
\]
which generates a one-parameter family of smooth functions depending on \(\eps > 0\) 
which we complete to \(\varepsilon = 0\) by defining $f_0 = f$. We show that this process can be interpreted as 
a smooth extension in the augmented space \( N = M \times \mathbb{R}_{\ge 0} \), a manifold with boundary whose boundary component \( M = \{\varepsilon = 0\} \) represents the initial discontinuous domain.

In Section \ref{s5}, we prove that every function regularized by convolution, \( f^{\mathrm{reg}} = \mathrm{reg}_m(f) \), 
can be made globally smooth after a finite sequence of blowings-up of the ambient space \((N, M, \Sigma)\). Intuitively, the blowings-up gradually remove the components of smaller dimension in the stratification of discontinuity locus of \( f^{\mathrm{reg}} \).

In Section \ref{s6}, we extend the smoothing theorem proved in Section \ref{s5} to 
the case of discontinuous vector fields and prove the Smoothing Theorem stated above. This result is the vector-field counterpart of \textit{Theorem} 
\ref{theorem-smoothing-rpss} for scalar functions. In the particular case where the discontinuity locus \(\Sigma\) is a smooth hypersurface, the convolution-based regularization essentially coincides with the classical \textit{Sotomayor--Teixeira regularization} on the exceptional divisor.  
This establishes a precise link between the convolution approach and the traditional smooth regularization methods in non-smooth dynamics.

In Section \ref{s7} we present three examples showing some applications of the 
smoothing theorem for regularized piecewise-smooth vector fields.  The regularization by convolution will be illustrated in the cases of piecewise smooth vector fields on $\cR^2$ with discontinuity locus $xy=0$, and on $\cR^3$ with discontinuity locus $xyz=0$. 
The purpose is to show how such regularization allows to reveal new dynamical phenomena related to discontinuous dynamics.

\section{Some examples in dimension $n=2$}\label{section-examples-Paulo} 
In this section, we present the expressions for the regularizations by convolution of some well-known normal forms of 
discontinuous vector fields in the plane with smooth discontinuity locus. We refer the reader to \cite{BPT} for the detailed derivations of these normal forms. 

We firstly recall that a piecewise smooth vector field on the plane with discontinuity locus $\Sigma =\{x=0\}$ can be written as
$$
X = \mathds{1}_{\{x>0\}} X_+ + \mathds{1}_{\{x<0\}} X_-
$$
where $X_{\pm} = f_\pm \frac{\partial}{\partial x} + g_\pm \frac{\partial}{\partial y}$ are smooth vector fields in the plane. In this setting, given a mollifier $m$, the convolution integral (\ref{regularization-expression}) can be written as
$$
m_\eps \ast X = f_\eps \frac{\partial}{\partial x} + g_\eps \frac{\partial}{\partial y}
$$
where, by a simple coordinate change, we can write the coefficient $h_\eps = f_\eps$ or $h_\eps = g_\eps$ as
\[
\begin{aligned}
h_\varepsilon(x,y)
&= \int_{\{x-\varepsilon u>0\}}
   h_+(x-\varepsilon u,\,y-\varepsilon v)\, m(u,v)\,du\,dv \\
&\quad + \int_{\{x-\varepsilon u<0\}}
   h_-(x-\varepsilon u,\,y-\varepsilon v)\, m(u,v)\,du\,dv .
\end{aligned}
\]
Note that the subscript in the first integral indicates that the integration is computed over the domain $D_{x,\eps} = \{(u,v)\in \cR^2 :  x - \varepsilon u > 0\}$ (and similarly for the second integral). 

In order to exhibit some nice explicit expressions, we will henceforth suppose in this section that the mollifier has the special form of a product $m(x,y) = \mathbf{m}(x)\cdot\mathbf{m}(y)$, where $\mathbf{m}$ is a mollifier in $\cR$ such that
\begin{itemize}
    \item The support of $\mathbf{m}$ is contained in the interval $[-1,1]$.
    \item $\mathbf{m}$ is constant on the interval  $\left[-1 + \eta,1 - \eta\right]$.
\end{itemize}
where $\eta$ is some positive small constant (see Figure \ref{fig-mollifm}). 

\begin{figure}[htb] 
\centering
  \includegraphics[width=7.69519054178145cm]{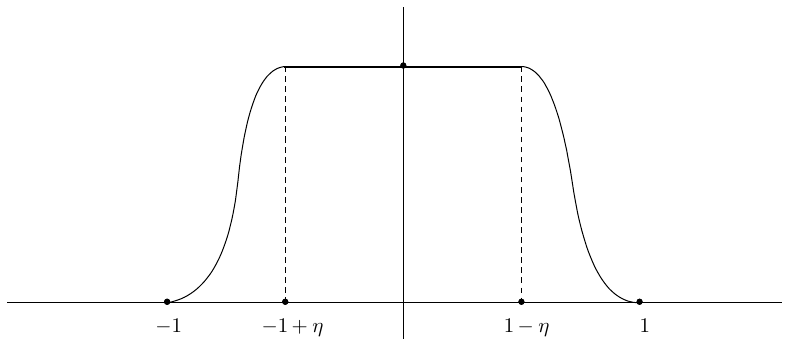}
  \caption{The mollifier $\mathbf{m}$.}\label{fig-mollifm}
\end{figure}

In other words, we suppose that our mollifier $m$ is a small perturbation (in the $L^1$-norm) of the characteristic function 
$\dfrac{1}{4}\mathds{1}_{[-1,1]^2}$. 
\begin{Remark}
We will see that, under such hypothesis, the regularization by convolution of an arbitrary {\em piecewise polynomial} vector field will result (in an appropriate blow-up chart) into a family of polynomial vector fields, {\em up to a correction term $O(\eta)$}, i.e.~a term which goes uniformly to 0 as $\eta \rightarrow 0$. 

We emphasize that this assumption on the mollifier is made solely for computational convenience. Indeed, we expect that the phase portraits presented below remain topologically equivalent for an arbitrary choice of mollifier which is sufficiently near the mollifier $m$ with respect to the $C^0$ norm. 
\end{Remark}
\subsection {Regularization of the Sewing}
We assume that 
\[
X_+ = \frac{\partial}{\partial x} + \frac{\partial}{\partial y},\quad X_- = 2\frac{\partial}{\partial x} + \frac{\partial}{\partial y}
\]
In this case, using the fact that $\displaystyle \int m = 1$, we can write
\begin{align}
f_\eps(x,y)
&= \int_{\{x-\eps u>0\}} m(u,v)\,du\,dv
   + \int_{\{x-\eps u<0\}} 2\,m(u,v)\,du\,dv \notag\\
&= 1 + \int_{\{x-\varepsilon u<0\}} m(u,v)\,du\,dv, \\
\end{align}
Similarly, we conclude that  $g_\eps = 1$. Therefore, the regularized family is given by
$$
m_\eps \ast X = \left( 1+\int_{\{x-\eps u<0\}} m(u,v)\, du\, dv \right) \frac{\partial}{\partial x} + \frac{\partial}{\partial y}
$$
Let us now apply an {\em $\eps$-directional blowing-up} (see subsection\ref{subsect-blowing-up-regular}), which, in the present setting, is simply the coordinate change 
$$x = \bar{\eps} \bar x, \quad \eps = \bar \eps.$$
The pull-back of the regularized family under such blowing-up assumes the form 
$$
\frac{1}{\eps} \left( 1 + \int_{\{u > x\}} m(u,v)\, du\, dv \right) \frac{\partial}{\partial x} + \frac{\partial}{\partial y}
$$
where we dropped the bars to simplify the notation. Upon multiplication by $\eps$ (which corresponds to a reparametrization of time) we obtain the family of vector fields
$$
\mathcal{X} = \left( 1 + \int_{\{u > x\}} m(u,v)\, du\, dv \right) \frac{\partial}{\partial x} + \eps \frac{\partial}{\partial y}
$$
which is easily seen to be globally smooth. This illustrates the {\em smoothing procedure} in this simple setting.

More explicitly, using the above assumptions on the mollifier, we can write
$$
\int_{\{(u,v)\in \cR^2: u > x\}} m(u,v)\, du\, dv = \frac{1}{2}(1-x), \quad \text{ for $|x| \le 1-\eta$}
$$
Therefore, restricted to the region $R = \{|x| \le 1-\eta\}$, the regularized family assumes the simple polynomial form
$$
\mathcal{X} =  \frac{1}{2}(3-x) \frac{\partial}{\partial x} + \eps \frac{\partial}{\partial y}
$$
\subsection{Regularization of other generic singularities}
Applying similar computations to other normal forms studied in \cite{BPT}, we obtain the smoothed families listed in the table below. 

We recall that the expressions of $\cal X$ in the rightmost column are written up to a $O(\eta)$-error term, where $\eta>0$ is the small parameter appearing in the definition of the mollifier $m$. 
\begin{table}[htb!]
\begin{tabular}{l|ccc}
Singularity & $X_+$ & $X_-$ & $\cal X$ mod $\mathrm{O}(\eta)$  \\
\hline
Escaping & $\frac{\partial}{\partial x} + \frac{\partial}{\partial y}$ & $-\frac{\partial}{\partial x} + \frac{\partial}{\partial y}$ 
& $x \frac{\partial}{\partial x}+ \eps\frac{\partial}{\partial y}$   \\
Sliding    &  $-\frac{\partial}{\partial x} - \frac{\partial}{\partial y}$  &    $\frac{\partial}{\partial x} - \frac{\partial}{\partial y}$    
&   $-x\frac{\partial}{\partial x}-\eps \frac{\partial}{\partial y}$  \\
Saddle&    $(x+1)\frac{\partial}{\partial x} - y\frac{\partial}{\partial y}$      &   $(x-1)\frac{\partial}{\partial x} - y\frac{\partial}{\partial y}$     
&   $\left( x + \e x \right)\frac{\partial}{\partial x }-\eps y\frac{\partial}{\partial y}$                                         \\
 Fold-regular    &   $y\frac{\partial}{\partial x} + \frac{\partial}{\partial y}$    &  $\frac{\partial}{\partial x} + \frac{\partial}{\partial y}$    
 &  $\frac{1}{2}\left(1-x + y(1+x)\right)\frac{\partial}{\partial x }+\eps \frac{\partial}{\partial y}$  \\
 Saddle-node    &   $-\frac{\partial}{\partial x} - y^2\frac{\partial}{\partial y}$    &  $\frac{\partial}{\partial x}$    
 &  $-x\frac{\partial}{\partial x }-\eps(1 +  x ) (\frac{\eps^2}{6}+\frac{y^2}{2})\frac{\partial}{\partial y}$  \\
 Elliptic fold    &   $-y\frac{\partial}{\partial x} + \frac{\partial}{\partial y}$    &  $ y\frac{\partial}{\partial x} + \frac{\partial}{\partial y}$    
 &  $-x y\frac{\partial}{\partial x }+\eps \frac{\partial}{\partial y}$  \\
 Hyperbolic fold    &   $y\frac{\partial}{\partial x} + \frac{\partial}{\partial y}$    &  $2y\frac{\partial}{\partial x} - \frac{\partial}{\partial y}$    
 &  $\frac{1}{2}(3y -   x  y)\frac{\partial}{\partial x }+\eps x\frac{\partial}{\partial y}$  \\
Parabolic fold    &   $ -y\frac{\partial}{\partial x} - \frac{\partial}{\partial y}$    &  $2y\frac{\partial}{\partial x} + \frac{\partial}{\partial y}$    
&  $\frac{1}{2} (y - 3 x  y)\frac{\partial}{\partial x }- \eps x \frac{\partial}{\partial y}$                                                         
\end{tabular}
\end{table}

The Figures~\ref{fig3}, \ref{fig4}, \ref{fig5}, and \ref{fig6} illustrate the phase portraits of the corresponding regularized families 
$m_\varepsilon \ast X$ for $\eps = 1/10$.

\begin{figure}[!htb]
	\centering 
	  \includegraphics[width=6cm, height=4cm]{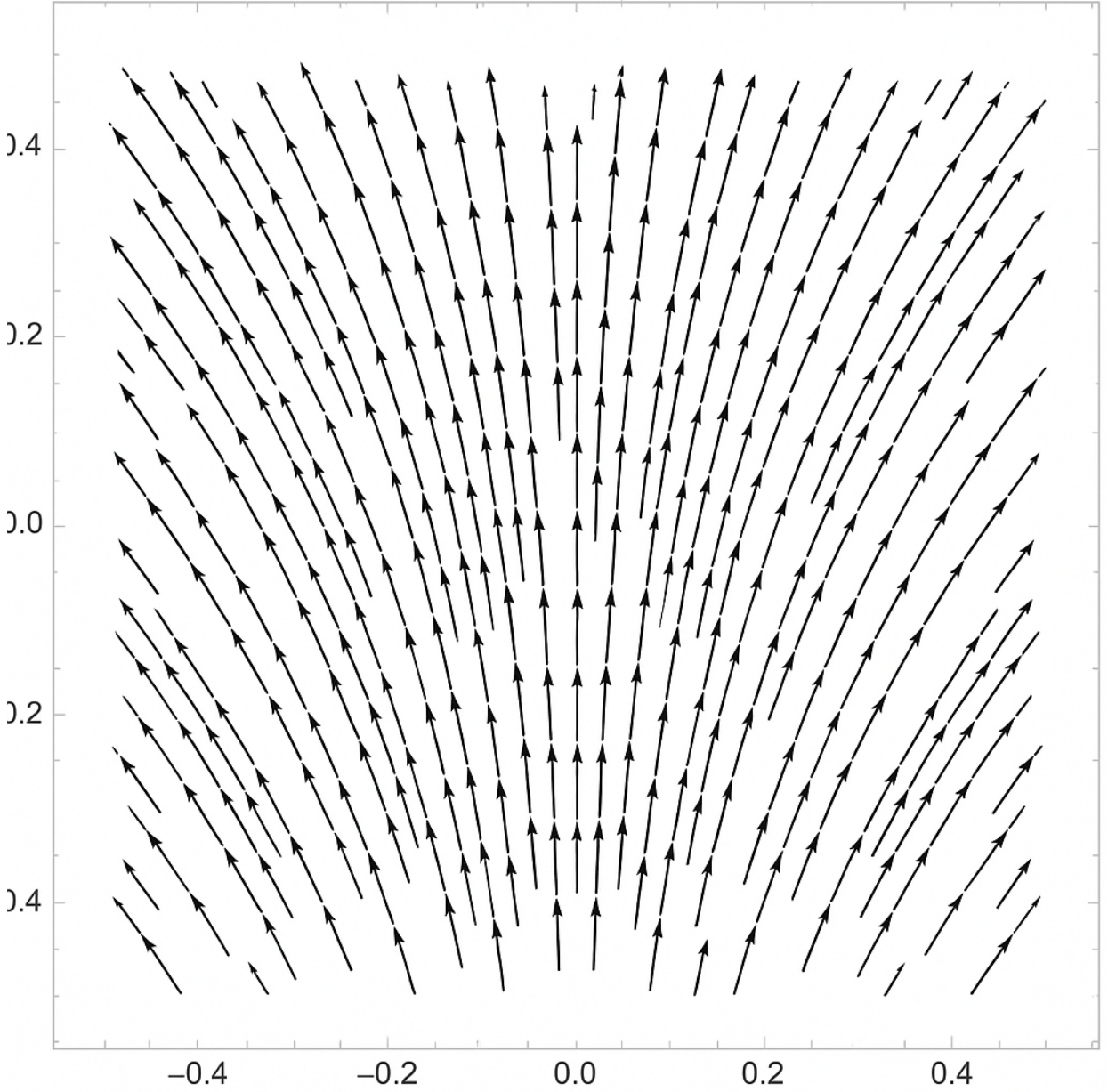} 
 \includegraphics[width=6cm, height=4cm]{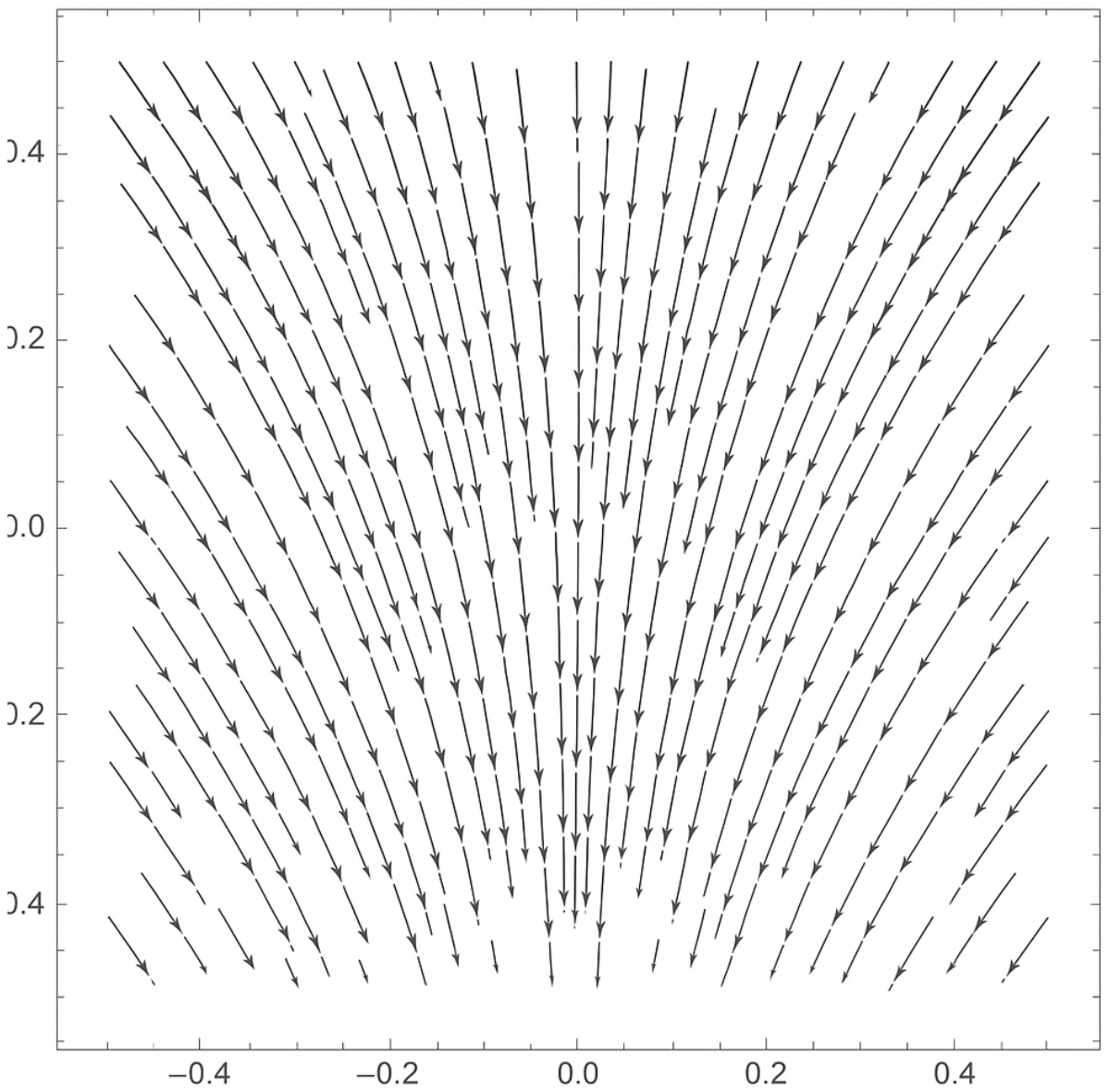} 
	\caption{Regularization around  escaping and  sliding points.}\label{fig3}
\end{figure}

\begin{figure}[!htb]	\centering 
	\includegraphics[width=6cm, height=4cm]{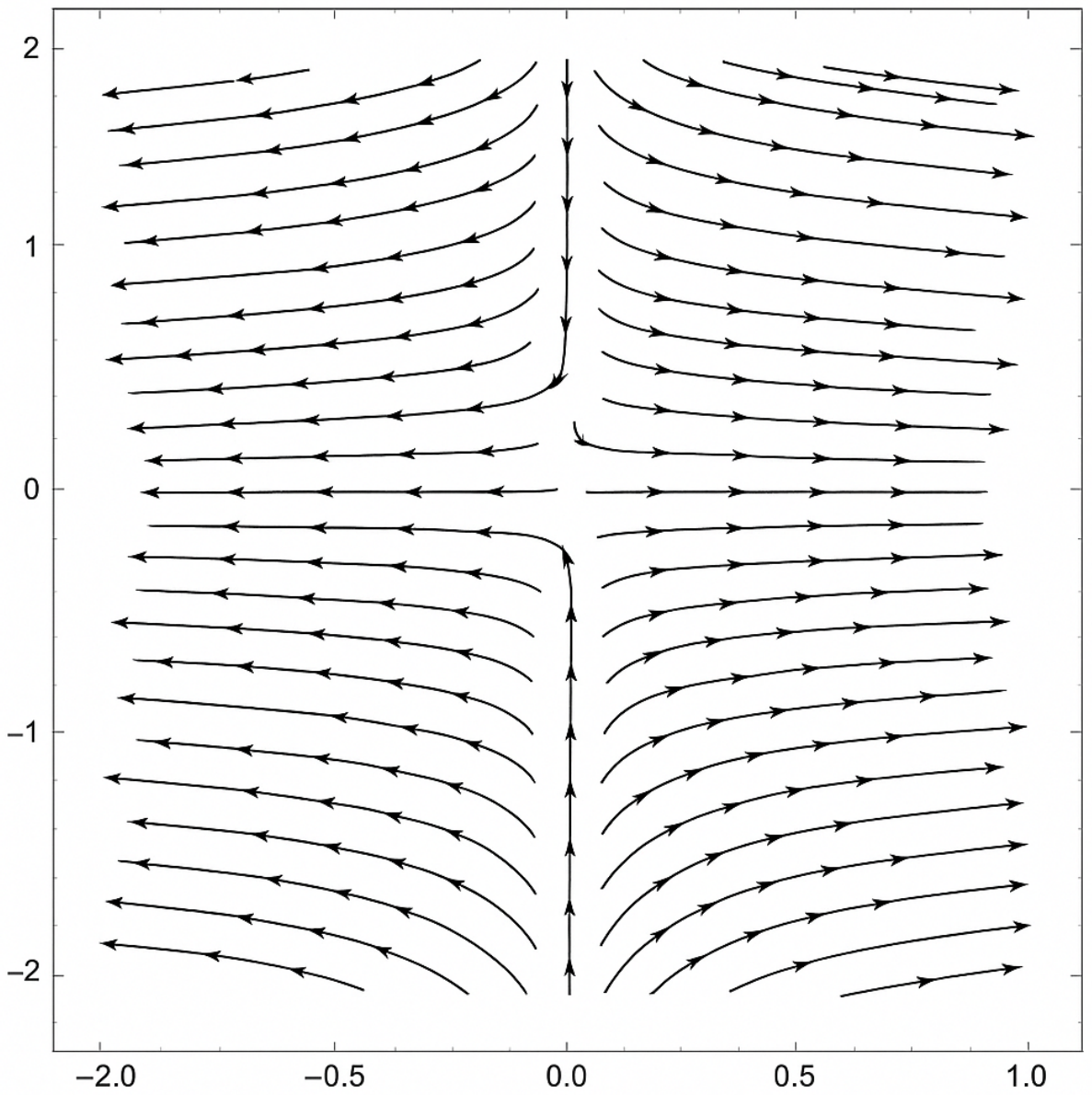} \includegraphics[width=6cm, height=4cm]{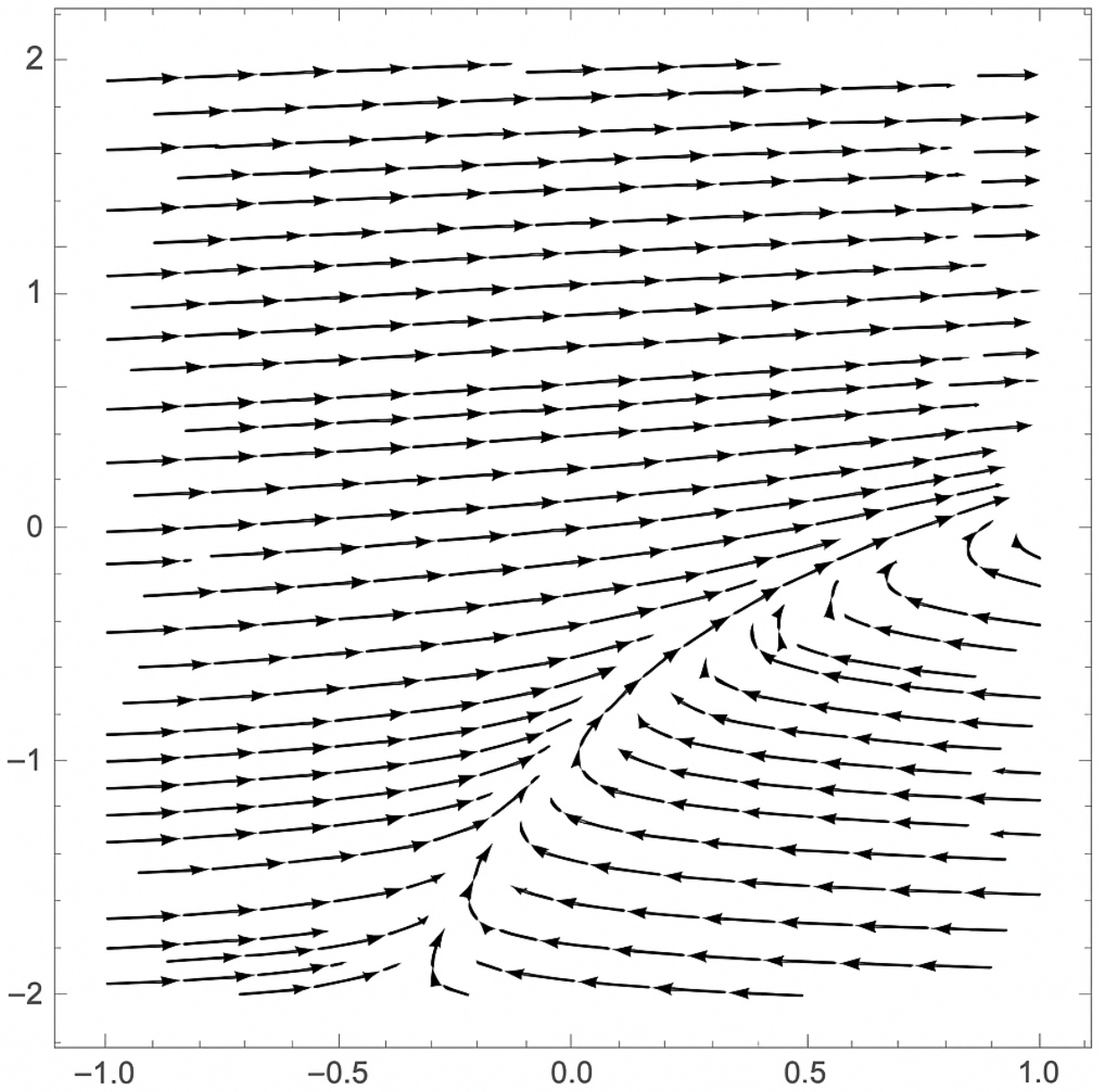} 
	\caption{Regularization around  saddle and   fold-regular points.}\label{fig4}
\end{figure}

\begin{figure}[!htb]
	\centering 
	\includegraphics[width=6cm, height=4cm]{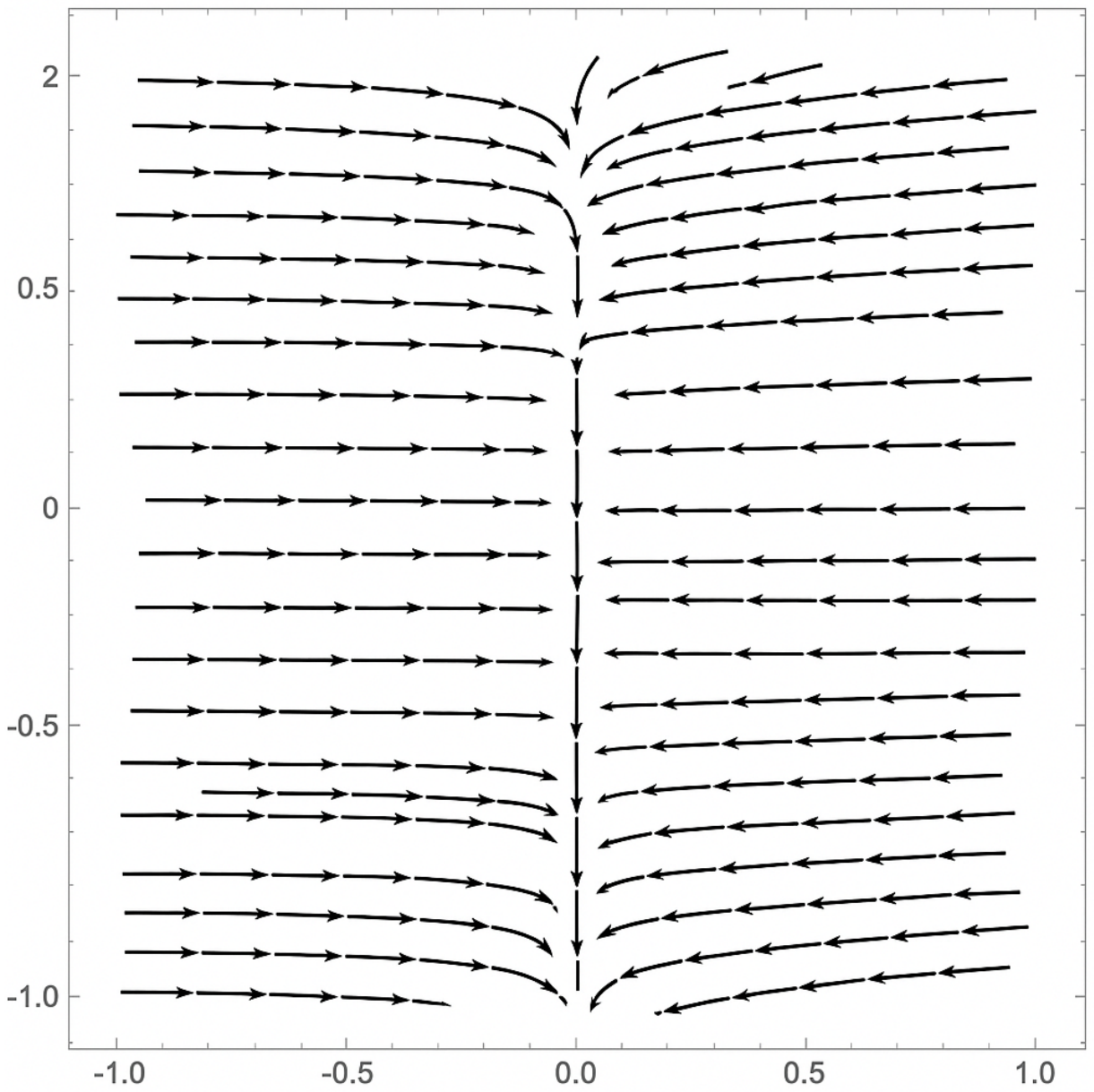} \includegraphics[width=6cm, height=4cm]{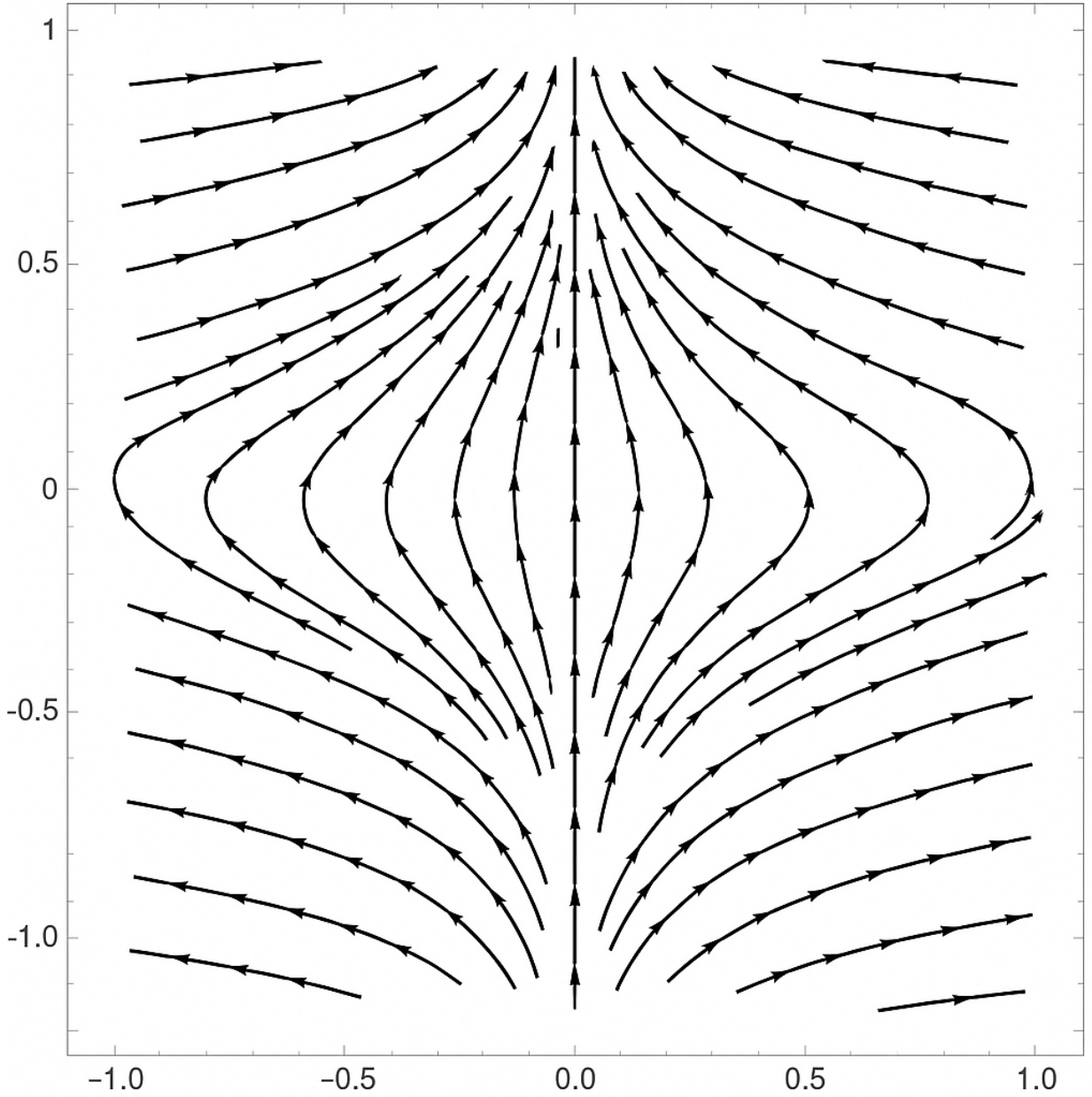} 
	\caption{Regularization around  saddle-node and elliptical fold  points.}\label{fig5}
\end{figure}

\begin{figure}[!htb]
	\centering 
	\includegraphics[width=6cm, height=4cm]{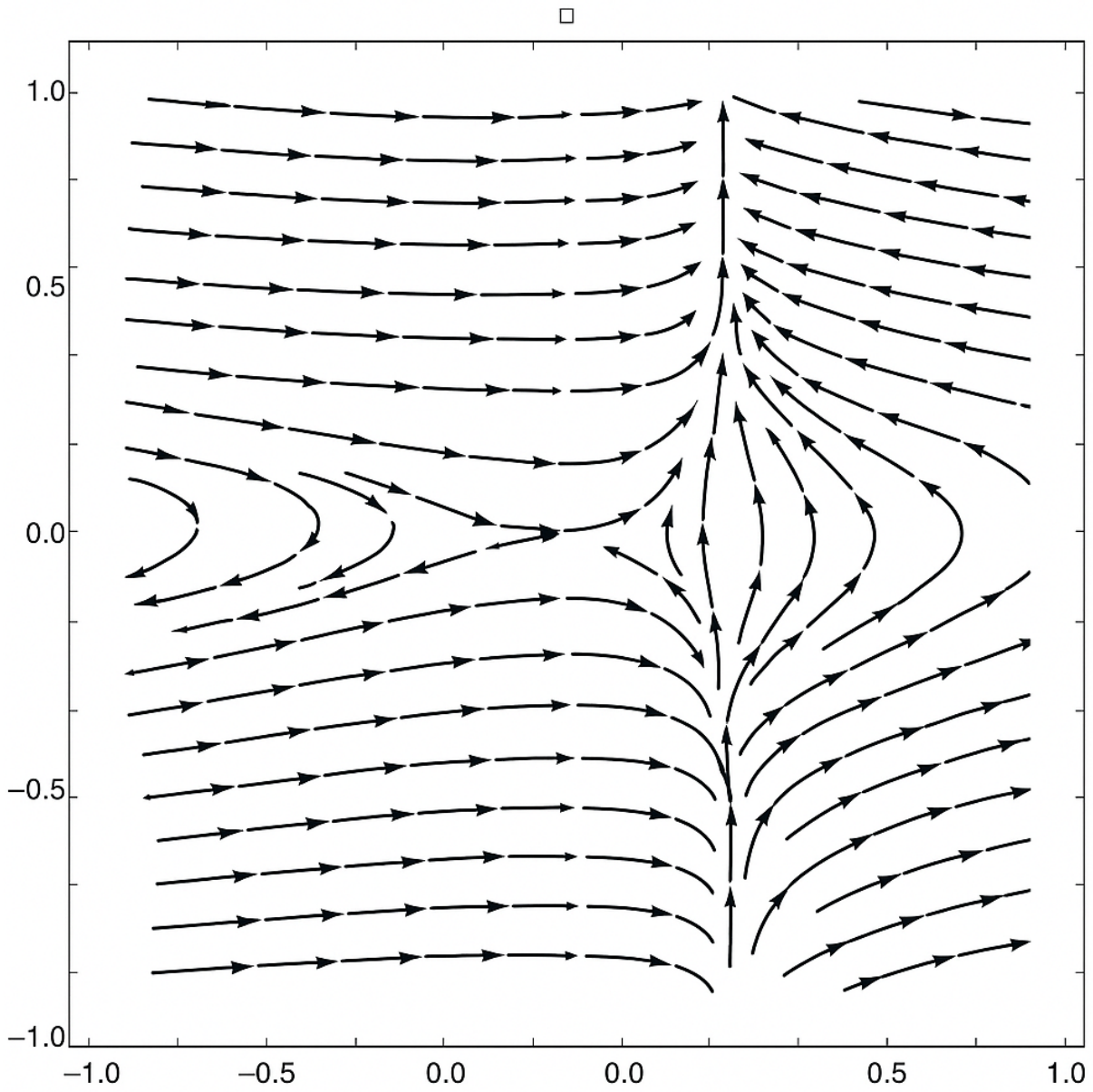} \includegraphics[width=6cm, height=4cm]{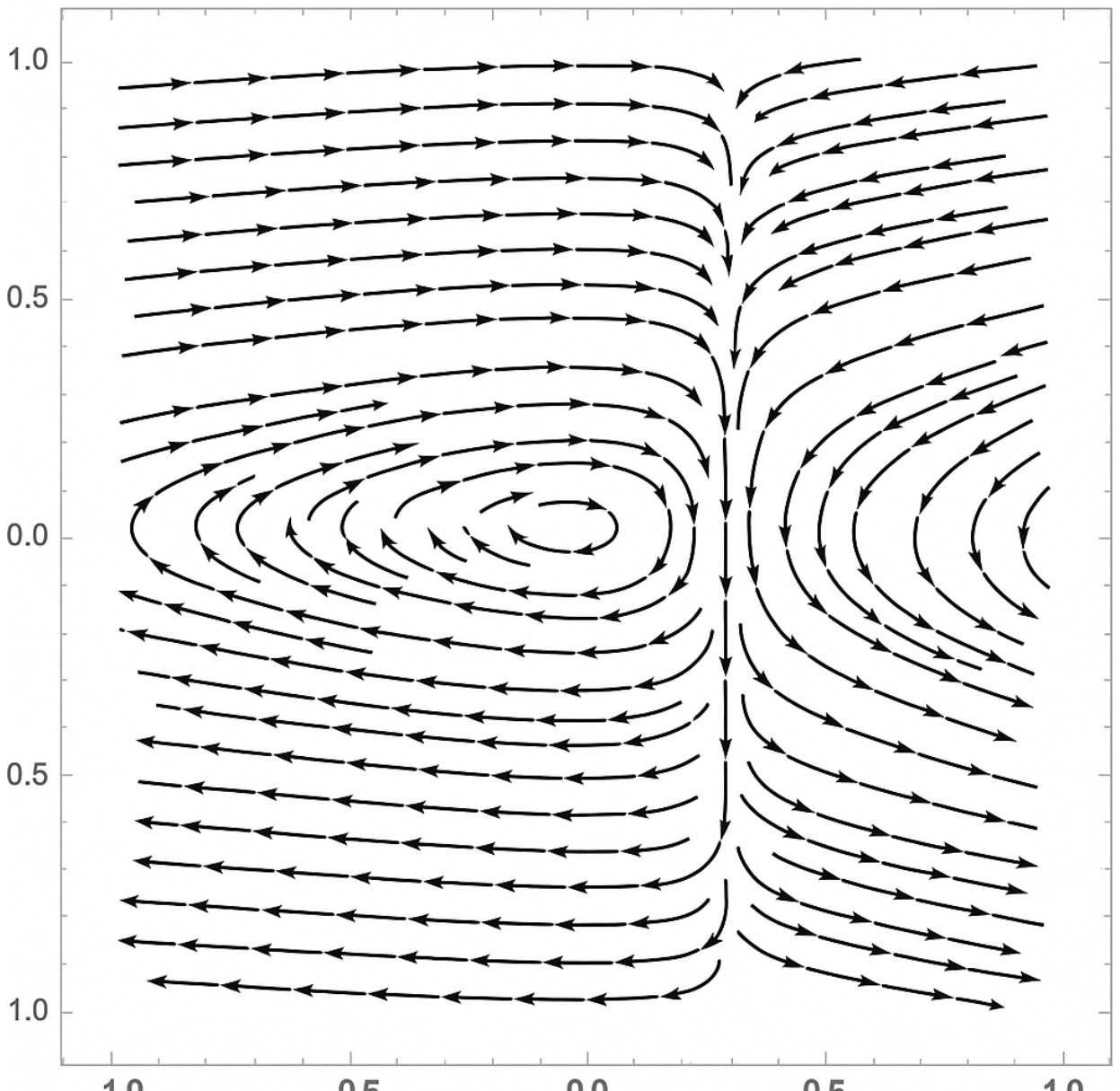} 
	\caption{Regularization around  hyperbolic   and parabolic folds points.}\label{fig6}
\end{figure}
\subsection {A family of piecewise smooth vector fields with poly-trajectories.} 
Let us now illustrate the regularization by convolution of a piecewise smooth vector fields presenting families of {\em closed poly-trajectories}. We refer the reader to
\cite{SM} for the relevant definitions. 

Consider the piecewise-smooth one-parameter family $X_\lambda$ with discontinuity locus $\Sigma = \{y=0\}$ given by
$$X_{\lambda}=
\mathds{1}_{\{y>0\}}X_{+,\lambda}+\mathds{1}_{\{y<0\}}X_-$$
where
\begin{equation}X_{+,\lambda}=\dfrac{\partial}{\partial x}+(-3(x+\lambda)^2+
2(x+\lambda)+\dfrac{7}{4})\dfrac{\partial}{\partial y}\label{falam}\end{equation}
and
\begin{equation}X_-=-\dfrac{\partial}{\partial x}+(3x^2-7x+2)\dfrac{\partial}{\partial y},\label{falam1}\end{equation}
which is illustrated in the Figure \ref{fig-polytraj}. 

 \begin{figure}[!htb]
	\begin{center}
		\begin{overpic}[width=4cm]{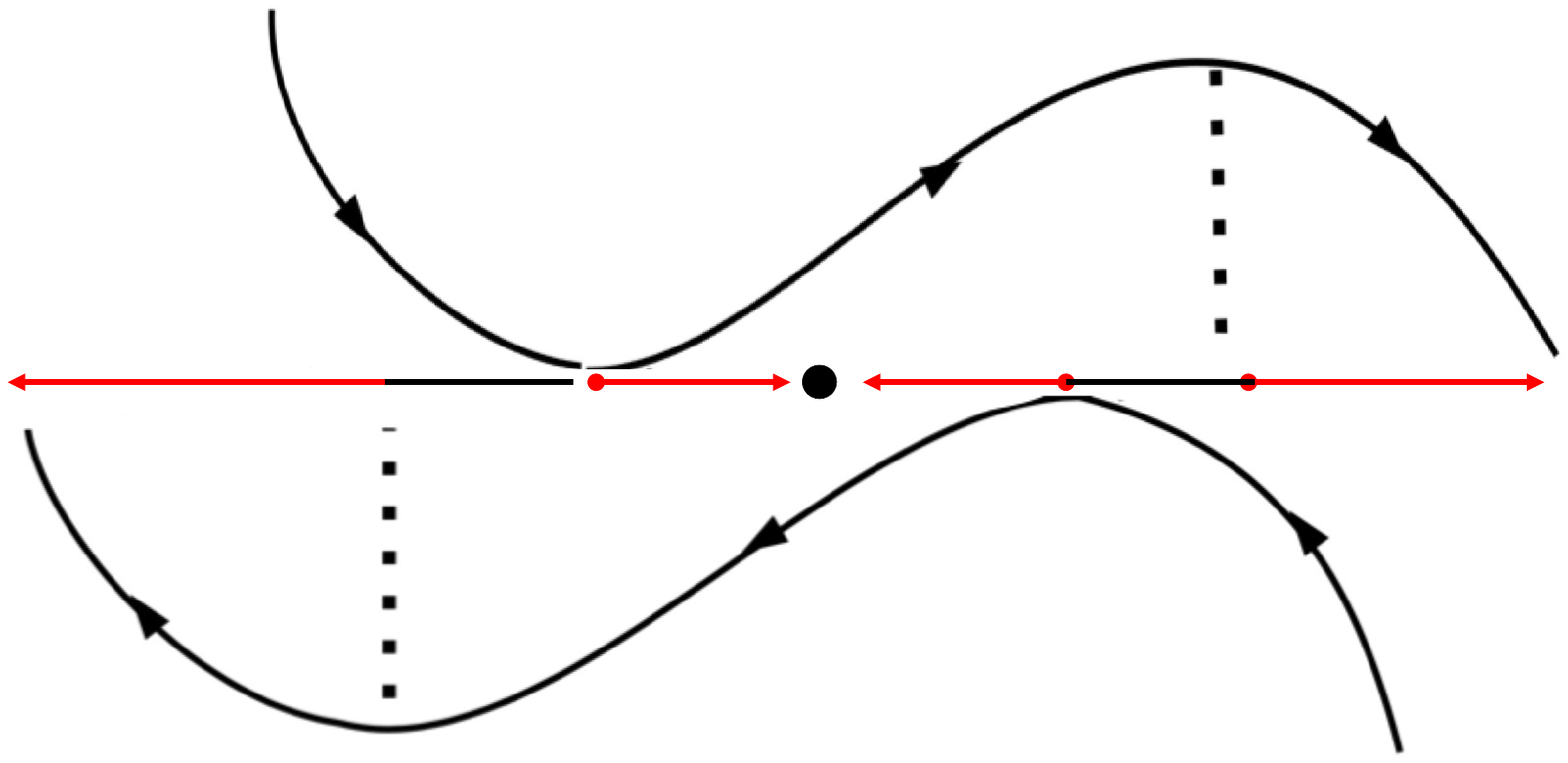}
				%\begin{overpic}[grid,tics=5,width=4cm]{c9}
				\put(35,12){$\lambda<-\dfrac{5}{6}$}		
			\end{overpic}\quad
		 \begin{overpic}[width=4cm]{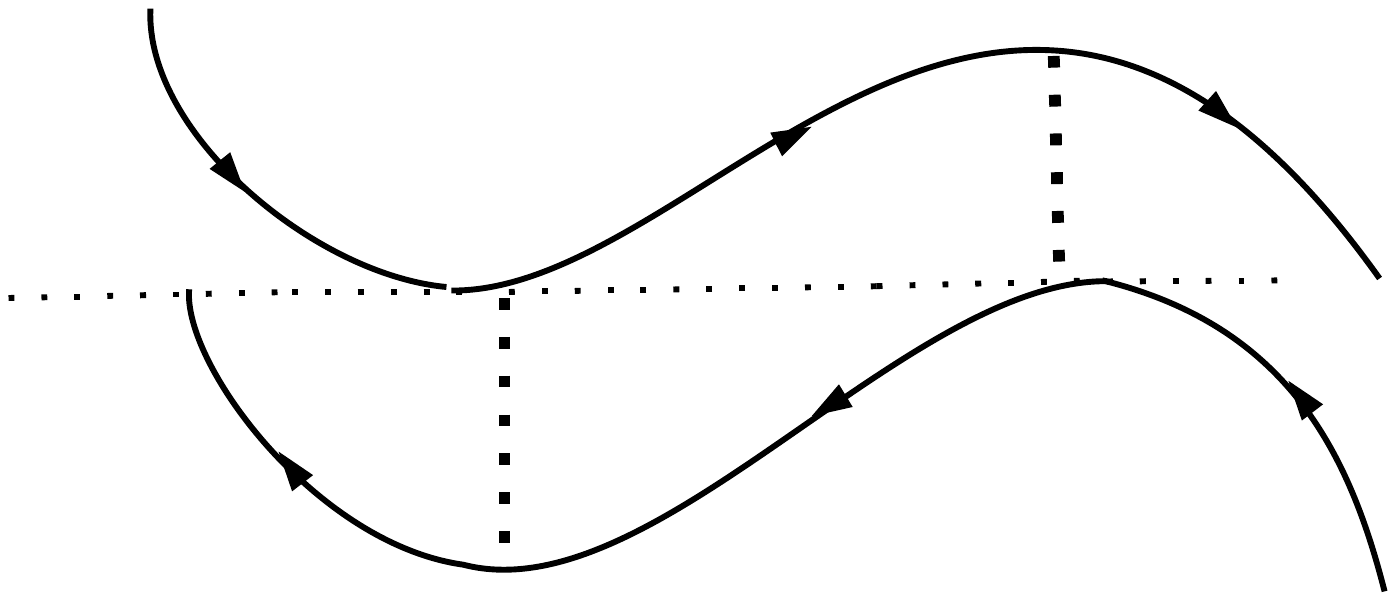}
			%	\begin{overpic}[grid,tics=7,width=4cm]{c8}
			\put(35,12){$\lambda=-\dfrac{5}{6}$}	
		\end{overpic}\\

		\begin{overpic}[width=4cm]{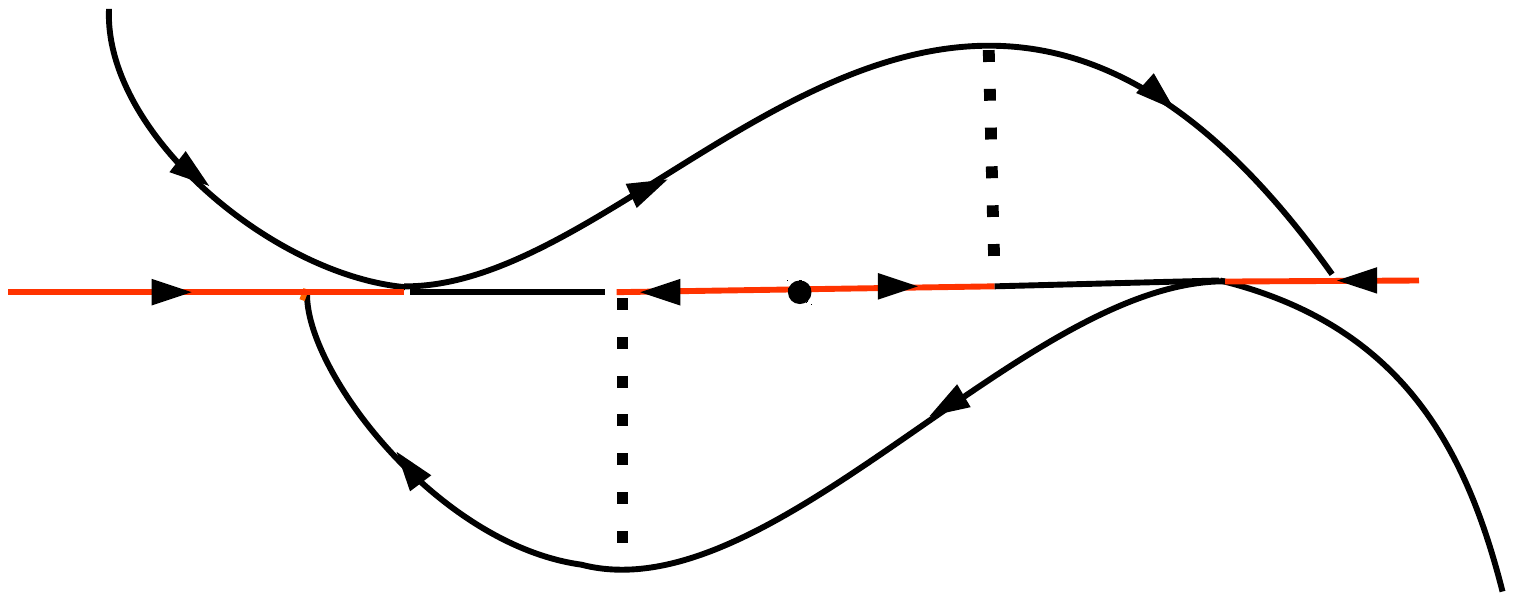}
		%\begin{overpic}[grid,tics=7,width=4cm]{c0}
		\put(49,14){$-\dfrac{5}{6}<\lambda<0$}
		\end{overpic}\quad
		\begin{overpic}[width=4cm]{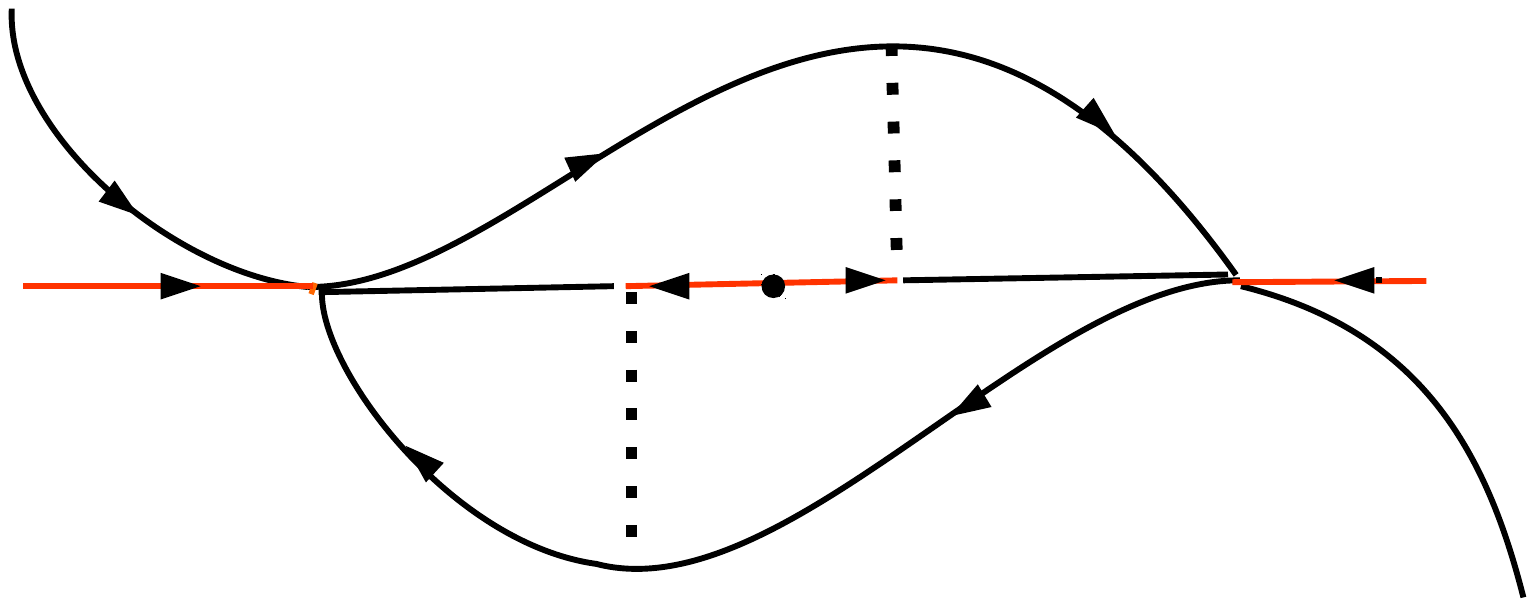}
		%\begin{overpic}[grid,tics=7,width=4cm]{c0}
		\put(49,14){$\lambda=0$}
		\end{overpic}
		\\
		\begin{overpic}[width=4cm]{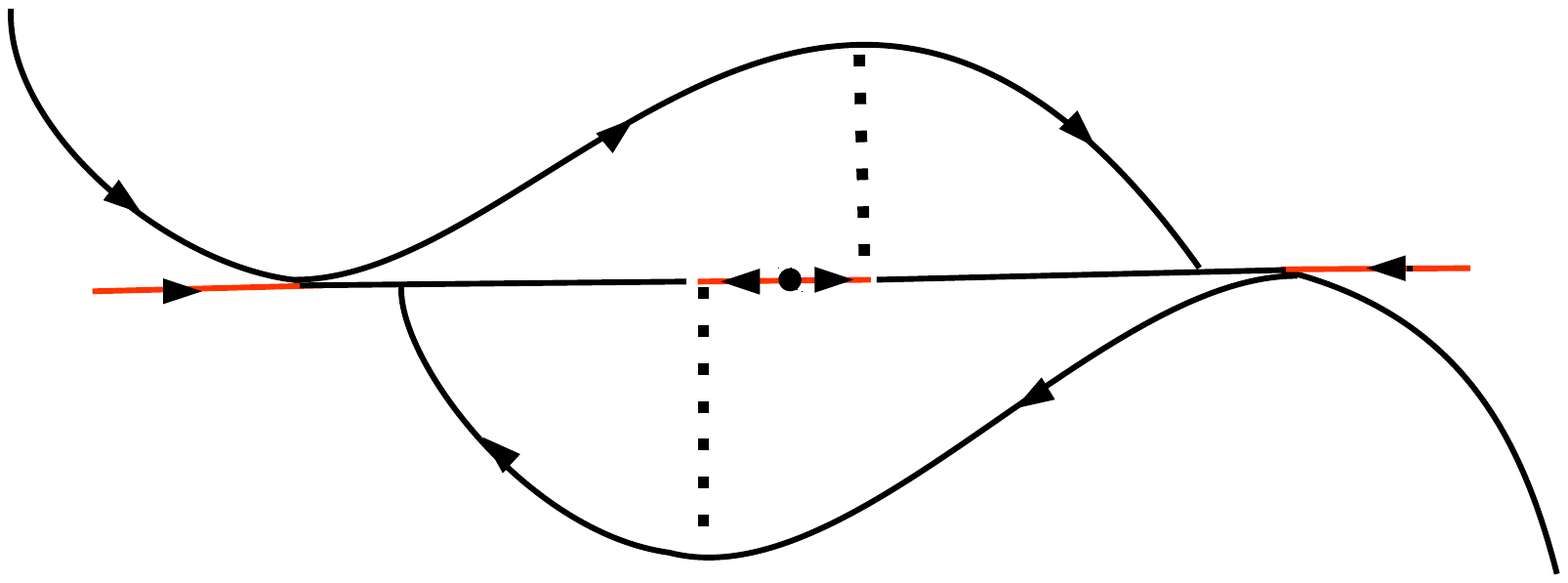}
			%\begin{overpic}[grid,tics=5,width=4cm]{c1}
			\put(35,12){ $0<\lambda<\dfrac{5}{6}$}
		\end{overpic}\quad
		\begin{overpic}[width=4cm]{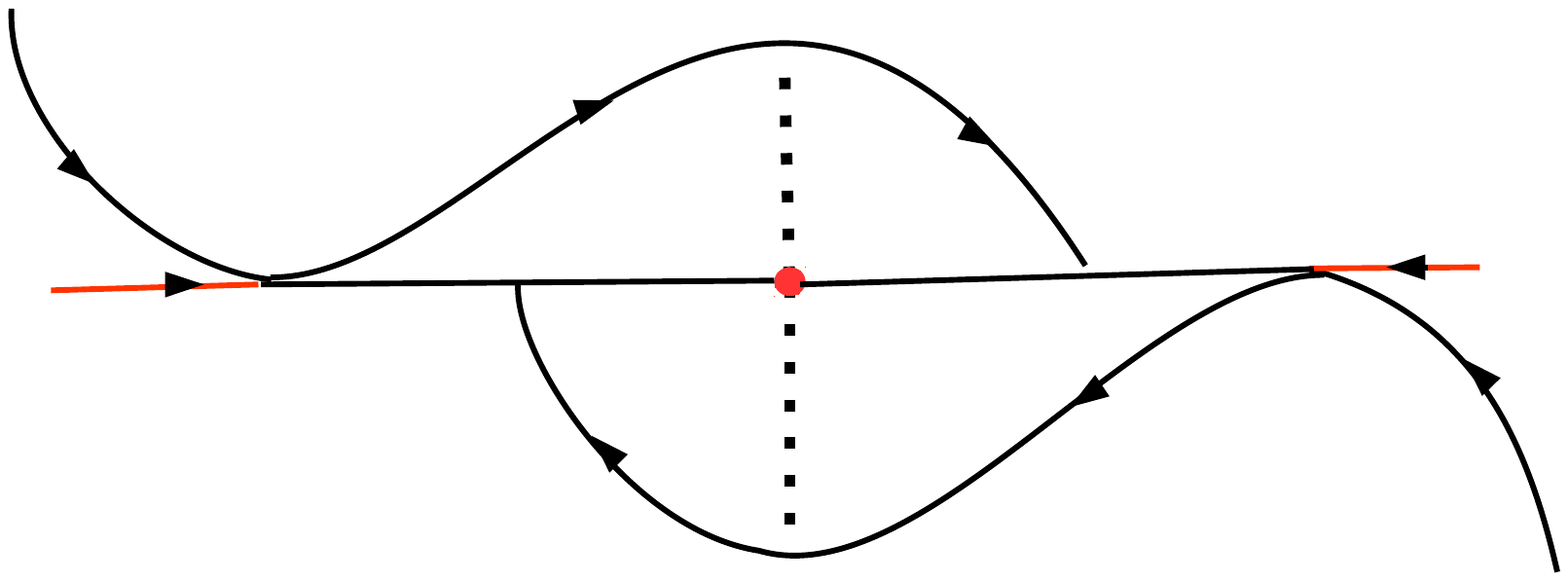}
			%\begin{overpic}[grid,tics=7,width=4cm]{c2}
			\put(49,14){$\lambda=\dfrac{5}{6}$}
	\end{overpic} 
	
	\end{center}
	
		\caption{\small Phase portrait of $X_{\lambda}$ for $-\dfrac{5}{6}\leq\lambda\leq\dfrac{5}{6}$. 
		The dotted lines at $y=0$ are the equilibrium points of the sliding field. 
		The solid black lines at $y=0$ are sewing points and the red lines at $y=0$ represent the sliding  or escaping points.} \label{fig-polytraj}
\end{figure}

It is not hard to see that closed poly-trajectories of $X_{\lambda}$
can occur only for $-\dfrac{5}{6}<\lambda<\dfrac{5}{6}$. In this case 
the sliding/scaping and sewing regions are the subsets $\Sigma^{sl} ,\Sigma^{sw} \subset \Sigma$ given respectively by
$$\Sigma^{sl}=\left(-\infty,-\frac{1}{2}-\lambda\right)\cup\left(\frac{1}{3},\frac{7}{6}-\lambda\right)\cup\left(2,+\infty\right)$$
and 
$$\Sigma^{sw}=\left(-\frac{1}{2}-\lambda,\frac{1}{3}\right)\cup\left(\frac{7}{6}-\lambda,2\right).$$
Notice that both $X_{+,\lambda}$ and $X_-$ have two distinct fold points, which are situated respectively at the coordinates $(-\frac{1}{2}-\lambda,0)$, $(\frac{7}{6}-\lambda,0)$ and  $(\frac {1}{3},0)$, $(2,0).$ We further observe that, for the parameter value $\lambda = -5/6$, the system presents the symmetry 
$$X_{+,\lambda} = -X_-.$$
Let us now describe (without proofs) some of the dynamical features of the regularized family $m_\eps \ast X_\lambda$ for $\eps > 0$ sufficiently small: 
\begin{itemize}
\item[(a)] For each parameter value $\lambda$ in the open interval $(-\frac{5}{6},0)$, the family $m_\eps \ast X_\lambda$ has an attracting limit cycle $\Gamma_{\lambda,\eps}$. As $\eps \rightarrow 0$, $\Gamma_{\lambda,\eps}$ converges (with respect to the Hausdorff distance) to a poly-trajectory $\Gamma_{\lambda}$ of $X_\lambda$ containing two fold points.  
\item[(b)] For each parameter value $\lambda$ in the open interval $(0,\frac{5}{6})$, the family $m_\eps \ast X_\lambda$ has an attracting limit cycle $\Gamma_{\lambda,\eps}$ which converges, as $\eps \rightarrow 0$ to a sewing-type poly-trajectory $\Gamma_{\lambda}$ of $X_\lambda$.
\item[(c)] As $\lambda$ crosses a bifurcation value of the form $5/6+O(\eps)$, the family of cycles $\Gamma_{\lambda,\eps}$ disappears in a Hopf-type supercritical bifurcation.
\end{itemize}
In the figure below, we illustrate several phase portraits obtained numerically for the regularized family with $\eps = \tfrac{1}{100}$.

\begin{figure}[!htb]
\centering
\begin{overpic}[width=13.5cm]{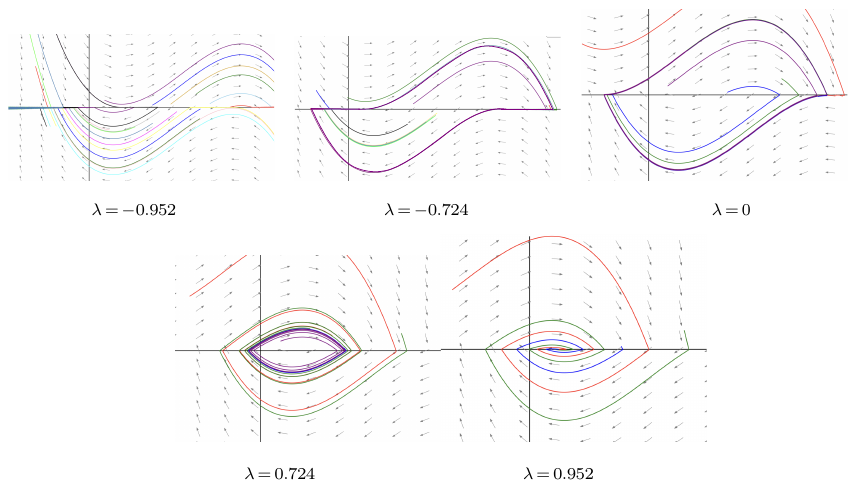}			
\end{overpic}
\caption{Phase portrait of  $m_\eps \ast X_\lambda$ for $\eps = \frac{1}{100}$.}
\end{figure}

We remark that, using the same mollifier of the previous subsection, the regularized family can be written, in the coordinates of the $\eps$-directional blowing-up as 
\[
\mathcal{X} = \eps y \frac{\partial}{\partial x} + 
G(x, y, \varepsilon, \lambda) \frac{\partial}{\partial {y}},
\]
where, restricted to the region $R = \{|y| < 1 - \eta\}$, the coefficient $G|_{\eps = 0}$ is the polynomial 
{\small 
\[\left(\frac{15}{8} + \lambda - \frac{3}{2}\lambda^2\right)
+ \left(-\frac{5}{2} - 3\lambda\right)x
+ \left(-\frac{1}{8} + \lambda - \frac{3}{2}\lambda^2\right) y
+ \left(\frac{9}{2} - 3\lambda\right)x y- 3x^2 y.\]
}
(up to a $O(\eta)$-correction term).

\section{Regularization by convolution}\label{s4}
This section introduces the main technical tools used in the paper. We begin by recalling the basic definitions of manifolds with corners, directional blow-up maps, and regularization by convolution. The main result is the smoothing procedure, which is proved in Theorem~\ref{theorem-smoothing-rpss}.

\subsection{Piecewise smooth spaces and functions}\label{subsect-piecewisesmoothfcs}
We work in the category of smooth manifolds with corners. Let us briefly recall the definitions and refer to \cite {KottkeMelrose2015} or \cite{Melrose} for a  detailed treatment.  A manifold with corners  of dimension $n$ is a paracompact Hausdorff space $M$ with a smooth structure which is locally
modelled through {\em adapted local coordinate charts} $x = (x_1,...,x_n)$ as open subsets of $\cR^k_{\ge 0}\times \cR^{n-k}$, for some $0\le k \le n$.  
A submanifold of a manifold with corners is a closed subset $M' \subset M$ which locally expressed, in adapted coordinate charts as a product of open subsets
$$\cR^{k'}_{\ge 0} \times \{0\}\subset \cR^k_{\ge 0},\quad\text{and}\quad \cR^{n-k''}\times\{0\} \subset \cR^{n-k}$$ 
for some integers $k',k''$ such that $0 \le k' \le k \le k'' \le n$.  In this case, $M'$ has codimension $k''-k'$.  A boundary component of $M$ is a closed connected submanifold as above such that $k''=k$.  The boundary of $M$ is the union of all boundary components, and is denoted by $\partial M$. 

A {\tmem{smooth normal crossings
variety}} in $M$ is a closed subset $\Sigma \subset M$ given by a finite union
of codimension one smooth submanifolds, called {\tmem{components of $\Sigma$}}, satisfying the following transversality
property: at each point $p$ there is a adapted coordinate chart $x = (x_1, \ldots,
x_n)$ as above  such that $x (p) = 0$ and
\begin{equation}
  \Sigma \cap V = \{ x_1 \ldots x_c = 0 \} \label{loc-adapted-coord}
\end{equation}
for some $0 \le  c \le  n$. We will say that these local coordinates
are {\tmem{adapted}} to $(M, \Sigma)$. The number $\comp(p) = c$ will be called
the {\em number of local components} of $\Sigma$ incident at $p$.

\begin{figure}[htb]
\centering
\includegraphics[width=7.75554243736062cm,height=4.88899547422275cm]{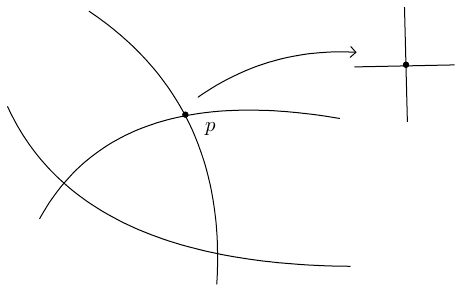}
\end{figure}

A {\tmem{piecewise smooth space}} is a pair $(M, \Sigma)$ formed by:
\begin{enumeratenumeric}
\item A manifold
with corners $M$ of dimension $n$
\item A (possibly empty) smooth normal
crossings variety $\Sigma$.
\end{enumeratenumeric}
We denote by $C^{\infty} (M, \Sigma)$ the space of smooth functions $f \in
C^{\infty} (M \setminus \Sigma)$ satisfying the following {\em extension property}:
For each connected component $W \subset M \setminus \Sigma$, the restriction
$f |_W \nobracket$ extends to a smooth function on $M$. 

The topology on $C^{\infty} (M, \Sigma)$ is the topology defined by the semi-norms
\[ \| f \|_{K, r} = \sum_{| \alpha | \le  r} \sup_K \left|
   {\partial^{\alpha}}  f (x) \right| \]
with $K$ varying over all compacts of $M$ and all $r \in \mathbb{N}$. We
remark that although $f$ is not defined at $\Sigma$, the above supremum is to
be taken among all smooth restrictions $f |_W \nobracket$ of $f$ to connected
components $W$ intersecting the compact $K.$ We will say that $C^{\infty} (M,\Sigma)$ is the space of \tmem{piecewise-smooth functions in M} with
discontinuity locus $\Sigma$.

Note that if $\Sigma = \emptyset$ then $C^{\infty}(M,\Sigma) = C^{\infty}(M)$ is simply the space of globally smooth functions on $M$, equipped with its usual topology. 
\begin{Remark}
Using adapted local coordinates $(x_1, \ldots, x_n)$ as above, a function $f
\in C^{\infty}(M, \Sigma)$ can be locally expressed as
\begin{equation} \label{loc-presentation}
  f = \sum_{\underline{s}} \tmmathbf{1}_{\{ \underline{s} \underline{x}  > 0 \}} f_{\underline{s}}
\end{equation}
where the sum is taken over all $c$-uples of signs $\underline{s} = (s_1,..,s_c)\in \{+,-\}^c$, which it will be convenient to identify with the set of {\em sign functions} 
\[
\begin{aligned}
\underline{s} : \llbracket 1, c \rrbracket &\longrightarrow \{+, -\}, \\
i &\longmapsto \underline{s}(i) = s_i .
\end{aligned}
\]
where note $\llbracket 1, c \rrbracket=\{1,\ldots,c\}$. We remark that the expression $\tmmathbf{1}_{\{ \underline{s} \underline{x}  > 0 \}}$ denotes the characteristic function of the set
$\{ s_1 x_1 >  0, \ldots, s_c x_c >  0 \}$ and that $\left\{ {f_{\underline{s}}}
\right\} $ is a collection of smooth functions in $V$. 
\end{Remark}
The above expansion
will be called a {\tmem{local presentation of $f$}} and the functions $f_{\underline{s}}$
will be called the {\tmem{local components of $f$}}.
\subsection{Regularization of piecewise-smooth functions}\label{subsection-regularizafcn}
In this subsection, we will suppose $(M,\Sigma)$ is a piecewise smooth space such that $M$ is an open subset of $\cR^n$.

Let us review the operation of regularization by convolution for functions in $L^1_{\mathrm{loc}}(M)$, the space of locally integrable functions on $M$. Recall that the singular support of a function $f \in L^1_{\mathrm{loc}}(M)$ is the closed subset
$$\tmop{singsupp} (f)=\{x \in M : \forall r>0,\, f|_{B_r(x)} \text{ is not smooth}\}$$ 
 A \tmem{mollifier} in $\cR^n$ is a smooth function $m\in C^\infty(\cR^n)$ satisfying the
following conditions: \
\begin{enumeratenumeric}
  \item $m \ge  0$
  \item $\int_{\cR^n} m = 1$, and
  \item $m$ has compact support. 
\end{enumeratenumeric}
We denote by $\Mol(\cR^n)$ the set of mollifiers in $\cR^n$, and by $\Mol_r(\cR^n)$ the subset consisting of mollifiers with support contained in the closed ball $\overline{B_r(0)}$. 

A sequence of mollifiers $\left\{ {m_k} 
\right\}$ in $\cR^n$ is called an {\tmem{regularizing sequence}} if $m_k \in
\Mol_{r_k}(\cR^n)$ with $r_k \rightarrow 0$ as $k \rightarrow \infty$.

We recall the following general result (see e.g.~\cite{HormanderI}, section 1.3):
\begin{Theorem}\label{rem-conv}Given a mollifier $m$, the convolution operator
  \[ (m \ast f) (x) = \int_{\cR^n} f (x - u) m (u) d u \]
  defines a linear map from $L^1_{\tmop{loc}} (\cR^n)$ to $C^{\infty}
  (\cR^n)$. If $\{ m_k \}$ is a regularizing sequence, then for each
  compact set $K$, the restriction of $m_k \ast f$ to $K$ converges to $f$ in
  the $L^1$-norm. Moreover, if $K \subset \cR^n \setminus
  \tmop{singsupp} (f)$ then $m_k \ast f$ converges to $f$ in the $C^{\infty}$-topology.
\end{Theorem}
We now consider a particular 
convolution-based regularization. This construction yields a one-parameter family 
of smooth functions, which may be interpreted as an \emph{unfolding} of the original 
discontinuous object.

Given a mollifier $m \in \Mol(\cR^n)$, we define, for each $\eps >0$, the {\em $\eps$-rescaled mollifier} by
\begin{equation}\label{rescaled-mollifier}
m_\eps(x) = \frac{1}{\eps^n}\, m\left(\frac{x}{\eps}\right)
\end{equation}
As a consequence, given an open set $M \subset \cR^n$ equipped with a smooth normal crossings variety $\Sigma$, and a function $\vf\in C^\infty(M,\Sigma)$, the expression
$$
m_\eps \ast \vf,
$$
defines a one-parameter family of smooth functions on $M$, depending smoothly on $\eps > 0$. More explicitly, a simple coordinate change allows to write the convolution integral as
\begin{equation}\label{expr-f_k}
m_\eps \ast \vf(x)= \int_{\cR^n} f (x - \eps t)\, m (t) d t
\end{equation}
We now complete this family to $\{\eps = 0\}$ as follows:
\begin{Definition}\label{def-regfconv}
The {\em regularization by convolution }of $\vf \in C^\infty(M,\Sigma)$ (with mollifier $m$) is the function defined by
$\vf^{\rg}(x,\eps) = m_\eps \ast \vf(x)$ for $\eps >0$ and $\vf^\rg(x,0) = \vf(x)$.
\end{Definition}

Note that $\vf^\rg$ defines a piecewise-smooth function in the product $N = M \times (\cR_{\ge 0},0)$, with discontinuity locus $\Sigma \times \{0\}$ (which we will still note $\Sigma$). Note also that $M=N \cap \{\eps = 0\}$ is the boundary of $N$. 

Based on this discussion, we introduce the following class of geometric objects, which will be the underlying ambient spaces for smoothing procedure to be defined in the next section. A {\tmem{piecewise smooth regularized space}}  is a triple $(N, M, \Sigma)$
formed by the following objects:
\begin{enumeratenumeric}
  \item A manifold with corners $N$ of dimension $n + 1$.
  \item A codimension one boundary component $M \subset \partial N$, called
  the {\tmem{initial manifold}}.
  \item A smooth normal crossings submanifold $\Sigma \subset M$, such that $(M,
  \Sigma)$ is a piecewise smooth space. 
\end{enumeratenumeric}
From now on, we consider the space 
$$
C^\infty(N,M,\Sigma)=\{f \in C^\infty(N\setminus \Sigma)\ :\ f|_M \in C^\infty(M,\Sigma)\}
$$
i.e. the space of functions $f$ which are smooth on $N \setminus \Sigma$ and whose restriction to the boundary component $M$ define an element $f|_M$ belonging to the space $C^\infty(M,\Sigma)$, as defined in subsection \ref{subsect-piecewisesmoothfcs}. Note that $C^\infty(N,M,\empty) = C^\infty(N)$.

We can summarize the construction described in this subsection as follows: Assume that $(M,\Sigma)$ is a piecewise smooth space such that $M$ is an open subset of $\cR^n$. Then, given mollifier $m\in \Mol(\cR^n)$, the regularization by convolution defines a linear operator
\begin{align*}\label{regularization-operator}
\reg_m \colon C^\infty(M,\Sigma) &\longrightarrow C^{\infty}(N,M,\Sigma) \\
\vf &\longmapsto \vf^\rg
\end{align*}
where $(N,M,\Sigma)$ is the piecewise smooth regularized space with ambient $N=M\times (\cR_{\ge 0},0)$ and initial manifold $M=N\times\{0\}$.  

We will say that $\reg_m$ is the {\em regularization operator} associated to $m$ and we will denote by $C^{\infty,\reg}(N,M,\Sigma)$ the image of the operators $\reg_m$,  for all mollifiers $m$ varying in $\Mol(\cR^n)$.  
\begin{Remark}\label{remark-dependence-base-point}
(1) Note that the convolution integral  \( f_\varepsilon = m_\varepsilon \ast f \)
is written in terms of the global Euclidean coordinates of \( \mathbb{R}^n \).
For our purposes, it will be necessary to express this integral in local adapted
coordinates, as in~\eqref{loc-adapted-coord}, which are defined in the vicinity of each point 
in discontinuity locus.

More precisely, let \((U,\varphi)\) be a local chart defining adapted coordinates, and let
\(\psi = \varphi^{-1}\) denote the inverse diffeomorphism, with domain on the open subset $V=\varphi(U)$ of $\cR^n$.  
The goal is to write the pull-back of the regularized function $f_\eps \circ \psi$ in terms of a convolution-type integral involving the pull-back of the piecewise smooth function $f\circ \psi$.  

By a change of variables in the convolution integral (\ref{expr-f_k}) we can write 
\begin{equation}\label{convolution-local-charts}
(f_\varepsilon \circ \psi)(x)
= \int_{\mathbb{R}^n} (f \circ \psi)(x - t)\,
\tmmathbf{m}_{\varepsilon,x}(t)\, dt .
\end{equation}
Here \(t \mapsto \tmmathbf{m}_{\varepsilon,x}(t)\) is a family of functions depending smoothly on $\eps,x$, with domain some neighborhood of the origin in
\(\mathbb{R}^n\).  More precisely, one has 
\[
\tmmathbf{m}_{\varepsilon,x}(t)
= \frac{1}{\varepsilon^n}\, u(x,t)\,
m\!\left(\frac{t\, v(x,t)}{\varepsilon}\right),
\]
where \(m\) is the original mollifier appearing in \eqref{rescaled-mollifier}, \(u\) is a smooth function and \(v\) is matrix valued smooth function defined on
some open neighborhood of the origin in \(V \times \mathbb{R}^n\) which are invertible near \((x,t) = (0,0)\).  

Note that, for each fixed value of \(x\), the support of the function
\(t \mapsto \mathbf{m}_{\varepsilon,x}\) shrinks to the origin as
\(\varepsilon \to 0\). Consequently, for \(\varepsilon > 0\) sufficiently
small, the above integral is well defined for all \(t \in \mathbb{R}^n\).

In other words, when expressed in the local adapted coordinates \(x\), the
regularized function \(f_\varepsilon \circ \psi\) can be written as the
convolution of \(f \circ \psi\) with a \emph{generalized mollifier}
\(\tmmathbf{m}_{\varepsilon,x}\), which, in these coordinates, depends both on
the integration variable \(t \in \mathbb{R}^n\) and on the base point \(x\).

For the computations that follow, this additional dependence of the mollifier
\(\tmmathbf{m}_{\varepsilon,x}\) on the base point is entirely harmless.
Accordingly, we shall omit it from the notation.

(2) The regularization by convolution can be defined more generally for an arbitrary Riemannian manifold $M$ through the {\em Greene-Wu convolution formula}
\[ f_{\varepsilon} (p) = \int_{v \in T_p M} f (\exp_p v)\, m_{\varepsilon} (v) d
   v \]
where $\exp_p:T_p M\rightarrow M$ denotes the exponential map associated to the metric and the integral is computed with respect to the Lebesgue measure on $T_p M$ induced by the metric. In this more general setting, the Greene–Wu convolution formula expressed in local charts has the same form as~\eqref{convolution-local-charts}. We refer the reader to~\cite{GreeneWu1979} for further details.
\end{Remark}
\section{Smoothing $m_\eps \ast f$ by a sequence of blowing-ups}\label{s5}
In this section we shall prove a first Smoothing Theorem for piecewise smooth functions, which states that the regularized function $f^\reg = \reg_m(f)$ can be smoothed by a finite sequence of {\em blowing-ups} in the ambient space $(N,M,\Sigma)$.
\subsection{Blowing-up map}\label{subsect-blow-up}
Let us briefly recall the blowing-up construction. Given natural numbers $1\le k \le n$, the blowing-up of $N=\cR^n$ with center on the subspace $C = \{0\} \times \cR^{n-k}$ is given by a surjective proper analytic map $\Phi: \tilde N \rightarrow N$, where 
\[
\tilde N = \cR_{\ge 0} \times \cS^{k-1} \times \cR^{n-k}
\]
is a manifold (with boundary) and the mapping is defined as follows: If we consider the coordinates $r\in \cR_{\ge 0}$, $(\bar y_1,\cdots,\bar y_k) \in \cS^{k-1}$ and $(y_{n-k+1},..,y_n) \in \cR^{n-k}$ then $\Phi$ is given by the equations
\[
x_i = r \bar y_i,\;\; i=1,..,k,\qquad x_j = y_j,\;\; j=k+1,..,n
\]
We say that $\tilde N$ is the {\em blown-up space} and that $D=\{r=0\} \approx \cS^{k-1} \times \cR^{n-k}$ is the {\em exceptional divisor} of the blowing-up. Note that $\Phi$ maps the codimension one manifold $D = \{r = 0\} = \cS^{k-1} \times \cR^{n-k}$ (which is the boundary of $\tilde N$) onto $C$. 
Note also that the blow-up restricts to a diffeomorphism between $\tilde N \setminus D$ to $N \setminus C$.  

The blowing-up construction can be iterated using local coordinates. Each blow-up eventually produces a new exceptional divisor (that is, a new boundary component). This naturally places the discussion in the framework of smooth manifolds with corners, where one defines the general operation of blowing-up with center on smooth submanifolds. We refer the reader to \cite{KottkeMelrose2015}, or \cite{Melrose}, chapter 5, for
the detailed definitions.
\subsection{Directional blowing-ups in piecewise smooth spaces}\label{subsect-local-situation-I}
As a preliminary step toward the proof of the Smoothing Theorem, we need to analyze the effect of a sequence of directional blowings-up on a piecewise smooth space $(M, \Sigma)$.

We suppose that $\Sigma$ is the
union of $c \geqslant 1$ connected codimension-one submanifolds (in normal
crossings position) and fix, once and for all, an enumeration of such
manifolds,
\[ \{ \Sigma_i \}_{i \in \llbracket 1, c \rrbracket} \]
For each subset $I \subset \llbracket 1, c \rrbracket$ we define the
{\tmem{stratum}} $\Sigma_I \subset \Sigma$ as the subset given by intersection of the corresponding hyperplanes
\[ \Sigma_I = \bigcap_{i \in I} \Sigma_i \]
By the normal crossings property, each $\Sigma_I$ is a closed submanifold of
dimension $n -\#I$, where $\#I$ denotes the cardinality of $I$. Note that
$\Sigma_I$ is empty whenever $\#I > n$ (since no more than $n$ manifolds can
intersect transversally at a point). Therefore, all strata are indexed by
subsets $I$ of cardinality $\leqslant n$.

At each point $p \in \Sigma_I$ there is a local chart $(V, x)$, with the
coordinates $x = (x_I, x')$, labeled $x_I = (x_i : i \in I)$ and $x' \in
\cR^{n -\#I}$, such that we can write
\[ \Sigma \cap V = \Sigma^{(I)}= \left\{ \prod_{i \in I} x_i = 0 \right\}
\]
We will refer to each one of these charts as {\tmem{a local adapted chart}}
(for $(M, \Sigma)$) {\tmem{centered}} at $\Sigma_I$.

\begin{Remark}
  Consider a piecewise smooth function $f \in C^{\infty} (M, \Sigma)$. The
  restriction of $f$ to the domain of an adapted chart $(V, x)$ centered at
  $\Sigma_I$ defines a function in $C^{\infty} (V, \Sigma^{(I)} \cap V)$. More
  precisely, in the coordinates of the chart, we can write the local expansion
  \begin{equation}
    f_I = \sum_{\underline{s} \in \{ +, - \}^I} \tmmathbf{1}_{\left\{
    \underline{s}  \underline{x} > 0 \right\}} f_{\underline{s}}
    \label{definition-fI}
  \end{equation}
  where $f_s$ are smooth functions and we use the notation $\left\{
  \underline{s}  \underline{x} > 0 \right\}$ to indicate the subset of
  $\cR^n$ defined by the inequalities $\{ s (i) x_i > 0 : i \in I \} .$
  Therefore, in these coordinates, we can interpret $f_I$ as a piecewise
  smooth function, {\tmem{discontinuous}} in the variables $x_I = (x_i)_{i \in
  I}$ and depending smoothly on the remaining variables $x'$. 
\end{Remark}

For each index $i_1 \in I$, we will denote by $B_{(I, i_1)}$ the
{\tmem{$i_1^{\tmop{th}}$-directional blowing-up map}}, given by
\[
\begin{cases}
     x_i = z_{i_1} z_i, & i \in I \setminus \{ i_1 \},\\[2pt]
     x_{i_1} = \, z_{i_1},& z_{i_1} \in \cR_{\ge 0}\\[2pt] 
     x' = z'.
\end{cases}
\]
Note that, geometrically, this map corresponds to one of the charts of the blowing-up of the ambient space with center $\Sigma_I$. Our next goal is to describe the behaviour of a piecewise function $f$ having an
expansion as (\ref{definition-fI}) under such directional blowing-up.

\noindent{\bf Convention.} To cover the entire blown-up space by directional charts, we must also consider the $(-i_1)^{th}$–directional blowing-up map
\[
\begin{cases}
     x_i = z_{i_1} z_i, & i \in I \setminus \{ i_1 \},\\[2pt]
     x_{i_1} = -\, z_{i_1},& z_{i_1} \in \cR_{\ge 0}\\[2pt]
     x' = z'.
\end{cases}
\]
All the results that we will prove using $i_1^{th}$–directional chart have completely analogous proofs in the $(-i_1)^{th}$–directional charts.  
However, including both families of charts simultaneously would make the notation unnecessarily heavy.  
Therefore, to simplify both notation and exposition, we will systematically restrict our attention to the positive directional charts in the statements that follow.

Based on this convention, we associate to $f_I$ and to the index $i_1\in I$ a new piecewise smooth function $f_{I\setminus \{i_1\}}$ given by
\begin{equation}\label{definition-fI-i1}
f_{I\setminus \{i_1\}} = \sum_{\underline{\tilde s} \in \{ +, - \}^{I\setminus {i_1}}} \tmmathbf{1}_{\left\{ \underline{\tilde s}
     \underline{x} > 0 \right\}} f_{{\underline{s}}}
\end{equation}
where the sum is now taken over all sign functions with domain $I \setminus \{i_1\}$, and we denote by $\underline{{s}} \in \{ +, - \}^I$ the sign
  function which extends $\underline{\tilde s}$ by defining $\underline{s} (i_1) = +$. 
  
 Note that the function $f_{I\setminus \{i_1\}}$ has discontinuity locus given by the closure of $\Sigma \setminus \{x_{i_1} = 0\}$. 
  Intuitively, $f_{I\setminus \{i_1\}}$ is obtained from $f_I$ by {\em removing} the $\{x_{i_1} = 0\}$-component from the discontinuity locus. 
  \begin{Example}
  Suppose that $n=2$ and $I=\{1,2\}$. Then, in adapted coordinates we can write $\Sigma = \{x_1 x_2 = 0\} \subset \cR^2$ and a function $f_I \in C^\infty(\cR^2,\Sigma)$ is given by a sum
{\small
 $$
 \tmmathbf{1}_{\{x_1 > 0, x_2 > 0\}} f_{(+,+)} + \tmmathbf{1}_{\{x_1 < 0, x_2 > 0\}} f_{(-,+)} + \tmmathbf{1}_{\{x_1 > 0, x_2 < 0\}} f_{(+,-)} + \tmmathbf{1}_{\{x_1 < 0, x_2 < 0\}} f_{(-,-)} 
$$
}
where each $f_{(\cdot,\cdot)}$ is a smooth functions in $\cR^2$. If we choose the index $i_1 = 1$ then we obtain
  $$
  f_{I\setminus\{1\}} = \tmmathbf{1}_{\{x_2 > 0\}} f_{(+,+)} + \tmmathbf{1}_{\{x_2 < 0\}} f_{(+,-)}
  $$
  which is a piecewise smooth function in $C^\infty(\cR^2,\{x_2 = 0\})$.
\begin{figure}[htb]
\centering
{\includegraphics[width=8.75896136691591cm,height=4.47006591892956cm]{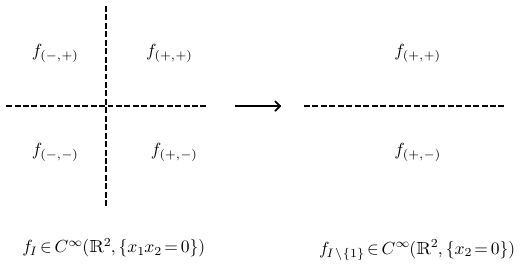}}
\end{figure}
  \end{Example}

\begin{Lemma}
  \label{Lemma-composedB1}Suppose that $f \in C^{\infty}(\cR^n,
  \Sigma^{(I)})$. Then the composed function $f_I \circ B_{I, i_1}$ defines an element
  of $C^{\infty} (W, \Sigma^{(I \setminus \{ i_1 \})})$, where $W = B_{(I,
  i_1)}^{- 1} (\cR^n)$. Moreover, we have the equality
  \[ f_I \circ B_{(I, i_1)} = f_{I \setminus \{ i_1 \}} \circ B_{(I, i_1)} \]
\end{Lemma}

\begin{Remark}
   In other words, this result says that the $B_{(I,i_1)}$-directional blowing up gives exactly the same results when applied to either $f_I$ or  $f_{I \setminus \{ i_1 \}}$. Intuitively, the $i_1$-component of the discontinuity
  locus $\Sigma$ is {\tmem{removed}} by the $i^{\tmop{th}}$ directional
  blowing-up.
\end{Remark}

\begin{proof}
We define the index set $J = I
  \setminus \{ i_1 \}$ and denote by $g = f_I \circ
  B_{(I, i_1)}$ the transformed function. Using the above expression for the blowing-up and the
  expansion (\ref{definition-fI}) for $f_I$, we can write
  \begin{align*}
g(z) &= 
\tmmathbf{1}_{\{ z_{i_1} > 0 \}}
\Biggl(
  \sum_{\underline{\tilde s}}
  \tmmathbf{1}_{\{ \underline{s x} \circ B > 0 \}}
  \, f_{\underline{s}} \circ B
\Biggr) \\[4pt]
&\quad + 
\tmmathbf{1}_{\{ z_{i_1} < 0 \}}
\Biggl(
  \sum_{\underline{\tilde s}}
  \tmmathbf{1}_{\{ \underline{s x} \circ B > 0 \}}
  \, f_{{\underline{s}}} \circ B
\Biggr).
\end{align*}
 where, the sum is taken over all sign functions $\underline{\tilde{s}}\in \{ +, - \}^J$ with domain $J$ and $\underline{s} \in \{ +, - \}^I$ denotes the unique sign
  function which extends $\underline{\tilde{s}}$ by defining
  $s(i_1) = +$.
  
  We now observe that, by the definition of the directional chart, we have $z_{i_1}\geqslant 0$.
  Therefore, only the first sum in the above expansion in non-vanishing.
 Also, the expression $\underline{s x} \circ B > 0$ appearing in the sum
  corresponds to the collection of inequalities
  \begin{equation}
    \{ s (i)\, z_{i_1}\, z_i > 0 \mid i \in J \} \label{equalities-si}
  \end{equation}
  Therefore, when restricted to the domain $W\setminus D = \{ z_{i_1} > 0 \}$, the above
  expression defines an element $g \in C^{\infty} (W \setminus D, \Sigma_J)$. The
  extension to an element $\overline{g}$ in $C^{\infty} (W, \Sigma_J)$ is 
  defined simply by removing the $z_{i_1}$-factor in the inequalities
  (\ref{equalities-si}), namely by defining
  \[ \overline{g} (z) = \sum_{\underline{\tilde s} \in \{ +, - \}^J} \tmmathbf{1}_{\left\{ \underline{\tilde s} 
     \underline{z} > 0 \right\}} f_{\underline{s}} \circ B_{I,i_1} (z) \]
 Notice that such expression of $\overline{g}$ coincides with the expression we would obtain by considering the function $f_{I \setminus \{ i_1 \}} \circ B_{(I, i_1)}$ instead of $f_{I} \circ B_{(I, i_1)}$.   Therefore, the last equality in the statement proved.
\end{proof}

More generally, consider an index set $I \subset \{ 1, \ldots n \}$ and a
sequence of directional blowing-up maps $B_{I_0, i_1}, B_{I_1, i_2}, \ldots,
B_{I_{k - 1}, i_k}$ defined as follows:
\begin{enumerate}
  \item $I_0 = I$ and $i_1 \in I_0$,
  \item For $j = 1, \ldots, k - 1$: \ $I_j = I_{j - 1} \setminus \{ i_j \}$
  and $i_{j + 1} \in I_j$.
\end{enumerate}
We write $\underline{i} = (i_1, \ldots, i_k)$ and denote by $I \setminus
\underline{i}$ the index set $I \setminus \{i_1, \ldots, i_k \}$. The map
\begin{equation}
  B_{\left( I, \underline{i} \right)} \hspace{0.27em} = \hspace{0.27em}
  B_{(I_0, i_1)} \circ \cdots \circ B_{(I_{k - 1}, i_k)}
  \label{composed-blowupdef}
\end{equation}
obtained by composition of the corresponding directional blowings-up will be called the {\tmem{$\left( I, \underline{i} \right)$-composed directional
blowing-up}}. By convention, $B_{\left( I, \underline{i} \right)}$ is the
identity if either $I = \emptyset$ or $k = 0$.

Given a function $f_I \in C^{\infty} (\cR^n,\Sigma^{(I)})$, we also define inductively
$$
f_{I \setminus
     \underline{i}} = (f_{I\setminus \{i_1\}})_{\setminus \{i_2\}})_{\cdots})_{\setminus \{i_k\}}
$$
where $f_{I\setminus \{i\}}$ is given by (\ref{definition-fI-i1}). In other words, $f_{I \setminus
     \underline{i}}$ is obtained from $f_{I}$ by successively removing the hyperplanes $\{x_{i_1}=0\},\ldots,\{x_{i_k}=0\}$ from the discontinuity locus. As a consequence, 
     $f_{I\setminus \{i\}}$ is an element of $C^\infty(\cR^n,\Sigma_{I\setminus \underline{i}})$, where 
$$
\Sigma_{I\setminus \underline{i}} = \left\{ \prod_{i \in I\setminus \underline{i}} x_i = 0 \right\}
$$
Now, a simple inductive application of Lemma \ref{Lemma-composedB1} leads to the
following result.
\begin{Lemma}
  \label{Lemma-composedB}Let $f_I \in C^{\infty} (\cR^n, \Sigma^{(I)}$. \
  Then, the composed function $f_I \circ B_{\left( I, \underline{i} \right)}$
  defines an element of $C^{\infty} \left( \cR^n, \Sigma^{(I \setminus
  \underline{i})} \right)$. Moreover, we have the identity
  \[ f_I \circ B_{\left( I, \underline{i} \right)} = f_{I \setminus
     \underline{i}} \circ B_{\left( I, \underline{i} \right)} \]
\end{Lemma}

\begin{Remark}
  \label{Remark-exprB-nonreg}For future reference, we observe that the composed blowing-up map
  $B_{\left( I, \underline{i} \right)}$ has the following explicit monomial
  expression
\[
\begin{cases}
x_{i_1} = z_{i_1}, \\[4pt]
x_{i_2} = z_{i_1} z_{i_2}, \\[4pt]
\vdots \\[4pt]
x_{i_{k-1}} = z_{i_1} \cdots z_{i_{k-1}}, \\[4pt]
x_{i_k} = z_{i_1} \cdots z_{i_k}, \\[4pt]
x_i = (z_{i_1} \cdots z_{i_k})\, z_i, \quad i \in I \setminus \underline{i}, \\[4pt]
x' = z', \hspace{2.3cm} x', z' \in \cR^{n - \#I}.
\end{cases}
\]
where $z_{i_1},\ldots,z_{i_k}$ are variables in $\cR_{\ge 0}$.
\end{Remark}

\subsection{Blowing-up in piecewise smooth regularized spaces}\label{subsect-local-situation-reg-I}
We now consider the effect of a sequence of phase directional blowings-up when applied to a regularized function $f^\rg \in C^{\infty,\rg}(N,M,\Sigma)$, where $(N, M , \Sigma)$ is a piecewise smooth regularized space. 

As previously, note that we can write $\Sigma \subset M$ as the union of $c$
(for some $c \geqslant 1$) connected submanifolds of codimension one
(contained in $M$). We fix, once and for all, an enumeration
\[ \{ \Sigma_i \}_{i \in \llbracket 1,  c\rrbracket} \]
of such manifolds. For each subset $I \subset \llbracket 1, c \rrbracket$ we
define the {\tmem{stratum}} $\Sigma_I \subset \Sigma \subset M$ as the subset
given by intersection
\[ \Sigma_I = \bigcap_{i \in I} \Sigma_i \]
As above, note that all strata are indexed by subsets $I$ of cardinality
$\leqslant n$.

At each point $p \in \Sigma_I$ there is a local chart $\left( V, \left( x,
\eps \right) \right)$, with the coordinates $x = (x_I, x')$, labeled $x_I =
(x_i : i \in I)$ and $x' \in \cR^{n -\#I}$, such that we can write
\[ M = \left\{ (x,\eps) \mid \eps = 0 \right\} \quad \tmop{and} \quad \Sigma \cap V
   = \Sigma^{(I)}  := \left\{ \eps = 0, \prod_{i \in I} x_i = 0 \right\}  \]
We will refer to each one of these charts as {\tmem{a local adapted chart}}
(for $(N, M, \Sigma)$) {\tmem{centered}} at $\Sigma_I$.
Let us denote by $(N_I,M_I,\Sigma^{(I)})$ the restriction of the $(N,M,\Sigma)$ to this local coordinate chart.

We consider the map $\phi:\tilde N \rightarrow N_I$ given by the blowing-up of with center 
$$C_I = \{x_i = 0 \mid i \in I\}\cap \{\eps = 0\}$$
We will denote respectively by $\tilde M,\tilde \Sigma \subset \tilde N$ the subsets given by the {\em strict transforms} of $(M_I,\Sigma^{(I)})$ under such blowing-up. More precisely, 
$$
\tilde M = \overline{\phi^{-1}(M_I\setminus C_I)},\quad \text{and} \quad \tilde \Sigma = \overline{\phi^{-1}(\Sigma^{(I)} \setminus C_I)}
$$
where $\overline S$ denotes the closure of a set $S$. Using the local expressions, we can prove that the triple $(\tilde N,\tilde M,\tilde \Sigma)$ defines a new piecewise smooth regularized space. Therefore, the blowing-up can be seen as a map
$$
\varphi: (\tilde N,\tilde M,\tilde \Sigma) \rightarrow (N_I,M_I,\Sigma^{(I)})
$$
and we will say that $(\tilde N,\tilde M,\tilde \Sigma)$ is the (result of) {\em local blowing-up of $(N_I,M_I,\Sigma^{(I)})$}, with center $C_I$. 

In the next two subsections we will study the action of such blowing-up map on a regularized function $f^\reg \in C^{\infty}(N_I,M_I,
  \Sigma^{(I)})$, using the directional charts.  

\subsubsection{Phase-directional blowing-up}\label{subsect-phase-blowing-up}
For each index $i_1 \in I$, we will denote by $B_{(I, i_1)}$ the
{\tmem{$i_1^{\tmop{th}}$-directional blowing-up map}}, given by
\[
\begin{cases}
     x_i = z_{i_1} z_i, & i \in I \setminus \{ i_1 \},\\[2pt]
     x_{i_1} = \, z_{i_1},& z_{i_1} \in \cR_{\ge 0}\\[2pt] 
     \eps = \, z_{i_1}  \rho,& \rho \in \cR_{\ge 0}\\[2pt]
     x' = z'.
\end{cases}
\]

\begin{figure}[htb]
\centering
\includegraphics[width=9cm]{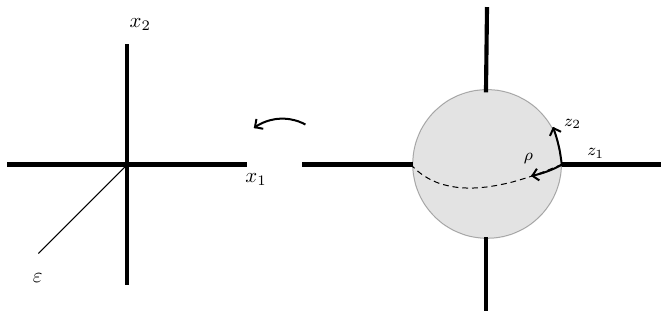}
\end{figure}

Here, we consider $(z,\rho)$ as local coordinates of a new piecewise smooth regularized space $(\tilde N,\tilde M,\tilde \Sigma)$, where 
$\tilde M = \{\rho = 0\}$ is the initial manifold and the discontinuity locus $\tilde \Sigma$ can be locally written as
$$\Sigma_{I \setminus \{ i_1 \}} 
= \left\{ \varepsilon = 0,\; \prod_{i \in I \setminus \{ i_1 \}} z_i = 0 \right\}.
$$
In this context, the following result is the analog of Lemma \ref{Lemma-composedB1} when $f_I$ is replaced by $f_I^\reg$:

\begin{Lemma}\label{lemma-+directional}
Suppose that $f^\reg \in C^{\infty}(N_I,M_I,
  \Sigma^{(I)}$ is the regularization of a piecewise smooth function $f\in C^{\infty}(M_I,
  \Sigma^{(I)})$. Then, in the coordinates of the $i_1^{th}$-directional blowing-up chart, there exists a neighborhood $W \subset \tilde{N}$ of the initial manifold $\tilde M$ (depending only on the support of the mollifier $m$) 
  such that the composed function $f_I^\reg \circ B_{(I, i_1)}$ is an element of the space
  of $C^{\infty} (W, \tilde{M},\Sigma_{I \setminus \{ i_1 \}})$. Moreover, we have the equality
$$
\vf_I^\reg \circ B_{(I,i_1)}= (\vf_{I\setminus\{i_1\}})^\reg \circ B_{(I,i_1)},
$$
when restricted to $W$.
\end{Lemma}
\begin{proof}
In order to simplify the notation, let us assume that $I=\{1,\ldots,c\}$ for some $1 \le c \le n$ and that $i_1 = 1$. The proof is clearly analogous in the other directional charts. 

Under this assumption, the original function $f=f_I$ has the local presentation given by (\ref{loc-presentation}), and we can write the regularized function $f^\reg$ as
\begin{equation}\label{freg-sum1}
f^{\reg}(x,\eps) = 
\int_{\cR^{n-c}} 
  \left( \sum_{\underline{s}}
    \int_{\{\, \underline{s}\, (x_I - t_I\eps ) > 0 \,\}} 
      f_s\!\left(x_I - t_I\eps,\, x'- t'\eps \right)\,
      m(t_I,t') \, d t_I 
  \right) d t' 
\end{equation}
where the notation $\{\underline{s}\, (x_I - t_I \eps)>0\}$ stands for the subset of $\cR^{c}$ defined by the conditions $s_1(x_1 - t_1 \eps)>0,\ldots,s_c(x_c - t_c \eps) > 0$. We recall that the convention established in Remark~\ref{remark-dependence-base-point} is in force for all subsequent computations.

Observe that the directional blowing-up $B_{(I,i_1)}$ has the expression
\begin{equation}\label{x1-directional-blowup}
x_1 = z_1, \quad x_J = z_1 z_J,\quad \eps = z_1 \rho, \quad x' = z'
\end{equation}
where we write ${x}_J = (x_2,\ldots,x_c)$ and ${z}_J = (z_2,\ldots,z_c)$. Hence, in the coordinates $(z,\rho)$, the innermost integral in (\ref{freg-sum1}) can be decomposed as
$$
\int_{\{\, \underline{\tilde s} ({z}_J - t_J\rho ) > 0 \,\}}
     \left( \int_{\cR} \Big(
      \mathbf{1}_{(1- t_1 \rho) > 0}f_{(+,\underline{\tilde s})}(\star) + \mathbf{1}_{(1-t_1 \rho) < 0}
      f_{(-,\underline{\tilde s})}(\star) \Big)
      m(t_1,{t}_J,t') \, d t_1 \right) d {t}_J
$$
where we write $s=(\pm,\underline{\tilde s})$ with $\underline{\tilde s} = (s_2,\ldots,s_c)$,  $t_I = (t_1,t_J) \in \cR \times \cR^{c-1}$ and the $\star$ symbol stands for 
the expression $(z_1(1-t_1\rho ),z_1({z}_J-{t}_J\rho),\, z' - t' z_1\rho)$. 

We recall now that the mollifier $m$ has compact support, say 
\(\operatorname{supp}(m) \subset B(0,r)\) for some radius $r>0$. 
Hence, the function 
\[
t_1 \longmapsto \mathbf{1}_{(1 - t_1 \rho) < 0}\cdot m(t_1,\cdot)
\]  
vanishes identically whenever \(\tfrac{1}{\rho} > r\) (see Figure below).
\begin{figure}[htb]
\centering
{\includegraphics[width=5.61907221566312cm,height=4.21707333070969cm]{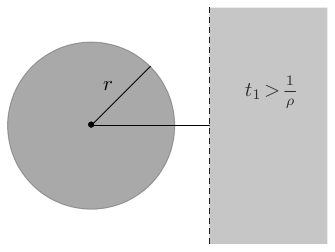}}
\end{figure}

In other words, restricted to the open set $W=\{\rho < 1/r\}$ (which is an open neighborhood of $\tilde M$ in the domain of this chart), the above integral reduces to  
\begin{equation}\label{reduced-integral}
\int_{\{\, \underline{\tilde s} ({z}_J - t_J\rho ) > 0 \,\}}
     \left( \int_{\cR}
      f_{(+,\underline{\tilde s})}(\star)\,
      m(t_1,{t}_J,t') \, d t_1
   \right) d \underline{t}_J.
\end{equation}
It remains to observe that the above expression for the regularization involves only those components $f_{\underline{s}}$ of $f$ whose first entry is +, that is, the terms of the form $f_{(+,\underline{\tilde{s}})}$. Equivalently, the same regularization integral would be obtained by replacing 
$f_I$ with $f_{I \setminus \{i_1\}}$. Consequently, we have
\[
f^{\reg}_I \circ B_{(I,i_1)} \;=\; (f_{I \setminus \{i_1\}})^\reg \circ B_{(I,i_1)},
\qquad \text{upon restriction to } W.
\]
This completes the proof.
\end{proof}
We will now iterate the above result by considering a composed directional blowing-up $B_{(I,\underline{i})}$, as defined in  (\ref{composed-blowupdef}). 
\begin{Remark}
\label{Remark-exprB}
Note that, in the present setting, the composed phase directional blowing-up $B_{(I,
  \underline{i})}$ assumes the following form (to be compared with Remark \ref{Remark-exprB-nonreg})
 \[
\begin{cases}
x_{i_1} = z_{i_1}, \\[4pt]
x_{i_2} = z_{i_1} z_{i_2}, \\[4pt]
\vdots \\[4pt]
x_{i_{k-1}} = z_{i_1} \cdots z_{i_{k-1}}, \\[4pt]
x_{i_k} = z_{i_1} \cdots z_{i_k}, \\[4pt]
x_i = (z_{i_1} \cdots z_{i_k})\, z_i, \quad i \in I \setminus \underline{i}, \\[4pt]
\eps  =  (z_{i_1} \cdots z_{i_k}) \rho & \\
x' = z', \hspace{2.3cm} x', z' \in \cR^{n - \#I}.
\end{cases}
\]
where $z_{i_1},\ldots,z_{i_k}$ are variables in $\cR_{\ge 0}$.
\end{Remark}
As in the previous Lemma, we define the {\em initial manifold} in the blown-up space as the manifold locally given, in the coordinates $(z,\rho)$, by $\tilde M = \{\rho = 0\}$.

By an inductive application of the previous result and Lemma \ref{Lemma-composedB}, we obtain the following
\begin{Proposition}
  \label{prop-initial-fiber}
Let $f$ and $f^\reg$ be as in the statement of the Lemma \ref{lemma-+directional}.  Then, there exists a neighborhood $W\subset \tilde N$ of the initial manifold
  $\tilde{M}$ such that the composed function $f_I^\reg \circ B_{(I, \underline{i})}$ is an element of the space
  of $C^{\infty} (W, \tilde{M},\Sigma_{I \setminus \{ \underline{i} \}})$. Moreover, we have the equality
  \[ f^{\reg}_I \circ B_{(I, \underline{i})} = ( f _{I \setminus
     \underline{i}})^\reg \circ B_{(I, \underline{i})} 
  \]
  in restriction to $W$. 
 \end{Proposition}
\subsubsection{Family-directional blowing-up}\label{subsect-blowing-up-regular}
We will denote by $B_F$ the
{family-directional blowing-up map}, given by
\[
\begin{cases}
     x_i = \rho\,  z_i, & i \in I ,\\[2pt]
     \eps =  \rho,& \rho \in \cR_{\ge 0}\\[2pt]
     x' = z'.
\end{cases}
\]

\begin{figure}[htb]
\centering
\includegraphics[width=9cm]{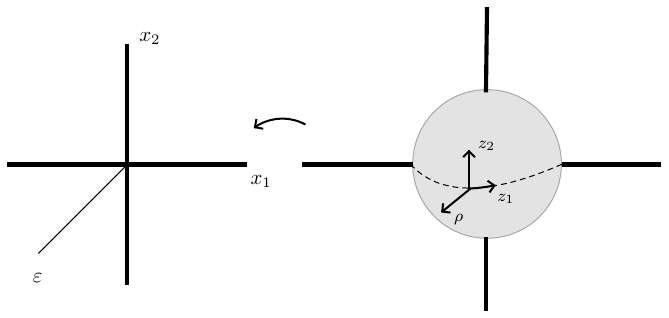}
\end{figure}

Therefore, given a piecewise smooth function $f\in C^\infty(M,\Sigma)$ and its regularization $f^\reg$ defined by (\ref{freg-sum1}), the composed function 
$f^\reg \circ B_F$ has the form
\begin{equation}\label{integral-conv-reg}
\int_{\cR^{n - \# I}} \left( \int_{\cR^I} \vf(\rho (x_I-t_I),x'-\rho t')\ m(t_I,t')\; dt_I \right) dt'
\end{equation}
where $t_I$ and $t'$ are variables in $\cR^I$ and $\cR^{n-\#I}$ respectively. We now prove the following result: 
\begin{Lemma}\label{Lemma-partial-conv}
The function $f^\reg\circ B_F$ is smooth.
\end{Lemma}
\begin{proof}
We write $(N_I,M_I,\Sigma^{(I)})$ as $(N,M,\Sigma)$ to simplify the notation, and let $W=B_F^{-1}(N)$ denote the domain of the family chart.  

We consider the auxiliary function $\mathbf{g}(x_I,x',\eps) = \vf(\eps x_I,x')$, given by
$$
\mathbf{g}(x_I,x',\eps) = \sum_{\underline{s}} \mathbf{1}_{\underline{s} x_I > 0} f_{\underline{s}}(\eps x_I,x')
$$
which is a piecewise smooth function on $N$ with discontinuity locus given by $\{\prod_{i\in I}x_i = 0,\eps = 0\}$. Note that $\mathbf{g}$ depends smoothly on the variable $x'$. 

Now, for each fixed $t' \in \cR^{n-\# I}$, we define the function  
$$
G(x_I,x',\eps,t') =  \int_{\cR^{c}} \mathbf{g}(x_I-t_I,x',\eps)\ m(t_I,t')\; dt_I 
$$
Or, equivalently, $G$ is given by the convolution product  $m^{t'} \ast \mathbf{g}$, where $m^{t'}\in C^\infty_0(\cR^{I})$ is the smooth function (with compact support) defined by $t_I \mapsto m(t_I,t')$. From the classical properties of the convolution, it follows that $G $ is a globally smooth function on $N$. Finally, it suffices to observe that if we further integrate with respect to $t'$, and we define the smooth function
\begin{equation}\label{integral-H}
H(x_I,x',\eps) = \int_{\cR^{n-\# I}} G(x_I,x'-\eps t',\eps,t')\; dt',
\end{equation}
then, $f^\reg \circ B_F$ is obtained from $H$ simply by performing the change of variables $x_I \mapsto z_I, x'\mapsto z',\eps \mapsto \rho$. This shows that $f^\reg \circ B_F$ is smooth.
\end{proof}
\begin{Remark}
Note that the mollifier $m$ has compact support. Therefore, for each fixed $(x_I,x',\eps)$, 
the function $t' \mapsto G(x_I,x',\eps,t')$ has compact support and the integral (\ref{integral-H}) is well-defined. \\
\end{Remark}
\subsection{The Smoothing procedure for functions}\label{subsect-smoothingprocedure}
We can now state the main result of this section. Let $(M,\Sigma)$ be a piecewise smooth space such that $M$ is an open subset of $\cR^n$. 
Let $(N,M,\Sigma)$ be the associated regularized piecewise smooth space and let 
$$
\reg_m \colon C^\infty(M,\Sigma) \longrightarrow C^{\infty,\rg}(N,M,\Sigma)
$$
be the regularization operator associated to a mollifier $m\in\Mol(\cR^n)$, as defined in  subsection \ref{subsection-regularizafcn}.
\begin{Theorem}\label{theorem-smoothing-rpss}
There exists a finite sequence of blowing-ups 
$$
(N, M, \Sigma) = (N_0, M_0, \Sigma_0) \longleftarrow  \cdots \longleftarrow (N_r, M_r, \Sigma_r)
$$
such that, for each $k=0,\ldots,r$, the following property holds: 
\begin{itemize}
\item The blow-up center $C_k$ is contained in the discontinuity locus $\Sigma_k$,
\item If $\Phi_k = \varphi_1 \circ \cdots \circ \varphi_k$ denotes the composition of the first $k$ blowings-up then the composition
$$
\Phi_k^*\circ \reg_m  
$$
defines a linear operator from $C^\infty(M,\Sigma)$ to  $C^{\infty}(N_k,M_k,\Sigma_k)$, for every choice of mollifier $m$. Here $\varphi^*$ denotes the pull-back operator
$\varphi^* f = f\circ \varphi$. 
\item The discontinuity locus $\Sigma_r$ is empty. 
\end{itemize}
\end{Theorem}
In particular, we conclude that for any piecewise smooth function $f \in C^\infty(M,\Sigma)$ and any mollifier $m$, the composed function 
$$
\vf^\reg \circ \Phi_r
$$
is a {\em globally smooth function} on the manifold $N_r$. 
The Figure \ref{figure-sequence-d2} illustrates the procedure in the case where $n=2$.  The gray shaded region represents the initial manifold.
\begin{figure}[htb]
\centering
 \includegraphics[width=10.2514675980585cm,height=4.62522956841139cm]{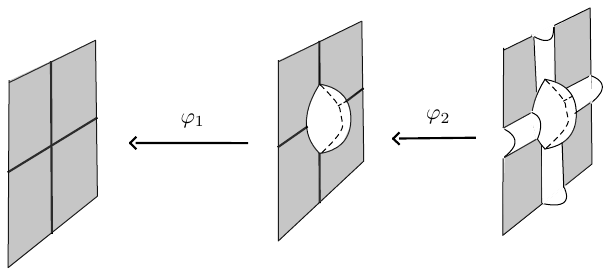}
 \caption{Sequence of blowing-up leading to a smoothing in dim 2.}\label{figure-sequence-d2}
\end{figure}

The proof of the Theorem is done by induction. So, assume by induction that the resolution sequence has already been defined up
to step $k$,
\[ (N, M, \Sigma) = (N_0, M_0, \Sigma_0) \longleftarrow (N_1, M_1, \Sigma_1)
   \longleftarrow \cdots \longleftarrow (N_k, M_k, \Sigma_k) \]
for some $0 \leqslant k \leqslant n$. \ Let $(\tilde{N}, \tilde{M},
\tilde{\Sigma}) = (N_k, M_k, \Sigma_k)$ denote the
$k^{\tmop{th}}$ piecewise piecewise smooth space in such blowing-up sequence and define
\[ \varphi : = \varphi_1 \circ \cdots \circ \varphi_k : (\tilde{N}, \tilde{M},
   \tilde{\Sigma}) \rightarrow (N, M, \Sigma) \]
to be the $k^{\tmop{th}}$-composed map. We recall that $\tilde{M},
\tilde{\Sigma}$ denote respectively the strict transforms of $M, \Sigma$ under
$\varphi$. If $k = 0$, we simply set $\varphi = \tmop{id}$.

We now state the following induction hypothesis.

\begin{Hypothesis}
  The regularized piecewise smooth space $(\tilde{N}, \tilde{M}, \tilde{\Sigma})$ is equipped with an indexing $\{ \tilde{\Sigma}_i \}_{i \in \llbracket
  1, c \rrbracket}$ of the irreducible components of $\tilde{\Sigma}$ satisfying the following
  conditions:
  \begin{enumeratealpha}
    \item Let $f^\reg \in C^{\infty} (N,M,\Sigma)$ be  the regularization of an arbitrary function in $C^\infty(M,\Sigma)$. Then its pull-back $\tilde{f}
    = f^\reg \circ \varphi$ under $\varphi$ defines a piecewise smooth function in
    $C^{\infty} (\tilde{N}, \tilde{M},\tilde{\Sigma})$. \
    
    \item For each subset $I \subset \llbracket 1, c \rrbracket$ of
    cardinality $\#I >  n - k$, the stratum
    \[ \tilde{\Sigma}_I = \bigcap_{i \in I} \tilde{\Sigma}_i \]
    is empty.
    
    \item For each subset $I \subset \llbracket 1, c \rrbracket$ of
    cardinality $\#I \le n - k$ there exists an open covering of
    $\tilde{\Sigma}_I$ by local adapted charts $(U, (z, \rho))$ such that to each one these charts is
    associated the following data:
    \begin{itemize}
      \item A subset of indices $I_0 \subset \llbracket 1, c \rrbracket$,
      
      \item A (possibly empty) ordered list $\underline{i} = (i_1, \ldots,
      i_l) \subset I_0 .$
      
      \item An adapted chart $\left( V, \left( x, \eps \right) \right)$ for
      $(N, \Sigma)$ centered at $\Sigma_{I_0}$.
    \end{itemize}
    Satisfying the following properties:
    \begin{enumeratenumeric}
      \item $I = I_0 \setminus \underline{i}$
      
      \item The image $\varphi (U)$ is contained in the domain $V$.
      
      \item In the coordinates $\left( x, \eps \right), (z, \rho)$, the map
      $\varphi$ assumes the form
      \begin{equation}
        \left( x, \eps \right) = B_{I_0, \underline{i}}  (z, \rho)
        \label{local-expr-blow-up}
      \end{equation}
      i.e. it is given by a sequence of phase-directional blowing-ups. By convention, $B_{I, \underline{i}} = \tmop{id}$ if $\underline{i} =
      \emptyset$.
    \end{enumeratenumeric}
  \end{enumeratealpha}
\end{Hypothesis}

If these conditions hold, we say for short that $[(\tilde{N},\tilde{M},
\tilde{\Sigma}), \varphi]$ {\tmem{satisfies hypothesis}} $(\tmop{H}_k)$.

Assuming that this is true, it follows from items (a) and (b) that $\tilde{\Sigma}$
has no strata of dimension $\leqslant k - 1$ and the union
\[ \tilde{C} = \bigcup_{I : \#I = n - k} \Sigma_I \]
is a smooth closed submanifold of $\tilde{N}$. Based on this, we define the
$(k + 1)^{\tmop{th}}$ {\tmem{resolution step}} as the map
\begin{equation}
  \varphi_{k + 1} : (N_{k + 1},M_{k+1}, \Sigma_{k + 1}) \longrightarrow (N_k,M_k,
  \Sigma_k) = (\tilde{N}, \tilde{M},\tilde{\Sigma}) \label{k-blowing-up}
\end{equation}
given by the blowing-up of $(\tilde{N}, \tilde{\Sigma})$ with center
$\tilde{C}$. We recall that $M_{k+1},\Sigma_{k + 1}$ are defined as the strict
transforms of $M_k,\Sigma_k$ under $\varphi_{k + 1}$.

Using this hypothesis, the Theorem \ref{theorem-smoothing-rpss} will be an immediate consequence of the following result:
\begin{Proposition}
  The pair $[(\nobracket N_{k + 1}, M_{k+1},\Sigma_{k + 1}), \varphi \circ \varphi_{k + 1} \nobracket]$ satisfies hypothesis $(\tmop{H}_{k + 1})$. 
\end{Proposition}

\begin{proof}
  To simplify the notation, we write $(N', M',\Sigma') = (N_{k + 1}, M_{k+1},\Sigma_{k +
  1})$ and denote the blowing-up (\ref{k-blowing-up}) by $\psi : (N', M',\Sigma')
  \rightarrow (\tilde{N}, \tilde{M},\tilde{\Sigma}) .$
  
  The proof consists in studying the local expression of such blowing-up
  through the local adapted charts $(U, (z, \rho))$ given by item $(c)$ of
  $(\tmop{H}_k)$.
  
  Suppose firstly that the domain of $(U, (z, \rho))$ does not intersect the
  blowing-up center $\tilde{C}$. Then, $\psi$ defines a diffeomorphism in
  restriction to $U$, and the pair
  \[ (\psi^{- 1} (U), (z, \rho) \circ \psi) \]
  defines a local adapted chart for $(N', M',\Sigma')$ satisfying all the needed
  conditions.
  
  Suppose now that the chart $(U, (z, \rho))$ intersects $\tilde{C}$.
  Then the chart is centered on a stratum $\tilde{\Sigma}_I$ given by some
  index set $I \subset \llbracket 1, c \rrbracket$ of cardinality $\#I = n -
  k$. Expressed in this chart, the blowing-up is covered by the
  family-directional chart $B_F$ and by the union of the phase-directional
  charts $B_{I, j}$, where the index $j$ varies in $I$, as described in subsections \ref{subsect-phase-blowing-up} and \ref{subsect-blowing-up-regular}.
  
  Let us consider respectively the expressions of the composed map $B_{I_0,
  \underline{i}} \circ B_F$ and $B_{I_0, \underline{i}} \circ B_{I, j}$, where
  $B_{I_0, \underline{i}}$ is defined by (\ref{local-expr-blow-up}).
  
  Based on Remark \ref{Remark-exprB}, if we write the coordinates of the
  family-directional chart as $(w, \eta)$, the map $B_{I_0, \underline{i}}
  \circ B_F$ assumes the form
  \begin{equation}
    \left\{\begin{array}{llll}
      x_i & = & (w_{i_1} \cdots w_{i_l} \cdot \eta) w_i & , \quad i \in I\\
      \eps & = & (w_{i_1} \cdots w_{i_l} \cdot \eta) & \\
      x_{i_1} & = & w_{i_1} & \\
      x_{i_2} & = & w_{i_1} w_{i_2} & \\
      \vdots &  &  & \\
      x_{i_l} & = & w_{i_1} \cdots w_{i_l} & \\
      x' & = & w' & , \quad x', w' \in \cR^{n -\#I_0}
    \end{array}\right. \label{chartBF}
  \end{equation}
  Similarly, if we write the coordinates of the $j^{\tmop{th}}$-directional
  chart $B_{I, j}$ as $(w, \eta)$, then map $B_{I_0, \underline{i}} \circ
  B_{I, j}$ takes the form
  \[ \left\{\begin{array}{llll}
       x_j & = & (w_{i_1} \cdots w_{i_l} \cdot w_j)  & \\
       x_i & = & (w_{i_1} \cdots w_{i_l} \cdot w_j) w_i & , \quad i \in I
       \setminus \{ j \}\\
       \eps & = & (w_{i_1} \cdots w_{i_l} \cdot w_j) \eta & \\
       x_{i_1} & = & w_{i_1} & \\
       x_{i_2} & = & w_{i_1} w_{i_2} & \\
       \vdots &  &  & \\
       x_{i_l} & = & w_{i_1} \cdots w_{i_l} & \\
       x' & = & w' & , \quad x', w' \in \cR^{n -\#I_0}
     \end{array}\right. \]
 We now prove that $f^{\reg} \circ B_{I_0, \underline{i}} \circ B_F$
  defines a smooth map. 
Firstly,  it follows from Proposition \ref{prop-initial-fiber} that, since $I=I_0 \setminus \underline{i}$, we can write
 \begin{equation}\label{freg-reg}
  f_{I_0}^\reg \circ B_{I_0,\underline{i}} = f_I^\reg \circ B_{I_0,\underline{i}}
 \end{equation}
  Therefore, the regularization integral can be locally decomposed as
   \[ f_I^{\reg} = \int_{\cR^{n -\#I}} \int_{\cR^I} f_{I}
     \left( x_{I} - \eps t_{I}, x'' - \eps t'' \right) m (t_{I}, t'')\, d t_{I} \,d t'' \]
  where we denote by $x''$ the
  variables $(x_{i_1}, \ldots, x_{i_l}, x')$. Note that $f_I$ depends smoothly on $x''$. 
  
  By right-composing with the
  map $B_{(I_0, \underline{i})} \circ B_F$, we obtain
  \[ f_I^{\reg} \circ B_{(I_0, \underline{i})} \circ B_F =
     \int_{\cR^{n -\#I}} {\int_{\cR^I}}  f_I \left( \mon \cdot
     (w_I - t_I), b (w'') - \mon t'' \right) m (t_I, t'') d t_I d t'' \]
  where $\mon$ denotes the monomial $w_{i_1} \cdots w_{i_l} \cdot \eta$ and we
  denote by $w''$ collection of variables $(w_{i_1}, \ldots, w_{i_l}, w')$,
  and by $x'' = b (w'')$ the map defined by the third through the last lines
  of the system of equations in (\ref{chartBF}). 
  
 This last integral has precisely the form of the integral (\ref{integral-conv-reg}), studied in Lemma \ref{Lemma-partial-conv}, up to substituting the positive variable $\rho$ by the positive monomial $\mon$ and replacing the {\em smooth variable} $x'$ by the smooth function $b(w'')$.  It follows
  from that result that $f_I^\reg \circ B_{(I_0, \underline{i})} \circ B_F$ is a smooth
  function.
  
 We now study the phase directional blowing-ups.  
 Note that, by defining the new list $\underline{i}' = \underline{i} \cup \{ j \}$, we can write
  $B_{I_0, \underline{i}} \circ B_{I, j} = B_{I_0, \underline{i'}}$. Therefore, applying again Proposition \ref{prop-initial-fiber}, we obtain
  the equality
  \[ f_{I_0}^{\reg} \circ \left( B_{I_0, \underline{i}} \circ B_{I, j}
     \right) = ( f_{I_0 \setminus \underline{i}'} )^{\reg}\circ B_{I_0,
     \underline{i'}}  \]
  which shows that, when restricted to the domain $W = (B_{I, j})^{- 1} (V)$,
  the map $f^{\reg} \circ \varphi \circ \psi$ lies in the space $C^{\infty} \left(
  N', M',\Sigma'_{I_0 \setminus \underline{i}'} \right)$. Therefore, the pair
  \[ (W, (w, \eta)) \]
  defines a new adapted chart satisfying all conditions stated in item (c) of
  $(\tmop{H}_{k + 1})$.  
 We remark that the argument used in the proof of Lemma \ref{lemma-+directional} assumes that the coordinates are restricted to the open set domain $W=\{\rho < 1/r\}$, which in principle depends on the support of the mollifier $m$. However, in the phase chart, the complement of $W$ lies in the domain of the family chart, where Lemma \ref{Lemma-partial-conv} applies. This concludes the proof.
\end{proof}
\section{Piecewise-smooth vector fields: regularization and smoothing}\label{s6}
We now consider the problem of regularization and smoothing of piecewise smooth vector fields. Our goal is to establish a version of the Smoothing Theorem \ref{theorem-smoothing-rpss}  that applies to vector fields.  
\subsection{Regularization of piecewise-smooth vector fields}\label{subsect-reg-vf}

Given a piecewise smooth space $(M,\Sigma)$,  denote by $\mathfrak{X }^{\infty} (M, \Sigma)$ the space of vector fields
which can be locally written as
$$
X = \sum_{k = 1}^n \vf_k \ \frac{\partial}{\partial x_k}
$$
where the components $\vf_k$ are elements of $C^{\infty} (M, \Sigma)$. The topology in
$\mathfrak{X }^{\infty} (M, \Sigma)$ is the componentwise topology induced
from $C^{\infty} (M, \Sigma)$. Note that, for $\Sigma = \emptyset$,
$\mathfrak{X }^{\infty} (M, \Sigma)$ is equal to the space $\mathfrak{X }^{\infty} (M)$ of globally smooth vector fields on $M$.

Let us assume that $M$ is an open subset of $\cR^n$. Then, the regularization by convolution defined in subsection \ref{subsection-regularizafcn} can be easily adapted to the case of piecewise-smooth vector fields in $\mathfrak{X }^{\infty} (M,\Sigma)$: Given a piecewise-smooth vector field $X = \sum_{k = 1}^n \vf_k (x)
\frac{\partial}{\partial x_k}$ in $\mathfrak{X }^{\infty} (M, \Sigma)$ and a mollifier $m$, we consider the family of vector fields with parameter $\eps >0$ defined by
\begin{equation}\label{reg-conv-X}
m_\eps \ast X := \sum_{k = 1}^n \big(\, m_\eps \ast \vf_k\, \big) 
   \frac{\partial}{\partial x_k} 
\end{equation}
where each component $m_\eps \ast \vf_k$ is given by the expression (\ref{expr-f_k}).
As in definition \ref{def-regfconv}, we extend $m_\eps \ast X$ to a one-parameter family with parameter $\eps \in (\cR_{\ge 0},0)$ as follows:
\begin{Definition}\label{def-regularized-family}
The {\em regularization by
convolution of $X$} (with mollifier $m$) is the one-parameter family of vector fields $X^\reg$ in $M$, with parameter $\eps \in  (\cR_{\ge 0},0)$, defined by 
$X^\reg |_{\eps =0} =X$ and $X^\reg |_{\eps > 0} = m_\eps \ast X$. 
\end{Definition}
The following result is an immediate consequence of Theorem \ref{rem-conv}:
\begin{Proposition}
The family $X^\reg$ satisfies the following properties:
\begin{enumerate}
\item $X^\reg$ is smooth for $\eps >0$.
\item As $\eps \rightarrow 0$, $X^\reg$ converges to $X$ (with respect to the $C^\infty$-topology) on each compact set $K \subset M\setminus\Sigma$.
\end{enumerate}
\end{Proposition}
As in the subsection \ref{subsection-regularizafcn} we observe that we can interpret $X^\reg$ as piecewise smooth vector field in the space
$\mathfrak{X}^{\infty} (N, \Sigma)$, where 
$N=M\times (\cR_{\ge 0},0)$ and the variety $\Sigma$ is embedded in $N$ as $\Sigma \times \{0\}$. 

Notice that $X^\reg$ has the additional property of being tangent to the fibers of the fibration $ \{\eps = cte\}$ defined by the level sets of the linear projection 
$$\pi: N \rightarrow \cR_{\ge 0}$$
onto the $\eps$-coordinate.  
Since we will need to keep track of this extra structure, we introduce the following notion. A {\em fibered piecewise smooth regularized space} is a $4$-uple
$(N,M,\Sigma,\pi)$, where
\begin{enumeratenumeric}
  \item $(N,M,\Sigma)$ is a piecewise smooth regularized space (see definition at subsection \ref{subsect-piecewisesmoothfcs})
  \item $\pi:N\rightarrow (\cR_{\ge 0},0)$ is smooth map such that $F_0=\pi^{-1}(0)$ coincides with $\partial N$ and $\pi$ is a submersion restricted to $N\setminus\partial N$.
 \end{enumeratenumeric}
In particular, notice that the {\em zero-fiber} $F_0$ contains the {\em initial manifold} $M$ and that the fiber $F_\eps = \pi^{-1}(\eps)$ is a smooth submanifold of $N$ for any $\eps > 0$. 
\begin{figure}[htb]
\centering
\includegraphics[width=4.96444969172242cm,height=6.73237242555424cm]{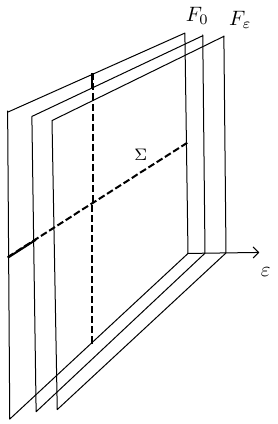}
\end{figure}

We denote by $\mathfrak{X}^{\infty} (N, M,\Sigma,\pi)\subset \mathfrak{X}^{\infty}(N,\Sigma)$ the subspace of piecewise smooth vector fields whose components are functions lying in the space $C^\infty(N,M,\Sigma)$ (as defined in subsection \ref{subsection-regularizafcn}), and which are moreover everywhere tangent to the fibers of the fibration $\{d \pi = 0\}$. 

Furthermore, a {\em diffeomorphism} between two fibered piecewise smooth regularized spaces $(N, M,\Sigma,\pi)$ and $(\tilde N, \tilde{M}, \tilde{\Sigma},\tilde{\pi})$ is defined by a diffeomorphism $\psi:N \rightarrow \tilde N$ such that $\tilde{M} = \psi(M)$, $\tilde{\Sigma} = \psi(\Sigma)$ and
$\tilde \pi = \pi \circ \psi^{-1}$. We will denote such diffeomorphism by
$$
\psi: (N, M,\Sigma,\pi) \longrightarrow (\tilde N, \tilde{M}, \tilde{\Sigma},\tilde{\pi}) 
$$
To summarize, given a piecewise smooth space $(M,\Sigma)$ with $M\subset \cR^n$ an open subset, and given a mollifier $m$, the regularization by convolution defines a linear map
$$
\reg_m: \mathfrak{X}^{\infty} (M, \Sigma) \longmapsto \mathfrak{X}^{\infty} (N, M,\Sigma,\pi)
$$
where $(N, M,\Sigma,\pi)$ is a fibered piecewise smooth regularized space with domain $N=M\times(\cR_{\ge 0},0)$ and $\pi:N \rightarrow (\cR_{\ge 0},0)$ is the linear projection. 
\subsection{Piecewise-smooth oriented 1-dimensional foliations}\label{subsect-piecewisesmoothfoliations}
Although our primary objects of study are piecewise-smooth vector fields, we will see that
successive blow-up operations naturally lead us to consider the broader class of 
piecewise-smooth foliations, which have also been studied in \cite{PS}. We now recall the relevant definitions.

Let $(N,M,\Sigma,\pi)$ be a fibered piecewise smooth regularized space as defined in the previous subsection. A \emph{piecewise-smooth 1-dimensional oriented foliation} (or, for short, a {\em piecewise 1-foliation}) on $(N,M,\Sigma,\pi)$ is a collection of pairs
$$
\vF = \{(U_i,X_i)\}_{i \in I}
$$
such that:
\begin{enumeratenumeric}
\item  $\{U_i\}$ is an open covering of $N$. 
\item For each $i\in I$, $X_i$ is a vector field in $\mathfrak{X}^{\infty} (U_i, M \cap U_i,\Sigma\cap U_i,\pi|_{U_i})$.
\item For each pair $i,j \in I$, we have the equality
\begin{equation}\label{transition-function}
X_i = \phi_{ij}\, X_j
\end{equation} 
for some strictly positive smooth function $\phi_{ij}$ defined on $U_i\cap U_j$. 
\end{enumeratenumeric}
We will denote the set of all such foliations by $\Fol^\infty(N,M,\Sigma,\pi)$.
\begin{Remark}
Note that we require the transition functions $\phi_{ij}$ to belong to $C^\infty(U_i \cap U_j)$, rather than to $C^\infty(U_i \cap U_j, \Sigma)$. 
\end{Remark}
A pair $(V,Y)$ will be called a \emph{local generator }of the foliation $\vF$ if the augmented collection
$$
\{(U_i,X_i)\}_{i \in I} \; \cup \; \{ (V,Y)\}
$$ 
also satisfies conditions 1.-3.~of the above definition.  From now on, we will 
suppose that the collection $\vF$ is {\em saturated}, meaning
that it contains all such local generators.
\begin{Example}
A piecewise smooth vector field $X \in \mathfrak{X}^{\infty} (N, M,\Sigma,\pi)$ defines a piecewise 1-foliation 
$\vF = \vF_X$ which is globally generated by $X$. 
\end{Example}
Let $(N, M,\Sigma,\pi)$ and $(\tilde N, \tilde{M}, \tilde{\Sigma},\tilde{\pi})$ be two fibered piecewise smooth regularized spaces. 
We will say that two piecewise 1-foliations $\vF \in \mathfrak{F}^{\infty} (N, M,\Sigma,\pi)$ and $\tilde \vF \in \mathfrak{F}^{\infty}(\tilde N, \tilde{M}, \tilde{\Sigma},\tilde{\pi})$ are {\em smoothly equivalent} if there exists a diffeomorphism 
$$\psi : (N, M,\Sigma,\pi) \rightarrow (\tilde N, \tilde{M}, \tilde{\Sigma},\tilde{\pi})$$ 
defined according to the previous subsection, such that:
For each local generator $(U,X)$ of $\vF$, its push-forward under $\psi$, namely 
$$
(\tilde U,\tilde X) := \big( \psi(U), \psi_* X \big),
$$
is a local generator of $\tilde \vF$.
\subsection{Blowing up of piecewise 1-foliations}
We now define the blowing-up operation in the context of fibered piecewise smooth regularized spaces and piecewise 1-foliations.  

The {\em blowing-up} of $(N, M,\Sigma,\pi)$ (with center $C$) is a new fibered piecewise smooth regularized space $(\tilde{N},\tilde{M},\tilde{\Sigma},\tilde \pi)$ defined as follows:
\begin{enumerate}
\item The manifold $\tilde{N}$ is obtained from $N$ by the blowing-up 
$$\varphi:\tilde{N} \rightarrow N$$ 
with a center $C$ contained in the discontinuity locus $\Sigma$. 
\item $\tilde M$ and $\tilde \Sigma$ are the strict transforms of $M$ and $\Sigma$ under $\varphi$.
\item The projection map $\tilde \pi:\tilde N \rightarrow (\cR_{\ge 0},0)$ is given by $\tilde \pi = \pi \circ \varphi$. 
\end{enumerate}
We will denote the blowing-up map by 
$$\varphi: (\tilde{N},\tilde{M},\tilde{\Sigma},\tilde \pi) \longrightarrow (N, M,\Sigma,\pi)$$
Suppose now that we fix a piecewise 1-foliation $\vF \in \mathfrak{F}^{\infty} (N, M,\Sigma,\pi)$. We will say that a piecewise 1-foliation $\tilde \vF \in \mathfrak{F}^{\infty}(\tilde N, \tilde{M}, \tilde{\Sigma},\tilde{\pi})$ is a {\em blowing-up of $\vF$} if $\varphi$ induces a smooth equivalence between the foliations $\vF$ and $\tilde \vF$ outside the blowing-up locus and moreover the restriction of $\tilde \vF$ to the so-called {\em exceptional divisor} $D = \varphi^{-1}(C)$ is nontrivial.

More precisely, observe that $\varphi$ establishes a diffeomorphism between 
$$
(\tilde{N} \setminus D,\tilde M\setminus D,\tilde{\Sigma}\setminus D,\tilde \pi)\quad\text{and}\quad
(N \setminus C,M \setminus C,\Sigma \setminus C, \pi)
$$
Therefore, $\tilde \vF$ is a blowing-up of $\vF$ if the following conditions hold: 
\begin{enumerate}
\item The restriction $\tilde \vF |_{\tilde N\setminus D}$ to be smoothly equivalent to $\vF |_{N\setminus C}$ via $\varphi$, and, 
\item For each local generator $(V,Y)$ of the foliation $\tilde \vF$, if we denote by $D\cap V={r=0}$ the local equation of the exceptional divisor in a local chart then $(\varphi(V), \varphi^*(r Y))$ is a local generator of $\vF$.
\end{enumerate}
The last condition means that the local generators of $\tilde \vF$ can be defined from the pull-back of the local generators of $\vF$ upon division by $r$.  
\subsection{Smoothing Theorem for hyperplane discontinuity: link to ST-regularization}\label{subsection-st-reg-link}
Before addressing the proof of Smoothing Theorem, let us illustrate the procedure in the particular setting where $M = \cR^n$, with coordinates $(x_1,\ldots,x_n)$, and where the discontinuity locus $\Sigma = \{x_1 = 0\}$ is a coordinate hyperplane.  In this case, we will see that the regularization by convolution is closely related to the well-known {\em Sotomayor-Teixeira regularization} introduced in \cite{ST}. 

In the above setting, a vector field $X\in \mathfrak{X}^\infty(M,\Sigma)$ can be written as
\begin{equation}\label{def-X}
 X = \tmmathbf{1}_{\{ x_1 > 0 \}} X^+ + \tmmathbf{1}_{\{ x_1 < 0 \}} X^- 
\end{equation}
 where $X^{\pm} = \sum_{k = 1}^n \vf_k^{\pm} (x) \frac{\partial}{\partial x_k}$ are smooth vector fields and, as previously, $\tmmathbf{1}_S$ denotes the characteristic function of a set $S$.  

The regularized family given by definition \ref{def-regularized-family} is a piecewise-smooth vector field $X^\reg$ in the product manifold $N = \cR^n \times (\cR_{\ge 0},0)$ with discontinuity locus given by the subspace $\Sigma=\{x_1 = \eps = 0\}$. 

Let us consider the map 
$$\varphi: \tilde N \rightarrow N$$
given by the blowing-up of $N$ with center on $\Sigma$. 
\begin{figure}[htb]
\centering
  \includegraphics[width=6.69519054178145cm,height=4.17839105339105cm]{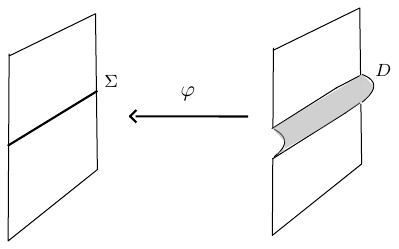}
\end{figure}
As previously discussed, $\varphi$ restricts to a diffeomorphism 
$$
\varphi|_{\tilde N\setminus D} : \tilde N \setminus D \longrightarrow N \setminus \Sigma
$$
where $D = \varphi^{-1}(\Sigma)$ is the exceptional divisor. 
As a consequence the restriction of the vector field $X^\reg$ to $\tilde N \setminus \Sigma$ can be pulled-back by $\varphi$ to a smooth vector field $\varphi^\star X^\reg$ on $\tilde N \setminus D$.

We now prove that, after multiplication by $r$, $X^\reg$ extends \emph{smoothly} to the divisor $D$, defined by $\{r=0\}$. Although this result is a particular case of the Smoothing Theorem, we include a direct proof here since the explicit expression of the blown-up foliation will be needed later in subsection \ref{subsection-sewing}.
\begin{Proposition}\label{prop-smoothing-one-comp}
The vector field 
$$\mathcal{X} = r\cdot (\varphi^\star X^\reg)$$ 
defines an element of $\mathfrak{X }^{\infty}(\tilde N,\tilde M,\emptyset,\tilde \pi)$.  In particular, $\mathcal{X}$ is smooth.
\end{Proposition}
\begin{proof}
We firstly study the expression of the regularized vector field $\mathcal{X}$ in the initial coordinates. 
By hypothesis, the coefficients of the original vector field $X$ in (\ref{def-X}) can be written as
\begin{equation}\label{coeff-fkX}
\vf_k = \tmmathbf{1}_{\{ x_1 > 0 \}} \vf_{k }^+ + \tmmathbf{1}_{\{ x_1 < 0 \}} \vf_{k }^- 
\end{equation}
where both $\vf_k^+$ and $\vf_k^-$ are smooth functions on $\cR^n$. 
By Fubini's Theorem, the integral expression (\ref{expr-f_k}) for $f_k$ can be rewritten as
\begin{equation}\label{Fubini-f_k}
  f_k \left( x, \underline{x}, \eps  \right) = \int_{\cR^{n - 1}}
  \left( \int_{\cR}  \vf_k \left( x_1- \eps \tau, \underline{x} -
  \eps  \underline{\tau} \right) m \left( \tau, \underline{\tau} \right)
  d \tau \right) d \underline{\tau}
\end{equation}
where we write $x = (x_1,\underline{x})$ and $t = (\tau,\underline{\tau}) \in \cR \times \cR^{n-1}$. Using the above decomposition of $\vf_k$, the innermost integral can be further decomposed as 
{\small \begin{align}\label{equation-integexpr}
\int_{\cR} 
   \vf_k \left( x_1- \eps \tau, \underline{x} -
  \eps  \underline{\tau} \right) m \left( \tau, \underline{\tau} \right) d\tau
   &= \int_{\{ x_1- \eps \tau \ge 0 \}} 
        \vf_k^+ \left( x_1- \eps \tau, \underline{x} -
  \eps  \underline{\tau} \right) m \left( \tau, \underline{\tau} \right)  \\
   &\;\;+ \int_{\{ x_1- \eps \tau \le 0 \}} 
        \vf_k^- \left( x_1- \eps \tau, \underline{x} -
  \eps  \underline{\tau} \right) m \left( \tau, \underline{\tau} \right) \notag
\end{align}}
which obviously has a discontinuous limit as $\eps \rightarrow 0$. 

Let us now compute the pull-back of $X^\reg$ under the blowing-up through the coordinates of the directional charts.  Firstly, the $\eps$-directional chart is given by the coordinate change $x_1 = \eps  y$, with a new variable $y\in \cR$.  In this chart, the vector field assumes the form 
\[ 
 \frac{1}{\eps } F_1 \left( y, \underline{x}, \eps 
   \right) \frac{\partial}{\partial y} + \sum_{k \ge  2} F_k \left( y,
   \underline{x}, \eps  \right) \frac{\partial}{\partial x_k} \]
where each function $F_k = f_k \circ \varphi$ is given $F_k \left( y, \underline{x}, \eps  \right) : = f_k \left(
\eps  y, \underline{x}, \eps  \right)$. Therefore, since the divisor $D$ is expressed in this chart by the equation $\{\eps = 0\}$, we conclude that ${\mathcal{X}}$ 
is locally generated by the vector field
\begin{equation}\label{expression-calX}
  F_1 \left( y, \underline{x}, \eps 
   \right) \frac{\partial}{\partial y} + \eps \left(\sum_{k \ge  2} F_k \left( y,
   \underline{x}, \eps  \right) \frac{\partial}{\partial x_k}\right)
\end{equation}
As a consequence, it suffices to prove that each coefficient $F_k$ extends smoothly to $\{\eps = 0\}$. 
By substituting $x_1=\eps  y$, the integral (\ref{equation-integexpr}) assumes the form
\begin{align}\label{equation-integexprF}
\int_{\cR} 
   \vf_k\!\left( \eps  (y - \tau), \, \underline{x} - \eps  \underline{\tau} \right) 
   m \!\left( \tau, \underline{\tau} \right) d\tau
   &= \int_{\{ \tau \le  y \}} 
        \vf_k^+ \!\left( \eps  (y - \tau), \, \underline{x} - \eps  \underline{\tau} \right) 
        m \!\left( \tau, \underline{\tau} \right) d\tau  \\
   &\; + \int_{\{ \tau \ge  y \}} 
        \vf_k^- \!\left( \eps  (y - \tau), \, \underline{x} - \eps  \underline{\tau} \right) 
        m \!\left( \tau, \underline{\tau} \right) d\tau \notag
\end{align}
And, by substituting $x_1=\eps  y$ into the integral (\ref{Fubini-f_k}), we conclude that $F_k$ is given by the integral
$$
F_k(y,\underline{x},\eps) = \int_{\cR^{n-1}} \; \int_{\cR} 
   \vf_k\!\left( \eps  (y - \tau), \, \underline{x} - \eps  \underline{\tau} \right) 
   m \!\left( \tau, \underline{\tau} \right) d\tau\; d\underline{\tau}
$$
which is precisely the integral studied in Lemma \ref{Lemma-partial-conv} in the particular case where $I=\{1\}$. Therefore, $F_k$ is a smooth function. 

We now consider the $\pm x_1$-directional charts, where the blow-up takes the form $x_1 = \pm z, \eps= z \rho$, with new variables $z,\rho\in \cR_{\ge 0}$. Using these coordinates, the vector field $\mathcal{X}$ (outside the divisor $D$) is given by
\[
 \mathcal{X}= \frac{1}{z} G_{1,\pm} \left( z, \underline{x}, \rho
   \right) \left(z\frac{\partial}{\partial z}-\rho\frac{\partial}{\partial \rho}\right) + \sum_{k \ge  2} G_{k,\pm} \left( z,
   \underline{x}, \rho \right) \frac{\partial}{\partial x_k} 
\]
where $G_{k,\pm}=f_k \circ \varphi$, i.e.\ $G_{k,\pm}( z,\underline{x}, \rho) = f_k(\pm z,\underline{x},z \rho)$.  
In this chart, the divisor has equation $D=\{ z = 0\}$ and therefore ${\mathcal{X}}$ is equivalent to 
\begin{equation}\label{ref-normal-direction}
G_{1,\pm} \left( z, \underline{x}, \rho
   \right) \left(z\frac{\partial}{\partial z}-\rho \frac{\partial}{\partial \rho}\right) + z \left( \sum_{k \ge  2} G_{k,\pm} \left( z,
   \underline{x}, \rho \right) \frac{\partial}{\partial x_k}  \right)
\end{equation}
We claim that each coefficient $G_{k,\pm}$ is smooth.  To simplify the notation, we make the computations for $G_{k,+}$, the negative case being similar.

Referring again to the integral in (\ref{equation-integexpr}), the substitution $x_1 = z, \eps= z \rho$ yields
\begin{align*}
\int_{\cR} 
   \vf_k\!\left( z(1-\rho), \, \underline{x} - z \rho \underline{\tau} \right) 
   m \!\left( \tau, \underline{\tau} \right) d\tau
   &= \int_{{\{ \rho \tau \le 1 \}}} 
        \vf_k^+ \!\left( z(1-\rho \tau), \, \underline{x} - z \rho \underline{\tau} \right) 
        m \!\left( \tau, \underline{\tau} \right) d\tau  \\
   &\; + \int_{{\{ \rho\tau  \ge 1 \}}} 
        \vf_k^- \!\left( z(1-\rho \tau), \, \underline{x} - z \rho \underline{\tau} \right) 
        m \!\left( \tau, \underline{\tau} \right) d\tau
\end{align*}
We claim that this is smooth function of the variables $z,\rho,\underline{x}$ and $\underline{\tau}$. To see this, the crucial fact to observe is that for $\rho$ sufficiently small, the second integral in the right-hand side {\em vanishes identically}. Indeed, since the mollifier $m$ has compact support, there exists an $\rho_0 > 0$ such that $m(\tau,\underline{\tau})$ vanishes identically for $\tau \ge \frac{1}{\rho_0}$. Therefore, for all $\rho \le \rho_0$, the above integral assumes the simple form 
$$
\int_{{\cR}} 
        \vf_k^+ \!\left( z(1-\rho \tau), \, \underline{x} - z \rho \underline{\tau} \right) 
        m \!\left( \tau, \underline{\tau} \right) d\tau
$$
which is a smooth function since $\vf_k^+$ is smooth. 
A further integration with respect to $\underline{\tau}$ shows that $G_{k,+}$ is smooth. This concludes the proof. 
\end{proof}
Note that the restriction of $\mathcal{X}$ to the zero-fiber $F_0 = \tilde{\pi}^{-1}(0)$
defines a smooth foliation (on $F_0$)  which projects into the discontinuous
foliation defined by the original vector field $X$ (see Figure \ref{fig-zero-fiber}).
\begin{figure}[htb]
\centering
  \includegraphics[width=5.65039190607372cm,height=4.05404696313787cm]{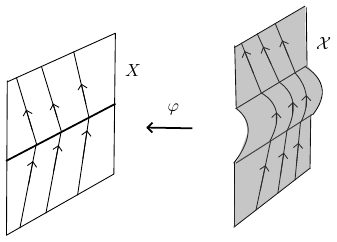}
  \caption{The zero fiber $F_0$ is the gray-shaded set.}\label{fig-zero-fiber}
\end{figure}

We conclude this subsection by studying the first order terms of the expansion of ${\mathcal{X}}$ in powers of $r$. These expressions will allow to relate the regularization by convolution defined here with the Sotomayor-Teixeira regularization. They will also be used in subsection \ref{subsection-sewing}.

In the $\eps$-directional chart, we write the first order expansion of the coefficients $F_k$ appearing in (\ref{expression-calX}) in terms of $\eps$ as
$$
F_k  = F^0_k + \eps R_k
$$
where $F^0_k = F_k|_{\{\eps = 0\}}$ and $R_k$ is a smooth function. Then, up to equivalence, we can write ${\mathcal{X}}$ as a sum $\mathcal{X}^0 + \eps \mathcal{R}$, where
$$
\mathcal{X}^0=F_1^0 \frac{\partial}{\partial y} + \eps \sum_{k \ge  2} F_k^0 \frac{\partial}{\partial x_k}, \quad\text{and}\quad
\mathcal{R}= R_1 \frac{\partial}{\partial y} + \eps \sum_{k\ge 2}{R_k \frac{\partial}{\partial x_k}}
$$
Let us show that the first term $\mathcal{X}^0$ has a simple expression in terms of the original piecewise-smooth vector field $X$.  By setting $\eps = 0$ in the integral expression (\ref{equation-integexprF}), we obtain 
\begin{align*}
 F_k^0(y,\underline{x}) =& \vf_k^+(0,\underline{x}) M_+(y) + 
 \vf_k^-(0,\underline{x}) M_-(y) 
\end{align*}
where $\vf_k^{\pm}$ are the coefficients appearing in (\ref{def-X}) and the {\em weight-functions} $M_+$ and $M_-$ are given by
$$
M_+(y)=\int_{-\infty}^y \int_{\cR^{n-1}} m(\tau,\underline{\tau}) d \underline{\tau} d\tau,\quad \text{and}\quad 
M_-(y)=\int_y^{\infty} \int_{\cR^{n-1}} m(\tau,\underline{\tau}) d \underline{\tau} d\tau
$$
In particular, by setting $\{\eps = 0\}$, we conclude that the restriction of ${\mathcal{X}}$ to the exceptional divisor $D$ is given by the {\em purely vertical} vector field
\begin{equation}\label{vX-purely-vertical}
{\mathcal{X}}|_D = \Big( \vf_1^+(0,\underline{x}) M_+(y) + 
 \vf_1^-(0,\underline{x}) M_-(y) \Big)\; \frac{\partial}{\partial y}
\end{equation}
\begin{Remark}
Intuitively, the functions $M_+(y)$ and $M_-(y)$ measure respectively the {\em mass} of the mollifier $m$ situated above and below the hyperplane 
$\{\tau = y\}$. In particular, $M_+ + M_- =1$ and the above expression shows that ${\mathcal{X}}$ is (up to a $\eps$-correction term) a {\em convex combination} of the values of the boundary vector fields $X_+$ and $X_-$ restricted to the discontinuity locus. 
\end{Remark}
\begin{Remark}
From the definition of a mollifier it follows that the function $\phi(y) = 2 M_+(y) - 1= - 2 M_-(y) + 1$ is a {\em smoothed sign function}, i.e.~there exists constants $y_- <  y_+$ such that:
\begin{enumerate}
\item $\phi = -1$ identically for $y\le y_-$
\item $\phi$ is non-decreasing on the interval $(y_-,y_+)$, and
\item $\phi = 1$ identically for $y \ge y_+$. 
\end{enumerate}
\begin{figure}[htb]
\centering
 \includegraphics[width=8.66662567230749cm,height=4.02737603305785cm]{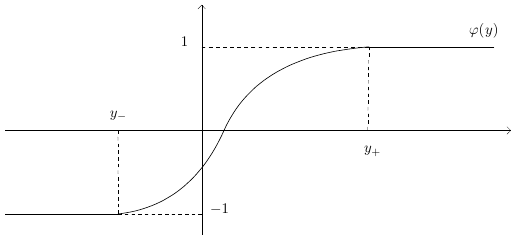}
 \end{figure}
Using such function we can write the vector field $\mathcal{X}^0$ defined above as
$$
\mathcal{X}^0 = \left(\frac{1+\phi(y)}{2}\right) \mathcal{X}^0_+ + \left(\frac{1-\phi(y)}{2}\right) \mathcal{X}^0_-
$$
where 
$$
\mathcal{X}^0_\pm =  \vf_1^\pm(0,\underline{x}) \frac{\partial}{\partial y} +  \eps \sum_{k\ge 2}
\left( \vf_k^\pm(0,\underline{x}) \frac{\partial}{\partial x_k}\right)
$$
This expression coincides with the well-known \emph{Sotomayor–Teixeira regularization} introduced in \cite{ST} when written in the blow-up chart. Consequently, the convolution regularization is asymptotically equivalent to the Sotomayor- Teixeira regularization in a neighborhood of the exceptional divisor, up to an error of order $O(\varepsilon)$. 
 \end{Remark}
A similar computation in the $\pm x_1$-directional charts shows that (up to equivalence) we can write
${\mathcal{X}}$ as a sum  $\mathcal{X}^0_\pm + z\, \mathcal{R}_\pm$, where $\mathcal{R}_\pm$ is some smooth vector field and
$$
\mathcal{X}^0_\pm = G_{1,\pm}^0
    \left(z\frac{\partial}{\partial z}-\rho\frac{\partial}{\partial \rho}\right) + z \left( \sum_{k \ge  2} G_{k,\pm}^0 \left( z,
   \underline{x}, \rho \right) \frac{\partial}{\partial x_k}  \right)
$$
has coefficients $G_{k,\pm}^0=M_+\left(\pm \frac{1}{\rho}\right) \vf_k^+(0,\underline{x})  + M_-\left(\pm \frac{1}{\rho}\right)  \vf_k^-(0,\underline{x})$, 
with the same functions $M_+$ and $M_-$ as above.
\subsection{Proof of the Smoothing theorem for regularized vector fields}\label{subsect-smoothing-vf}
The Smoothing Theorem stated at the introduction is the analog of Theorem \ref{theorem-smoothing-rpss} in the context of regularized piecewise-smooth vector fields. 

We recall its enunciate. Let $(M,\Sigma)$ be a piecewise smooth space and let $(N,M,\Sigma,\pi)$ be the associated fibered regularized piecewise smooth space such that 
$$
\reg_m: \mathfrak{X}^{\infty} (M, \Sigma) \longmapsto \mathfrak{X}^{\infty} (N, M,\Sigma,\pi)
$$
is the regularization operator on piecewise-smooth vector fields associated to a mollifier $m\in\Mol(\cR^n)$, as defined in  subsection \ref{subsect-reg-vf}. \begin{SmoothingTheorem}
There exists a finite sequence of blowings-up
\[
(N, M, \Sigma, \pi) = (N_0, M_0, \Sigma_0, \pi_0)
\xleftarrow{\;\varphi_1\;}
\cdots
\xleftarrow{\;\varphi_r\;}
(N_r, M_r, \Sigma_r, \pi_r)
\]
such that the following properties hold:
\begin{itemize}
  \item Each blow-up center \(C_k\) is contained in the corresponding discontinuity locus \(\Sigma_k\).
  
  \item Given an arbitrary vector field \(X \in \mathfrak{X}^{\infty}(M, \Sigma)\) and a mollifier \(m\), let \(\vF = X^{\mathrm{reg}} = \mathrm{reg}_m(X)\) denote the $1$-foliation in \(\mathfrak{F}^{\infty}(N, M, \Sigma, \pi)\) defined by the regularization of \(X\). Then there exists a sequence of $1$-foliations
  \begin{equation}\label{severalvFk}
  \vF_k \in \mathfrak{F}^{\infty}(N_k, M_k, \Sigma_k, \pi_k), \quad k = 0, \ldots, r,
  \end{equation}
  such that \(\vF_0 = \vF\) and \(\vF_{k}\) is a blowing-up of \(\vF_{k-1}\) under \(\varphi_k\).
  
  \item The final discontinuity locus \(\Sigma_r\) is empty.
\end{itemize}
\end{SmoothingTheorem}
The proof is a straightforward adaptation of the proof of Theorem \ref{theorem-smoothing-rpss}. 

We consider initially the analog of the local situation described in subsection \ref{subsect-local-situation-reg-I}. 
Namely, there are local coordinates $x = (x_I,x')$ and $\eps$ such that we can write $M= M_I =\{(x,\eps) \mid \eps = 0\}$ and $\Sigma=\Sigma^{(I)} = \prod_{i\in I} x_i$. 
In these coordinates, the regularized vector field can be written as 
$$
X^\reg = \sum_{i\in I} f_i^{\reg} \frac{\partial}{\partial x_i} + \sum_{k=1}^{K} g_{k}^\reg \frac{\partial}{\partial x'_k} 
$$
where we write $x' = (x'_1,\ldots,x'_l)$ and $K = n-\# I$.  Note that the coefficients $f_i^\reg,g_k^\reg$ are regularized functions lying in $C^{\infty}(N_I,M_I,
\Sigma^{(I)})$.

Following the notation of subsection \ref{subsect-local-situation-reg-I}, let $\varphi:\tilde N \rightarrow N_I$ denote the local blowing-up with center $C_I$. 
As previously, let us describe  the local expressions of $\varphi^\star X^\reg$ in the phase directional charts.

In the coordinates of the $i_1^{th}$-phase directional chart described in subsection \ref{subsect-phase-blowing-up}, we can write $\varphi^\star X^\reg$ as a sum
$\frac{1}{x_{i_1}} Y_1 + Y_2$, where 
$$
Y_1 = (f_{i_1}^{\reg} \circ B) \left( 
z_{i_1} \frac{\partial}{\partial z_{i_1}} 
- \sum_{i \in I \setminus \{ i_1 \}} z_i\frac{\partial}{\partial z_i} - \rho \frac{\partial}{\partial \rho}
\right)
+ \sum_{i \in I \setminus \{ i_1 \}} (f_i^{\reg} \circ B) \frac{\partial}{\partial z_i}
$$
and 
$$
Y_2 = \sum_{k=1}^{K} (g_{k}^\reg \circ B) \frac{\partial}{\partial z'_k} 
$$
Here, we note by $B = B_{(I,i_1)}$ the directional blowing-up map. Therefore, the multiplication by $x_{i_1}$ (which is a local equation for the exceptional divisor $D$) defines the vector field
$$
x_{i_1}\cdot ((B_{(I,i_1)})^\star\, X^\reg) = Y_1 + x_{i_1} Y_2
$$
 which, according to Lemma \ref{lemma-+directional}, is a piecewise smooth vector field with coefficients lying in $C^{\infty} (W, \tilde{M},\Sigma_{I \setminus \{ i_1 \}})$. 
 
Very similar computations apply to the family of directional blowing-up charts described in subsection \ref{subsect-blowing-up-regular}. Using Lemma \ref{Lemma-partial-conv}, one verifies that the resulting vector field
$$
\rho\cdot  (B_F^\star\, X^\reg)
$$
is smooth, where $\{\rho = 0\}$ is the local equation of the exceptional divisor $D$. 

These computations show that the inductive blowing-up procedure described in subsection \ref{subsect-smoothingprocedure} operates naturally in the components of the regularized vector field $X^\reg$, thereby eliminating successively the strata of the discontinuity locus $\Sigma$.

Moreover, for each $k = 1, \ldots, r$, it shows that the piecewise 1-foliation 
$\vF_{k}$ in (\ref{severalvFk}) can be obtained from its predecessor $\vF_{k-1}$ as follows: if $(U, X)$ is a local generator of $\vF_{k-1}$, then
$$\bigl(\varphi_k^{-1}(U),\, r \cdot (\varphi_k^{-1})^\star X\bigr)$$
is a local generator of $\vF_{k}$, where $(\varphi_k^{-1})^\star X$ denotes the pullback of $X$ under $\varphi_k$, and where $\{r = 0\}$ is the local equation of the exceptional divisor $D_{k} = \varphi_k^{-1}(C_{k-1})$. This concludes the proof of the Theorem.
\begin{Remark}
Arguing as in subsection~\ref{subsection-st-reg-link}, we can describe the blown-up vector
field in the family chart of the blowing-up of a stratum
$\Sigma_I\times\{0\}=\{\eps=0,\ x_i=0\ (i\in I)\}$, with $c=\#I$.
Let $p$ be a point of the open stratum $\Sigma_I\setminus\bigcup_{J\supsetneq I}\Sigma_J$, and
take adapted coordinates $x=(x_I,x')$ centred at $p$, in which
$X=\sum_{\us\in\{\pm\}^I}\mathbf{1}_{\{\us x_I>0\}}X_{\us}$ with
$$X_{\us}=\sum_{i\in I}f_{i,\us}\,\partial_{x_i}+\sum_{j}g_{j,\us}\,\partial_{x'_j}.$$
In the family chart $B_F\colon x_I=\rho z_I,\ \eps=\rho,\ x'=z'$ of
subsection~\ref{subsect-blowing-up-regular}, the generator $\mathcal{X}=\rho\,B_F^{*}X^{\reg}$ is
\[
\mathcal{X}=\sum_{i\in I}\bigl(f_i^{\reg}\circ B_F\bigr)\frac{\partial}{\partial z_i}
  +\rho\sum_{j}\bigl(g_j^{\reg}\circ B_F\bigr)\frac{\partial}{\partial z'_j}.
\]
Hence, on the exceptional divisor $D_I=\{\rho=0\}$, the vector field $\mathcal{X}$ is tangent to
the fibres $\{z'=\mathrm{const}\}$, and on the fibre over $p$ it is given by
\begin{equation}\label{inner-field}
\mathcal{X}|_{D_I}=\sum_{\us\in\{\pm\}^I}\mathbb{M}_{\us}(z_I)\sum_{i\in I}f_{i,\us}(p)\,
  \frac{\partial}{\partial z_i},
\qquad
\mathbb{M}_{\us}(z_I)=\int_{\{t\in\R^n:\ \us(z_I-t_I)>0\}}m(t)\,dt .
\end{equation}
The weight $\mathbb{M}_{\us}(z_I)\geq 0$ is the mass that $m$ assigns to the translated orthant
$\{t:\ s_i(z_i-t_i)>0,\ i\in I\}$, and $\sum_{\us}\mathbb{M}_{\us}\equiv 1$. 
Therefore, at each
point $z_I$ of the fibre, $\mathcal{X}|_{D_I}$ is a {\em convex combination} of the normal parts
$\sum_{i\in I}f_{i,\us}(p)\,\partial_{x_i}$ of the $2^c$ vectors $X_{\us}(p)$. This is strongly related to the {\em canopy} studied in \cite{kaklamanos2019}.

Note that for $c=1$ the weights $\mathbb{M}_{\pm}$ are the functions $M_{\pm}$ of
subsection~\ref{subsection-st-reg-link} and we reobtain the expression (\ref{vX-purely-vertical}).

\end{Remark}

\section{Further examples and applications}\label{s7}
We now discuss three examples illustrating the regularization procedure developed in this paper. The first subsection examines the regularization in a neighborhood of sewing-type periodic orbits. Using the blowing-up construction, we show that the Poincaré first return map defined near such orbits embeds smoothly into a family of Poincaré first return maps for the regularized vector field.

The remaining two subsections have a more exploratory character. Their aim is to highlight interesting dynamical phenomena that may arise from the regularization of piecewise smooth vector fields when the discontinuity locus $\Sigma$ is non-smooth.

Our goal is not to provide a systematic account of all possible dynamical behaviours. Nevertheless, the some of the bifurcation phenomena described here are robust under small perturbations and therefore occur also in versal families of piecewise smooth vector fields with the prescribed discontinuity locus.

\subsection{Regularization in the vicinity of sewing periodic orbits} \label{subsection-sewing}
Let $(M,\Sigma)$ be a piecewise smooth space and let $X \in \mathfrak{X }^{\infty} (M, \Sigma)$ be a piecewise-smooth vector field. A {\em sewing-type periodic
orbit} for $X$ is a continuous, oriented, simple closed curve $\gamma \subset
M$ that can be expressed as a concatenation of (oriented) smooth segments
\begin{equation}
  \gamma = \gamma_1 \cup \cdots \cup \gamma_k \label{decomp}
\end{equation}
such that, for each $i \in \{ 1, \ldots, k \}$, the following conditions hold:
\begin{enumeratenumeric}
  \item The interior $\ring{\gamma_i}$ of $\gamma_i$ is contained in a
  connected component $W_i \subset M \setminus \Sigma$ and is a regular
  orbit of the smooth vector field $X_i = X |_{W_i} \nobracket$. The orientation of $\ring{\gamma_i}$ coincides with the orientation of the flow of $X_i$.   
  \item The closure of $\gamma_i$ intersects $\Sigma$ transversely and $X_i
  |_{W_i} \nobracket$ is non-zero at these intersection points. 
\end{enumeratenumeric}
Without loss of generality, we can further suppose that the decomposition
(\ref{decomp}) is {\tmem{minimal}}, namely if $k = 1$ then $\gamma = \gamma_1$
is a periodic orbit of $X  |_{M \setminus \Sigma} \nobracket$. If $k \geqslant
2$ then each $\gamma_i$ has both its initial and endpoints
\[ \{ p_i, q_i \} = \gamma_i \setminus \ring{\gamma_i} \]
contained in some smooth component of $\Sigma$, with the identification $q_i =
p_{i + 1}$ (with $i + 1 \in \{ 1, \ldots, k \}$ taken modulo $k$).

The condition (2) of the above definition implies that $\Sigma$ is a
transverse section for $X_i |_{W_i} \nobracket$ in the vicinity of both $p_i$ and
$q_i$. Therefore, we can define a (germ of)  {\em $\gamma_i$-transition map}
\[ P_i : (\Sigma, p_i) \rightarrow (\Sigma, p_{i + 1}) \]
simply by following the flow of $X_i |_{W_i} \nobracket$ in the vicinity of
$\gamma_i$.
\begin{figure}[htb]
\centering
 \includegraphics[width=6.0255722812541cm,height=4.84122064803883cm]{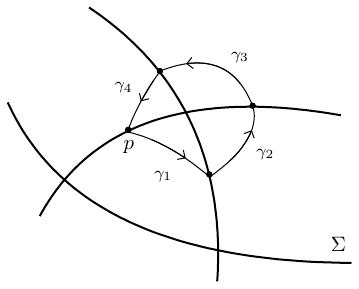}
 \end{figure}
The {\tmem{sewing Poincar{\'e} (first return) map}} associated to $\gamma$ at $p = p_1$
is the {\em smooth} map $P : (\Sigma, p) \rightarrow (\Sigma, p)$ defined by the
composition
\begin{equation}\label{poincareP}
  P = P_k \circ \cdots \circ P_1 
\end{equation}
of these successive $\gamma_i$-transition maps. 

Let us now suppose that $M$ is an open subset of $\cR^n$, and let 
$$X^\reg = \reg_m(X)$$ be the regularization of $X$ associated to an arbitrary mollifier $m$, as defined in subsection \ref{subsect-reg-vf}. 
The main result of this subsection is the following:
\begin{Proposition}\label{prop-regularized-poincare}
 The sewing Poincaré map
  $P$ extends smoothly to a map in the product space
  $$P^\reg:(\Sigma, p) \times (\cR_{\geqslant 0}, 0) \rightarrow (\Sigma, p) \times (\cR_{\geqslant 0}, 0)$$
 Moreover, if $\pi(p,\eps)=\eps$ denotes the linear projection then, for each $\eps > 0$ the map $P_\eps = P^\reg|_{\pi^{-1}(\eps)}$ is the Poincar{\'e} first return map of the
smoothed vector field $X_{\eps} = m_\eps \ast X$.  
\end{Proposition}
\begin{proof}
  The idea is to study separately each $\gamma_i$-transition map $P_i :
  (\Sigma, p_i) \rightarrow (\Sigma, p_{i + 1})$.
  
  Choose some $p = p_i$ and fix local coordinates $(x_1,\ldots,x_n) \in \cR \times \cR^{n
  - 1}$ as in subsection \ref{subsection-st-reg-link}. By the
  sewing hypothesis, we can further assume that such coordinates are chosen in
  such a way that
  \begin{equation}\label{sewing-hypothesis-f1}
    f_1^+ > 0, \quad f_1^- > 0 
  \end{equation}
  where $f_k^+,f_k^-$ are the coefficients of $X$ given by
  (\ref{coeff-fkX}). We extend the variety $\Sigma$ to the total space $M
  \times (\cR_{\geqslant 0}, 0)$ by defining
  $\Omega = \Sigma \times \cR_{\geqslant 0}$.
  Therefore, in the local coordinates $\left( x, \eps \right)$
  from the total space, such extended section is given by $\Omega = \{ x_1 = 0
  \}$. We claim that the transition map $P_i$, with domain the local
  section $(\Sigma, p_i)$, extends smoothly to the transition map for
  the regularized vector field $X^\reg$, defined on the extended local section $(\Omega,
  p_i)$.
  
  Following the construction made in subsection \ref{subsection-st-reg-link},
  consider the blowing-up
  \[ \varphi : \mathcal{N} \rightarrow M \times \cR_{\geqslant 0} \]
  which defines a smoothing of $X^\reg$ up to the zero-fiber $F_0 = \pi^{-1}(0)$. We
  study the Poincaré map using the coordinates of the directional charts described in that subsection.
  
  Consider two sections $\Sigma^+, \Sigma^-  \subset D$ lying in the
  exceptional divisor, defined in the coordinates of the respective $(+
  x_1)$-directional chart and $(- x_1)$-directional charts by
  \[ \Sigma^{\pm} = \left\{ z = \rho = 0 \right\} \]
  (see figure below). We denote by $\tilde{p}^{\pm}$ the point lying in
  $\Sigma^{\pm} \cap \varphi^{- 1} (\{ p \})$, which is given in these
  coordinates by $\left( z,\rho,\underline{x} \right) = 0$.
  Note that the strict transform of the extended section $\Omega$ lies in the
  domain of the $\eps$-chart and is given by
  \[ \bar{\Omega} = \{ y = 0 \} \]
  and denote by $\bar{p}$ the point in $\bar{\Omega} \cap \varphi^{- 1} (\{ p
  \})$, with coordinates $\left( y, \eps,\underline{x} \right) =
  0$.
  
  Using the original local coordinates at $p$ and the blowing-up charts, we can fix
  a common local parametrization $\underline{x}=(x_2,\ldots,x_n) \in \cR^{n - 1}$ for all the three
  sections $\Sigma, \Sigma^{\pm}$ and also for the section $\bar{\Omega}$.
  
  \begin{figure}[htb]
\centering
 \includegraphics[width=6.35624098124098cm,height=4.18684409025318cm]{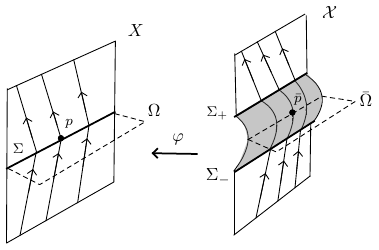}
  \end{figure}
 
 We now apply this construction to each point $p_1, \ldots, p_k$ in
  $\gamma \cap \Sigma$ and obtain a collection of sections $\Sigma^{\pm}_i,
  \bar{\Omega}_i$ (equipped with local parameterizations) such that the
  following properties hold:
  \begin{enumeratealpha}
    \item Each transition map $P_i : (\Sigma, p_i) \rightarrow (\Sigma, p_{i +
    1})$ defined by $X$ coincides with the corresponding transition map
    \[ \tilde{P}_i : (\Sigma^+_i, \tilde{p}_i^+) \rightarrow (\Sigma^-_{i +
       1}, \tilde{p}_{i + 1}^-) \]
    defined by the restriction of the foliation $\mathcal{X}$ to $F_0
    \setminus D$, where $\Sigma, \Sigma_i^{\pm}$ are parametrized as
    above.
    
    \item The restriction of the foliation  $\cal X$  to the divisor $D$ defines two transition maps
    \[ (\Sigma^-_i, \tilde{p}_i^-) \rightarrow (\bar{\Omega} \cap D,
       \bar{p}_i) \rightarrow (\Sigma^+_i, \tilde{p}_i^+) \]
    which are simply the identity map. 
  \end{enumeratealpha}
  The property (a) is a simple consequence of the fact that the restriction of
  $\mathcal{X}$ to $F_0 \setminus D$ coincides with $X|_{M\setminus \Sigma}$. 
  
  Property (b) is a consequence of the computations made in
  subsection \ref{subsection-st-reg-link}. Indeed, we have seen that the
  restriction of $\mathcal{X}$ to $\mathcal{D}$ is a {\tmem{purely vertical}}
  foliation, given by the expression (\ref{vX-purely-vertical}). 
  Furthermore, the sewing hypothesis
  (\ref{sewing-hypothesis-f1}) implies that the coefficient $\vf_1^+(0,\underline{x}) M_+(y) + 
 \vf_1^-(0,\underline{x}) M_-(y)$ appearing in that expression is strictly positive.  Therefore, upon
  division by such coefficient,  we see that $\mathcal{X} |_D$ is simply generated by the constant vector field
  $\frac{\partial}{\partial y}$. Hence, since the $\underline{x}$-coordinates are preserved, the
  transition maps given in (b) are equal to the identity.
  
  As a conclusion, the composition of all such maps defines a first return
  map for $\mathcal{X}$,
  \[ \bar{P} : (\bar{\Omega} \cap \mathcal{D}, \bar{p} ) \rightarrow
     (\bar{\Omega} \cap \mathcal{D}, \bar{p}) \]
  near the section at the initial point $\bar{p} = \bar{p}_1$. Furthermore, such map
  coincides with the original Poincar{\'e} map $P : (\Sigma, p) \rightarrow
  (\Sigma, p)$.
  
To finish the proof,  we need to show that $\bar P$ extends smoothly to a Poincaré first return map $P^\reg :
  (\bar{\Omega}, \bar{p} ) \rightarrow (\bar{\Omega}, \bar{p})$ associated the smoothed vector field $\cal X$. By the sewing condition (\ref{sewing-hypothesis-f1}),
  the foliation $\cal X$ is locally transverse to each local section 
  ($\bar{\Omega}_i$, $\overline{p_i}$). Therefore, using the above notation, it suffices to consider the following local situation: The vector field $X$ is given by (\ref{def-X}) and we need to prove the smoothness of the local transition maps defined by the flow of $\cal X$ in the coordinates of the $(\pm x_1)$-directional charts, given by (\ref{ref-normal-direction}).
  
Let us consider only the $(+ x_1)$ chart, since the argument is analogous in the other chart. In this chart, it follows from the sewing hypothesis (\ref{sewing-hypothesis-f1}) that the smoothed foliation $\cal X$ is locally generated by a vector field of the form
$$
Y=z\frac{\partial}{\partial z}-\frac{\partial}{\partial \rho}  + z \left( \sum_{k \ge  2} H_{k} \frac{\partial}{\partial x_k}  \right).
$$
since we can factor out the non-zero $G_{1,+}$ from (\ref{ref-normal-direction}). Let us fix local transverse sections $T=\{z = c\}$ and $S = \{\rho = c\}$, where $c>0$ is some sufficiently small constant. Then,  it suffices to prove that the transition map 
$$\Delta: T \longrightarrow S$$ 
defined by the flow of $Y$ is smooth. 

To show this, note that if we choose $c>0$ sufficiently small, the restriction of $Y$ to the domain $\{\rho < c\}$ coincides precisely (up to division by a non-zero smooth function) with the regularization $m_\eps \ast X_+$ of the {\em smooth} vector field $X_+$ seen in the $(+ x_1)$ directional chart. 

Therefore, we can study the transition map $\Delta$ in the original coordinates $(x_1,..,x_n,\eps)$, where it corresponds simply to the transition map between the sections $T=\{x_1=c\}$ and $S=\{x_1=\eps/c\}$ defined by the flow of the smooth vector field $m_\eps \ast X_+$ (see figure \ref{figure-smoothness}). Since $X_+$ is non-vanishing in the vicinity of $\{x_1=\eps = 0\}$ (again by (\ref{sewing-hypothesis-f1}), the flow-box theorem shows that $\Delta$ is smooth. This concludes the proof.

  \begin{figure}[htb]
\centering
 \includegraphics[width=6.35624098124098cm,height=4.18684409025318cm]{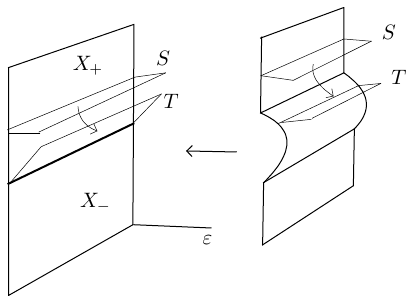}
 \caption{The transition map $\Delta:T\rightarrow S$ seen in the $(+x_1)$-directional chart and in the original coordinates.}
 \label{figure-smoothness}
  \end{figure}
\end{proof}

We will say that $\gamma$ is a {\tmem{$(s : u)$-hyperbolic periodic orbit of
sewing-type}} if there are non-negative integers $s,u$ with $s + u = n-1$ such that the differential $dP
(p)$ of the Poincaré map $P$ given in (\ref{poincareP}) has $s$ eigenvalues in $\{ z \in \C: | z | < 1 \}$ and $u$ eigenvalues in $\{ z \in \C
: | z | > 1 \}$, counted with multiplicity.

As a simple consequence of the Implicit Function Theorem, we obtain the
following structural stability result:

\begin{Corollary}
  Suppose that $\gamma$ is a $(s : u)$-hyperbolic periodic orbit of
  sewing-type. Then, there exists a neighborhood $U \subset M$ of $\gamma$ and
  a neighborhood of the origin $E \subset \cR_{\geqslant
  0}$, such that for each $\eps \in E\setminus\{0\}$, $X_\eps = m_\eps \ast X$ has a unique $(s : u)$-hyperbolic
  periodic orbit $\gamma_{\eps}$ contained in $U$. 
\end{Corollary}
Moreover, such family of periodic orbits $\left\{ \gamma_{\eps}
\right\}$ converges to $\gamma$ in the Hausdorff distance as $\eps \rightarrow 0$. 

We recall that, in the planar case, the hyperbolicity ratio of a periodic orbit of sewing-type can be explicitly written in terms of a divergence integral. (see also Proposition 13 in \cite{SM}). 

More precisely, suppose that $M$ is an open subset of $\cR^2$. For each $i=1,\ldots,k$, we denote by $\theta_i^{\mathrm{in}}$ and $\theta_i^{\mathrm{out}}$ the (oriented) angles of intersection of the orbit $\gamma_i$ with the $\Sigma$, at points $p_i$ and $p_{i+1}$ respectively (see Figure below).  

\begin{figure}[htb]
\centering
 \includegraphics[width=5.8cm]{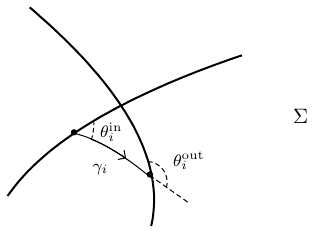}
  \end{figure}

Let us also fix a  
parametrization $x \in (\cR,0) \rightarrow (\Sigma,p)$ 
of the section $\Sigma$ at $p$, and consider the regularized Poincaré map $(x,\eps) \mapsto P^\reg(x,\eps)$, as defined in Proposition 
\ref{prop-regularized-poincare}. 

Then, as a direct application the classical variational formula, the identity (\ref{poincareP}) and the chain rule, we conclude that  the derivative $\frac{d P^\reg}{d x}$ at the origin is given by
$$
\frac{d P^\reg}{d x}(0) = \prod_{i=1}^{k}{\frac{\| X_i(p_i)\|\, \cdot\ \sin(\theta_i^{\mathrm{in}}\,) }{\| X_i(p_{i+1})\|\cdot \sin(\theta_i^{\mathrm{out}})}\; \exp \int_{0}^{T_i}{\mathrm{div}(X_i)(\gamma_i(t)) d t} }
$$
where, for each $i=1,\ldots,k$, $T_i > 0$ denotes  the travel time of $X_i$ along the orbit $\gamma_i$ from $p_i$ to $p_{i+1}$.  Note that the formula does not depend on the parametrization of $\Sigma$.

\subsection{Regularized dynamics on the planar cross}\label{subsect-planar-cross}
Consider the piecewise smooth space $(M,\Sigma)=(\cR^2,\{xy=0\})$. We will consider a vector field $X \in \mathfrak{X}^\infty(M,\Sigma)$ which is {\em piecewise constant}. Namely, we assume that $X$ has the form
\[ X = \sum_{s, t \in \{ - 1, 1 \}} \tmmathbf{1}_{\{ s x > 0, t y > 0 \}} X_{st} \]
where
\[ X_{s t} = a_{s t}  \frac{\partial}{\partial x} + b_{s t}
   \frac{\partial}{\partial y} \]
are constant vector fields.

Let $X^{\mathrm{reg}} = \reg_m(X)$ denote the regularization of $X$ with respect to a fixed mollifier $m$.  
In what follows, we will suppose that such mollifier has the form $m(x,y) = \mathbf{m}(x)\cdot\mathbf{m}(y)$, where $\mathbf{m}$ is a mollifier in $\cR$ such that
\begin{itemize}
    \item $\mathbf{m}$ is an even function and its support is contained in the interval $[-1,1]$.
    \item $\mathbf{m} = \frac{1}{2}$ on the interval  $\left[-\frac{1}{2},\frac{1}{2} \right]$.
    \end{itemize}
These hypothesis guarantee that the function $\mathbf{M}(x)=\displaystyle \int_{\{u < x\}} \mathbf{m}(u)\, d u$ has the form
$$
\mathbf{M}(x) = \frac{1}{2} (1+x),\qquad \forall x\in \left[- \frac{1}{2},\frac{1}{2} \right]
$$
(see Figure below). As a consequence, by Fubini's Theorem, we have 
\begin{equation}\label{expression-M}
M(x,y) := \int_{\{u < x,v < y\}} m(u,v)\, du\, dv = \frac{1}{4}(1+x)(1+y)
\end{equation}
when restricted to the square $Q = \left[-\frac{1}{2},\frac{1}{2} \right]^2$. We will say that $Q$ is the {\em core region} of the regularization.
\begin{figure}[htb] 
\centering
  \includegraphics[width=7.69519054178145cm]{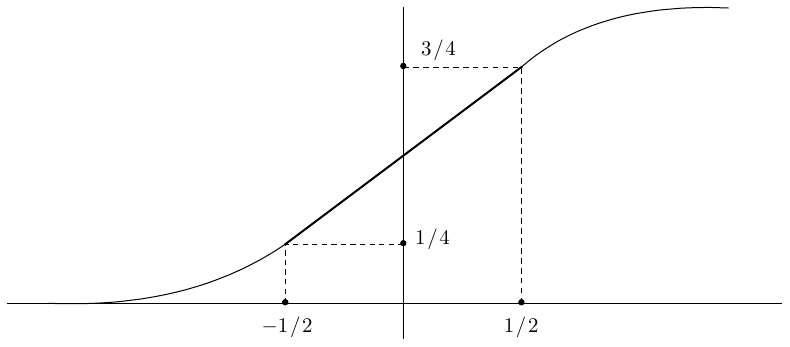}
\end{figure}

Now, we consider the smoothing of $X^\reg$, as described in subsection \ref{subsect-smoothing-vf}. In this setting, such smoothing is obtained by initially blowing-up the origin and then blowing-up the (strict transform of) the $x$-axis and $y$-axis.   

\begin{figure}[htb!]
\centering
\includegraphics[width=6cm]{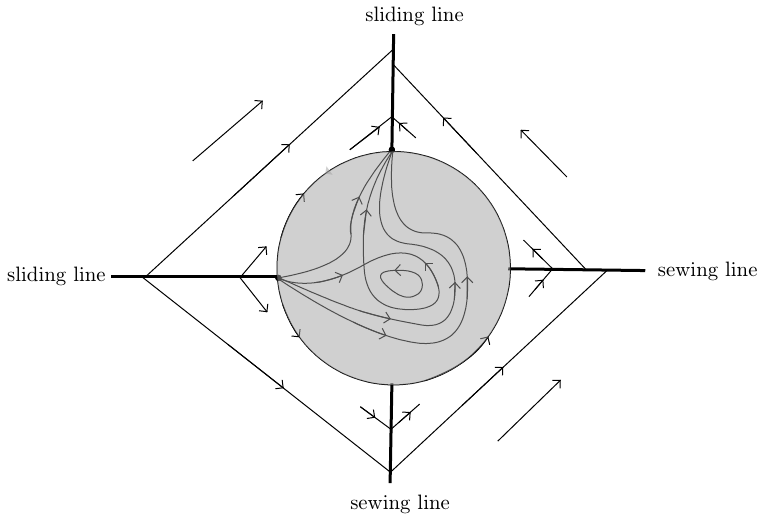}\qquad\includegraphics[width=4.7cm]{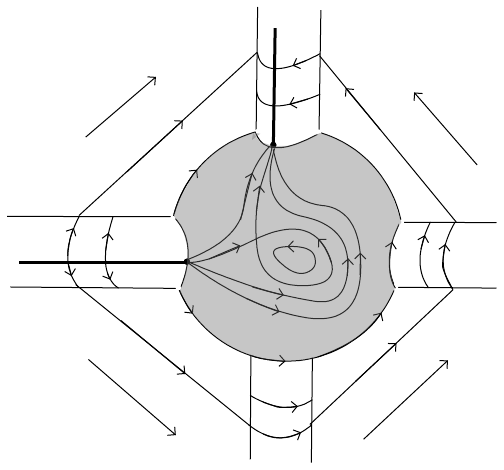}
\caption{A schematic drawing of the two-step smoothing process for the planar cross.}
\end{figure}

Here, we will only study the $\eps$-directional chart of the first blowing-up. In other words, we consider the directional blowing up of the origin on the product space $\cR^2 \times \cR_{\ge 0}$ given by the map
$$
x = \bar{\eps}\, \bar x,\quad y = \bar{\eps}\, \bar y,\quad \eps = \bar{\eps}
$$ 
Under this transformation, we obtain the following vector field 
\[ \mathcal{X} = \frac{1}{\varepsilon} \sum_{s, t \in \{ - 1, 1 \}} M_{s t}
   (x, y) X_{s t} \]
where $M_{s t}$ is defined by
\[ M_{s t} (x, y) = \int_{\cR^2} \tmmathbf{1}_{\{ s (x - u) > 0, t (y -
   v) > 0 \}} m (u, v) d u d v \]
where we drop the bars to simplify the notation. From (\ref{expression-M}), we obtain that
$$
  M_{s t} (x, y) = \frac{1}{4} (1 + s x) (1 + t y), 
 $$
for all $(x,y)\in Q$. Therefore, restricted to the {\tmem{core region}} \ $Q = [- 1/2, 1/2]^2$ we obtain
the quadratic vector field,
{\small
\[ {\mathcal{X} = \frac{1}{4 \varepsilon} [(1 + x) (1 + y) X_{+ +} + (1
   + x) (1 - y) X_{+ -} + (1 - x) (1 + y) X_{- +} + (1 - x) (1 - y) X_{- -}]}
\]
}
After multiplication by $\eps$, we write the associated differential equations as 
\begin{eqnarray*}
  x' & = & a_{+ +} (1 + x) (1 + y) + a_{+ -} (1 + x) (1 - y) + a_{- +} (1 - x)
  (1 + y) + a_{- -} (1 - x) (1 - y)\\
  y' & = & b_{+ +} (1 + x) (1 + y) + b_{+ -} (1 + x) (1 - y) + b_{- +} (1 - x)
  (1 + y) + b_{- -} (1 - x) (1 - y)
\end{eqnarray*}
By regrouping the terms in the quadratic polynomials, we can rewrite this family in the form
\begin{equation}
  \left\{ \begin{array}{l}
    x' = A (x - a) (y - b) - B = f (x, y)\\
    \\
    y' = C (x - c) (y - d) - D = g (x, y)
  \end{array} \right. \label{sistema}
\end{equation}
with parameters $A, B, C, D, a, b, c, d$. In what follows, we will suppose
that $A C \neq 0$, $a\ne c$, $b \ne d$. Under the action of the group of affine coordinate changes
\[ X = \alpha x + \beta, \quad Y = \gamma y + \delta \quad t = \nu T \quad \]
The system assumes the form
\[ \begin{array}{l}
     X' = \frac{\nu A}{ \gamma} (X - (\alpha a + \beta)) (Y - (\gamma b +
     \delta)) - \alpha \nu B \\
     \\
     Y' = \frac{\nu C}{\alpha} (X - (\alpha c + \beta)) (Y - (\gamma d +
     \delta)) - \gamma \nu D
        \end{array} \]
By a convenient choice of such coordinate change, we can suppose that the
system has the form
\[ \left\{ \begin{array}{l}
     x' = \left( x + \frac{1}{2} \right)  \left( y + \frac{1}{2} \right) - B\\
     \\
     y' = C \left( x - \frac{1}{2} \right) \left( y - \frac{1}{2} \right) - D
   \end{array} \right. \]
\begin{Remark}
  We tacitly assume that these affine coordinate changes should happen inside the core
  region $Q$ and that the two quadratic polynomials have distinct critical
  points. This corresponds to imposing additional semi-algebraic restrictions
  on the original coefficients $a_{s t}, b_{s t}$. \ 
\end{Remark}

From now on we will assume that the initial parameters are chosen such that $C, B, D$ are positive. The isoclines are
illustrated in the following figure. 

\begin{figure}[htb]
\centering
\includegraphics[width=4.39871933621934cm,height=4.60756099960645cm]{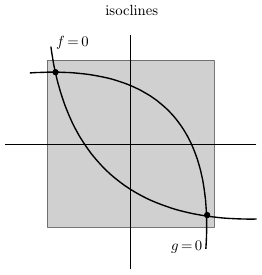}
\end{figure}

The trace and the determinant determine two straight lines inside the core
region, given by
\[ \tmop{Tr} = Cx - \frac{1}{2}  \hspace{0.17em} C + y + \frac{1}{2}, \quad
   \quad \tmop{Det} = C (x - y) \]
and this shows that the upper-left and lower-right singularities appearing in
the above figure are respectively a saddle and an anti-saddle point.

The intersection point
\[ x = y = \frac{1}{2}  \frac{C - 1}{ \hspace{0.17em} C + 1} \]
is a singularity for the original system if we choose the parameter values
\[ B = \frac{C^2}{(C + 1)^2}, \qquad D = \frac{C}{(C + 1)^2} \]
We claim that this singular point is always a {\tmem{Bogdanov-Takens (cusp) singularity}}.

\begin{figure}
\centering
\includegraphics[width=4.3925455857274cm,height=4.3103354978355cm]{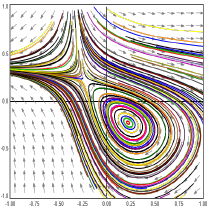}
\end{figure}

Indeed, if we put the system in the Bogdanov-Takens normal form $y
\frac{\partial}{\partial x} + (a x^2 + b x y + R) \frac{\partial}{\partial y}$
we have
\begin{itemize}
  \item $a b < 0$ for $C > 1$ (generic $\tmop{BT}_-$ point of codimension 2)
  
  \item $a b > 0$ for $C < 1$ (generic $\tmop{BT}_+$ point of codimension 2)
  
  \item $a < 0, b = 0$ for $C = 1$
\end{itemize}
Moreover, in a fixed plane $C = \tmop{constant}$ (for $C \neq 1$), the
parameters $(B, D)$ unfold into a {\tmem{versal Bogdanov-Takens family of
codimension 2}}.

For $C = 1$ the problem is much more subtle. For $B = D$ the subfamily is
reversible with respect to the reflection $(x,y) \mapsto (-y,-x)$ and has a generalized
Darboux-type first integral
\[ H (x, y) = \left( xy + \frac{1}{2} (y - x) - \hspace{0.17em} B -
   \frac{1}{4} \right) e^{y - x} \]

A rigorous analysis of the bifurcations occurring near this stratum would require investigating pseudo-Abelian integrals associated with systems admitting Darboux first integrals with exponential terms. For results in this direction, we refer the reader to Bobieński–Mardešić \cite{BobienskiMardesic2008}.

\begin{figure}[htb]
\centering
\includegraphics[width=6cm]{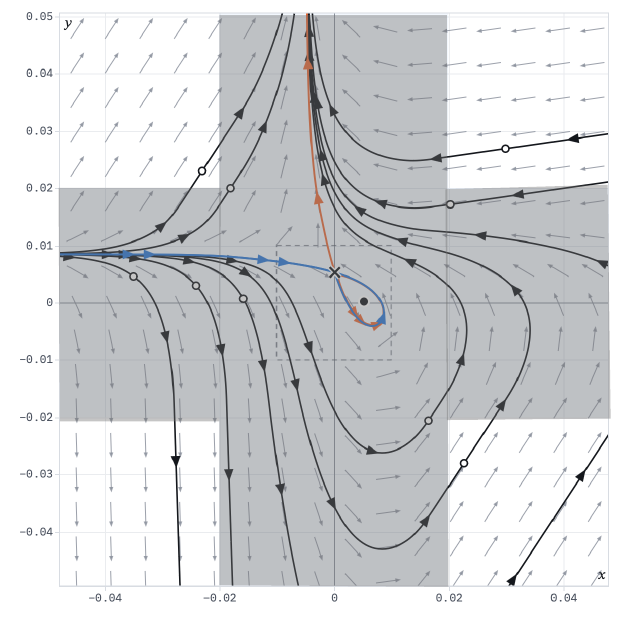}
\caption{Numerically computed phase portrait of the regularized vector field $m_\eps \ast X$, with $X_{++} = (-65,-10)$, $X_{+-} = X_{-+} = (43,62)$, $X_{–-} = (7,-154)$, and $\eps = 0.02$. The gray-shaded region indicates the set where $m_\eps \ast X$ is non-constant.}
\end{figure}
The above computations show the following result:
\begin{Theorem} There is an open set of piecewise constant vector fields on $(\mathbb R^2,\{xy=0\})$ such that, for every mollifier of the form considered and every $\varepsilon>0$, $m_\varepsilon*X$ has a hyperbolic limit cycle of diameter $O(\varepsilon)$ surrounding the origin.
\end{Theorem}
\subsection{Regularized cross in dimension three: Bykov cycles}

Consider a collection $\tmmathbf{C} = \left\{ C_{\us} \right\}$ of eight constant 3-dimensional vector fields
\[ C_{\us} = a_{\us} \frac{\partial}{\partial x} + b_{\us}
   \frac{\partial}{\partial y} + c_{\us} \frac{\partial}{\partial z} \]
indexed by all sign choices $\us = (s_1,s_2,s_3) \in \{-1,+1\}^3$. 

The {\tmem{piecewise constant spatial cross}} associated to such collection
$\tmmathbf{C}$ is the vector field $\tmmathbf{X}_{\tmmathbf{C}}$ in $\mathfrak{X}^\infty(\cR^3,\{xyz=0\})$ defined by
\[ \tmmathbf{X}_{\tmmathbf{C}} = \sum_{\us = (s_1, s_2, s_3) \in \{ - 1, 1
   \}^3} \tmmathbf{1 }_{\{ s_1 x > 0, s_2 y > 0, s_3 z > 0 \}} C_{\us} \]
Consider the regularization associated to the mollifier 
$$m(x,y,z) = \mathbf{m}(x)\mathbf{m}(y)\mathbf{m}(z)$$
where $\mathbf{m}$ is the same function given in the previous subsection. 
Then, on the $\eps$-chart of the first blowing-up, the resulting vector field has the form $\mathcal{X}_{\tmmathbf{C}} =
\frac{1}{\varepsilon} \tmmathbf{Y}_{\tmmathbf{C}}$, where the restriction of
$\tmmathbf{Y}_C$ to the {\tmem{core region}} $Q = [- 1/2, 1/2]^3$ is given by

\[ \tmmathbf{Y}_C = \sum_{\us \in \{ - 1, 1 \}^3} P_{\us} C_{\us} \]
with $\left\{ P_{\us} \right\}_{\us}$ being the collection of cubic
polynomials
\[ P_{\us} (x, y, z) = \frac{1}{8}(1 + s_1 x) (1 + s_2 y) (1 + s_3 z) \]
We claim that there exists a collection $\tmmathbf{C}$ of constant vector field
such that $\tmmathbf{Y}_{\tmmathbf{C}}$ has a spatial cusp at the origin $0
\in \cR^3$. Moreover, such spatial cusp unfolds versally as
$\tmmathbf{C}$ is slightly perturbed.

We recall (see \cite{BIRR}, section
5.6) that a singularity of a 3-dimensional vector field $\tmmathbf{X}$ is
called a {\tmem{spatial cusp}} if there exists local coordinates $(X, Y, Z)$
such that it can be written as
\[ Y \frac{\partial}{\partial X} + Z \frac{\partial}{\partial Y} + f (X, Y,
   Z) \frac{\partial}{\partial Z} \]
where $f$ is a function such that $j^1 f (0) = 0$ and
\[ \frac{\partial^2}{\partial X ^2} f (0) \neq 0 \]
The explicit example is given as follows. We consider the collection
$\tmmathbf{C} = \left\{ C_{\us} \right\}$ given by
\[ C_{(s_1, s_2, s_3)} = (s_1 - s_3) \frac{\partial}{\partial x} + (s_1 - s_3
   - s_1 s_2) \frac{\partial}{\partial y} + (s_2 - s_3)
   \frac{\partial}{\partial z} \]
or, equivalently, in vector form,
\[ C_{(s_1, s_2, s_3)} = \left(\begin{array}{c}
     s_1 - s_3\\
     s_1 - s_3 - s_1 s_2\\
     s_2 - s_3
   \end{array}\right) \]
We compute that the associated vector field $\tmmathbf{Y}_C$ has the
expression
\[ (x - z) \frac{\partial}{\partial x} + (x - z - x y)
   \frac{\partial}{\partial y} + (y - z) \frac{\partial}{\partial z} \]
Therefore, under the coordinate change $(X, Y, Z) = (x, x - z, x - y)$, we
obtain
\[ \tmmathbf{Y}_C = Y \frac{\partial}{\partial X} + Z \frac{\partial}{\partial Y} + X (X
   - Z) \frac{\partial}{\partial Z} \]
which shows that the origin is a spatial cusp.

Furthermore, if we consider the three-parameter unfolding $\tmmathbf{C}^{a,
b, c} = \left\{ C_{\us}^{a, b, c} \right\}_{\us}$, with $(a, b, c) \in
\cR^3$, and $C_{\us}^{a, b, c}$ given by
\[ C^{a, b, c}_{(s_1, s_2, s_3)} = \left(\begin{array}{c}
     s_1 - s_3\\
     s_1 - s_3 - s_1 s_2\\
     s_2 - s_3
   \end{array}\right) - a \left(\begin{array}{c}
     0\\
     1\\
     0
   \end{array}\right) - b \left(\begin{array}{c}
     0\\
     s_1 - s_3\\
     0
   \end{array}\right) - c \left(\begin{array}{c}
     0\\
     s_1 - s_2\\
     0
   \end{array}\right) \]
Then the same computations as above give the three-parameter family
\[ Y \frac{\partial}{\partial X} + Z \frac{\partial}{\partial Y} + [a + b Y +
   c Z + X (X - Z)] \frac{\partial}{\partial Z} + O (3) \]
which shows that the spatial cusp unfolds versally inside the family of
regularizations of the piecewise constant cross. \

We note that the above family satisfies the hypothesis of Theorem 5.11 in \cite{BIRR}. This result implies that this family contains a Bykov cycle,
which is a heteroclinic cycle between two saddle-foci of different indices as in the figure
below (see e.g.~\cite{BIRR}, Theorem 5.11).

\begin{figure}[htb]
\centering
 \includegraphics[width=7.44546930342385cm,height=4.25889577594123cm]{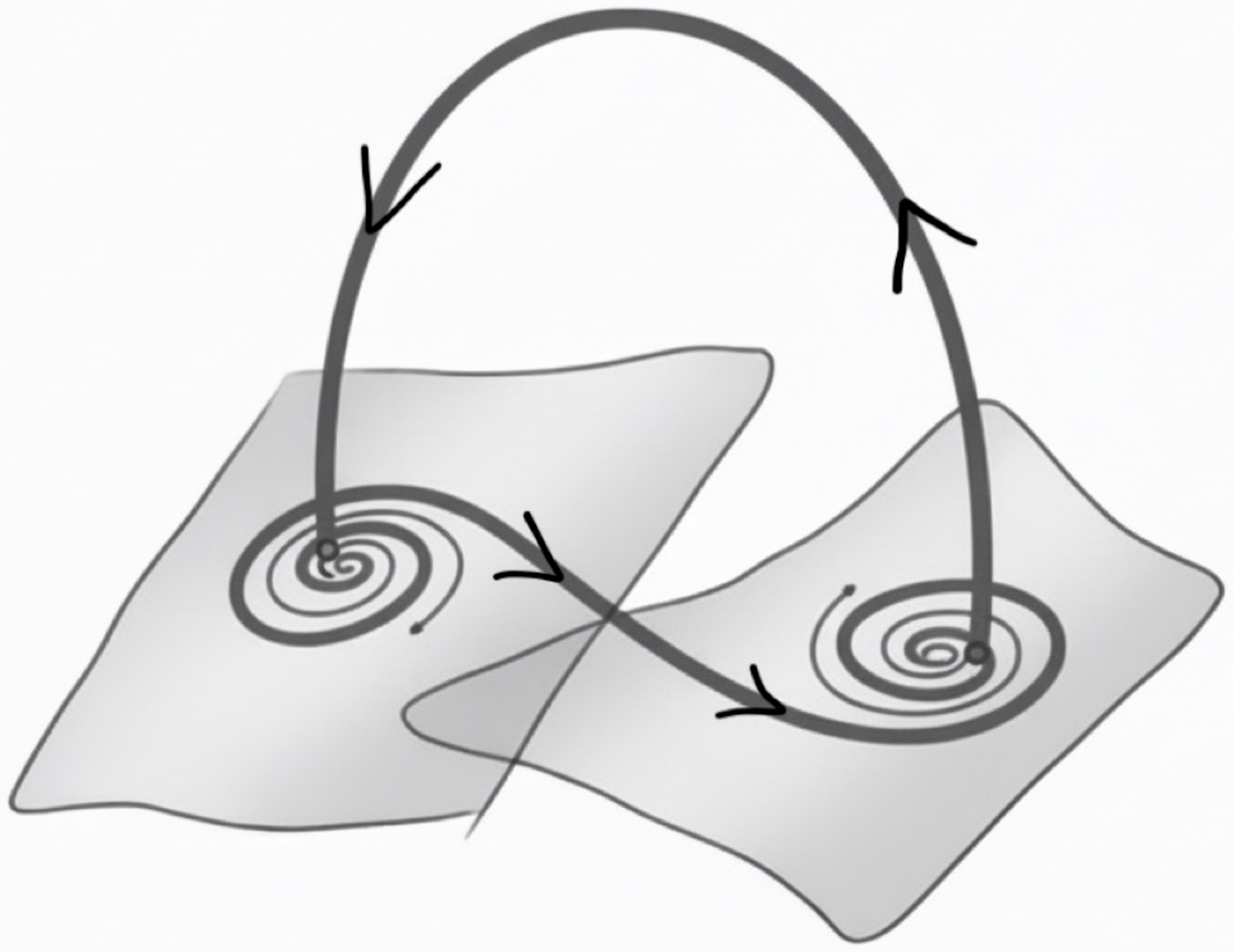}
\end{figure}
This illustrates the fact that chaotic phenomena can appear in the
regularization of the piecewise constant spatial cross.

We remark that the Bykov cycles are
obtained by rescaling, so they shrink to the singularity as $(a,b,c)\rightarrow0$ and lie in the core region.
Every regularized flow with $\eps > 0$ has such
cycles in a ball of radius $O(\eps)$ around the origin. Note also that the horseshoes near Bykov cycles are robust under small perturbations of the data.
Therefore, the above computations show the following result:
\begin{Theorem} 
There is an open set of piecewise constant vector fields on $(\mathbb R^3,\{xyz=0\})$ whose regularizations have, for every $\varepsilon>0$, a Bykov cycle and a horseshoe in a ball of radius $O(\varepsilon)$.
\end{Theorem}
Note that Filippov’s convention cannot capture this phenomena since the limit, as $\eps$ goes to zero, of the
regularized dynamics in the sphere is not captured by any sliding vector field.
\section*{Acknowledgement}
All authors are partially supported by ANR-23-CE40-0028, 
FAPESP grant 2023/02959-5. 
The  first and the third authors are partially supported by FAPESP grant 2024/15612-6.
The third author is partially supported by CNPq grant 302154/2022-1. For the purpose of Open Access, the
authors have applied a CC-BY public copyright licence to any Author Accepted Manuscript (AAM) version
arising from this submission.


\begin{thebibliography}{99}

\bibitem{ando2024} {And\`o, A., Edwards, R. and Guglielmi, N.} (2024).
Nonuniqueness phenomena in discontinuous dynamical systems and their regularizations,
\textit{SIAM J. Appl. Dyn. Syst.} \textbf{23}(2), 1345--1371.
% doi:10.1137/23M1587488

\bibitem {ALGM} {Andronov, A., Leontovich, E., Gordon, I. and Maier, A.} (1971).
\emph{Theory of Bifurcations of Dynamical Systems on a Plane},
IPST, Jerusalem.


\bibitem{BIRR}{Barrientos, P.G., Ib\'a\~nez, S., Rodrigues, A.A. and Rodr\'iguez, J.A.} (2019)
\emph{Emergence of Chaotic Dynamics from Singularities} $32^{\text{o}}$ Colóquio Brasileiro de Matemática, Editora do IMPA.

\bibitem{BobienskiMardesic2008}
{Bobi{\'e}nski, M. and Marde{\v{s}}i{\'c}, P.},
\emph{Pseudo-Abelian integrals along Darboux cycles},
Proc.\ Lond.\ Math.\ Soc.\ (3) \textbf{97} (2008), no.~3, 669--688.


\bibitem {BPT} {Buzzi, C., da Silva, P.R., and Teixeira, M.A.} (2006).
A Singular approach to discontinuous vector fields on the plane, \textit{J.Differential Equations} \textbf{231},
 633--655.
 
\bibitem {BRT} {Bonet-Rev\'{e}s, C. and M-Seara, T.} (2016).
 Regularization of sliding global bifurcations derived from the local fold singularity of Filippov systems,
\textit{Discrete Contin. Dynam. Systems} \textbf{36-7}, 3545--3601.
 
 \bibitem {BLT} {Bonet-Rev\'{e}s, C., Larrosa, J. and M-Seara, T.} (2018).
 Regularization around a generic codimension one fold-fold singularity,
 \textit{J.Differential Equations} \textbf{265}, 1761--1838.
 

\bibitem{DeMaesschalckHuzakPerez2025}
{De Maesschalck, P., Huzak, R. and Perez, O.H.},
\newblock “Canard cycles of non-linearly regularized piecewise smooth vector fields,”
\newblock preprint (2025), arXiv:2506.18099v1.

\bibitem{dieci2013} {Dieci, L. and Guglielmi, N.} (2013).
Regularizing piecewise smooth differential systems: co-dimension 2 discontinuity surface,
\textit{J. Dynam. Differential Equations} \textbf{25}, 71--94.
% doi:10.1007/s10884-013-9287-4

 \bibitem {DR} {Dumortier, F. and Roussarie, R.} (1996).
 {\it Canard cycles and center manifolds}, Memoirs Amer. Mat. Soc.
 \textbf{121}.
 

\bibitem {AF} {Filippov, A.F.} (1988).
\emph{Differential equations with discontinuous right--hand sides},
Mathematics and its Applications (Soviet Series), Kluwer Academic
Publishers, Dordrecht.

\bibitem{GreeneWu1979}
{Greene, R.E. and Wu, H.},
\newblock {\em $C^\infty$ approximations of convex, subharmonic, and
plurisubharmonic functions},
\newblock Ann. Sci. \'Ec. Norm. Sup\'er. (4) \textbf{12} (1979), no.~1, 47--84.

\bibitem{guglielmi2015} {Guglielmi, N. and Hairer, E.} (2015).
Classification of hidden dynamics in discontinuous dynamical systems,
\textit{SIAM J. Appl. Dyn. Syst.} \textbf{14}(3), 1454--1477.
% doi:10.1137/15100326X



\bibitem{HormanderI}
{H{\"o}rmander, L.},
\emph{The Analysis of Linear Partial Differential Operators I},
Springer-Verlag, Berlin, 2nd Edition (1990).

\bibitem{HormanderNHDE}
{H{\"o}rmander, L.},
\emph{Lectures on Nonlinear Hyperbolic Differential Equations},
Springer, Berlin, 1997.

\bibitem{kaklamanos2019}
{Kaklamanos, P. and Kristiansen, K.U.}
\newblock Regularization and Geometry of Piecewise Smooth Systems with Intersecting Discontinuity Sets.
\newblock \emph{SIAM Journal on Applied Dynamical Systems}, \textbf{18}(3), 1225--1264 (2019).
\newblock \url{https://doi.org/10.1137/18M1214470}

\bibitem{Kr2}{Kristiansen, K.U. and Hogan, S.J.} (2015).
Regularizations of two-fold bifurcations in planar piecewise-smooth systems using blow up,
\textit{SIAM J. Appl. Dyn. Syst.} \textbf{14-4}, 1731--1786.


\bibitem {LST} {Llibre, J., Silva, P.R. and Teixeira, M.A.} (2007).
Regularization of discontinuous vector fields via singular
perturbation, \textit{J. Dynam. Differential Equation} \textbf{19-2}, 309--331.

\bibitem {LST3} {Llibre, J., Silva, P.R. and Teixeira, M.A.} (2009).
Study of Singularities in non smooth dynamical systems via Singular Perturbation, 
\emph{SIAM Journal on Applied Dynamical Systems} \textbf{8},  508--526.

\bibitem {LST4} {Llibre, J., Silva, P.R. and Teixeira, M.A.} (2015). Sliding vector fields for non-smooth 
dynamical systems having intersecting switching manifolds, \emph{Nonlinearity} \textbf{28}, 493--507.

\bibitem {LST5} {Llibre, J., Silva, P.R. and Teixeira, M.A.} (2008).
Sliding vector fields via slow fast systems, \textit{Bulletin of the
Belgian Mathematical Society Simon Stevin} \textbf{15}, 851--869.

\bibitem{KottkeMelrose2015} {Kottke, C. and Melrose, R.B.} (2015). Generalized blow-up of corners and fiber products, \textit{Trans. Amer. Math. Soc.} \textbf{367}(1), 651--705.

\bibitem {Melrose} {Melrose, R.B.} Differential Analysis on Manifolds with Corners - https://math.mit.edu/~rbm/book.html \href{https://math.mit.edu/~rbm/book.html}{(online)}.

\bibitem{Melrose1993} {Melrose, R.\,B.} (1993). \textit{The Atiyah--Patodi--Singer Index Theorem}, Research Notes in Mathematics \textbf{4}, A K Peters, Wellesley, MA.

\bibitem{novaes2015} {Novaes, D.D. and Jeffrey, M.R.} (2015).
Regularization of hidden dynamics in piecewise smooth flows,
\textit{J. Differential Equations} \textbf{259}(9), 4615--4633.
% doi:10.1016/j.jde.2015.06.005

\bibitem {PS} {Panazzolo, D.C. and Silva, P.R.} (2017).
Regularization of discontinuous foliations: Blowing up
and sliding conditions via Fenichel theory, \textit{J.Differential Equations} \textbf{263},
8362–-8390.

\bibitem {ST} {Sotomayor, J. and Teixeira, M.A.} (1996).
Regularization of  discontinuous vector fields,
\textit{International Conference on Differential Equations}, Lisboa,
Equadiff 95, 207--223.

\bibitem {SM} {Sotomayor, J. and Machado, A.L.} (2002).
Structurally stable discontinuous vector fields on the plane, \textit{Qual. Theory of Dynamical Systems} \textbf{3}, 227--250.

\end{thebibliography}
\end{document}